\documentclass{amsbook}

\usepackage{amsmath, amsfonts, amssymb, mathrsfs, amsthm, hyphenat, url}%

\usepackage{mathtools}
\usepackage{color}
\usepackage{dsfont}
\usepackage{verbatim}
\usepackage[all]{xy}

\usepackage{etoolbox}

\usepackage{listings}

\usepackage[dvipsnames]{xcolor}

\newtheorem{theorem}{Theorem}[section]
\newtheorem{lemma}[theorem]{Lemma}
\newtheorem{proposition}[theorem]{Proposition}
\newtheorem{corollary}[theorem]{Corollary}
\newtheorem{fact}[theorem]{Fact}

\newtheorem{lem}{Lemma}[chapter]

\theoremstyle{definition}
\newtheorem{definition}[theorem]{Definition}
\newtheorem{example}[theorem]{Example}

\theoremstyle{remark}
\newtheorem{remark}[theorem]{Remark}

\newtheorem{rmk}[lem]{Remark}

\numberwithin{section}{chapter}
\numberwithin{equation}{chapter}

\def\mc#1{\mathcal{#1}}
\def\mb#1{\mathds{#1}}  
\def\tx#1{{\rm #1}}     
\newcommand{\R}{\mathds{R}}
\newcommand{\C}{\mathds{C}}
\newcommand{\Q}{\mathds{Q}}
\newcommand{\A}{\mathds{A}}
\newcommand{\Z}{\mathds{Z}}
\newcommand{\lmod}{\setminus}
\newcommand{\Gal}{\Gamma} 

\newcommand{\ul}[1]{\underline{#1}}

\def\from{\leftarrow}

\def\into{\hookrightarrow}

\def\onto{\twoheadrightarrow}

\newcommand{\ov}{\overline}

\newcommand{\lra}{\longrightarrow}
\newcommand{\lla}{\longleftarrow}

\newcommand{\ab}{\tx{ab}}
\newcommand{\ssc}{\tx{sc}}
\newcommand{\ad}{\tx{ad}}
\newcommand{\loc}{\tx{loc}}

\newcommand{\vbar}{{\overline{v}}}

\newcommand{\cH}{{\mc{H}}}
\newcommand{\cF}{{\mc{F}}}

\newcommand{\V}{V}
\newcommand{\Vbar}{{\,\ov \V\,}}
\newcommand{\G}{{\Gamma}}
\newcommand{\GEF}{\G}
\newcommand{\KK}{{\sf K}}
\newcommand{\GG}{{\mathds G}}
\newcommand{\Lam}{\Lambda}
\newcommand{\vk}{{\varkappa}}
\newcommand{\Fbar}{{\ov F}}

\DeclareMathOperator{\coker}{coker}
\DeclareMathOperator{\im}{im}

\newcommand{\tf}{{\rm tf}}
\DeclareMathOperator{\Tor}{Tor}
\DeclareMathOperator{\Hom}{Hom}
\DeclareMathOperator{\Ind}{Ind}
\newcommand{\SL}{{\rm SL}}

\DeclareMathOperator{\Ext}{Ext}

\newcommand{\fstar}{{ \framebox{\makebox[\totalheight]{$*$}} }}

\newcommand{\Tors}{\tx{Tors}}
\newcommand{\tor}{{\tx{Tors}}}

\newcommand{\sub}[1]{{(#1)}}
\newcommand{\sss}[1]{{(#1)}}

\newcommand{\bv}{{\breve v}}
\newcommand{\Gv}{{\G\sss{\hm\bv}}}
\newcommand{\Gt}{{\G\hm,\Tors}}
\newcommand{\Gvt}{{\Gv,\Tors}}
\newcommand{\Gtf}{{\G\hm,\hs\tf}}

\newcommand{\vp}{{\varphi'}}

\newcommand{\Mtf}{M_\tf}
\newcommand{\Ab}{{\tx{Ab}}}

\newcommand{\mtil}{{\tilde m}}
\newcommand{\cmp}{\complement}

\newcommand{\Tbul}{{T^\bullet}}

\DeclareMathOperator{\id}{id}
\DeclareMathOperator{\Aut}{Aut}
\newcommand{\lcm}{{\rm lcm}}
\DeclareMathOperator{\Cor}{Cor}

\newcommand{\FF}{{\mathds F}}
\newcommand{\pp}{{\mathfrak p}}

\newcommand{\FP}{{\mathcal{FP}}}
\newcommand{\MM}{{\bf M}}

\newcommand{\li}{{\,\labelt{\,\iota_*}\,}}
\newcommand{\lj}{{\,\labelt{\,\vk_*}\,}}
\newcommand{\rsa}{\rightsquigarrow}

\newcommand{\GF}{{\Gal(F^s\hm/F)}}

\newcommand{\phitil}{{\widetilde\phi}}
\newcommand{\res}{{\rm Res}}

\newcommand{\eff}{{\rm eff}}
\newcommand{\D}{\Delta}

\newcommand{\PGL}{{\rm PGL}}

\newcommand{\vphi}{\varphi}
\newcommand{\ups}{\upsilon}

\newcommand{\hs}{\kern 0.8pt}
\newcommand{\hssh}{\kern 1.2pt}

\newcommand{\hm}{\kern -0.8pt}
\newcommand{\hmm}{\kern -1.2pt}

\newcommand{\labelt}[1]{\xrightarrow{\makebox[1.em]{\scriptsize ${#1}$}}}
\newcommand{\labelto}[1]{\xrightarrow{\makebox[1.5em]{\scriptsize ${#1}$}}}
\newcommand{\labeltoo}[1]{\xrightarrow{\makebox[2em]{\scriptsize ${#1}$}}}
\newcommand{\labeltooo}[1]{\xrightarrow{\makebox[2.7em]{\scriptsize ${#1}$}}}

\newcommand{\isoto}{\xrightarrow{\sim}}
\newcommand{\longisot}{{\labelt\sim}}
\newcommand{\longisoto}{{\labelto\sim}}

\newcommand{\emm}{\bfseries}

\newcommand{\abcal}{{\mathcal{A}b}}

\newcommand{\Res}{{\rm Res}}
\newcommand{\Dt}{{\D,\Tors}}
\newcommand{\proj}{{\rm proj}}

\newcommand{\BarC}{{\mathrm{Bar}}}
\newcommand{\g}{{\gamma}}
\newcommand{\upgam}{{\gamma\hs}}
\newcommand{\Maps}{{\mathrm{Maps}}}

\DeclareFontEncoding{OT2}{}{} 
\DeclareTextFontCommand{\textcyr}{\fontencoding{OT2}
    \fontfamily{wncyr}\fontseries{m}\fontshape{n}\selectfont}
\newcommand{\Sha}{{\!\textcyr{Sh}}}

\begin{document}

\frontmatter

\title[Galois cohomology of reductive groups]
{Galois cohomology of reductive groups\\ over global fields}

\author{Mikhail Borovoi}

\address{Raymond and Beverly Sackler School of Mathematical Sciences,
Tel Aviv University, 6997801 Tel Aviv, Israel}
\email{borovoi@tauex.tau.ac.il}

\thanks{Borovoi was partially supported
by the Israel Science Foundation (grant 1030/22).}

\author{Tasho Kaletha\\{ \small with Appendix A  by} \\  Vladimir Hinich}

\address{Department of Mathematics, University of Michigan,
530 Church Street, Ann Arbor, MI 48109, USA}
\curraddr{Mathematics Center, Universit\"at Bonn, Endenicher Allee 60, 53115 Bonn, Germany}
\email{kaletha/math/uni-bonn/de}

\thanks{Kaletha was partially supported by NSF grants DMS-1801687 and DMS-2301507 and by a Simons Fellowship.}


\address{Department of Mathematics, University of Haifa, Mount Carmel,
Haifa 3498838, Israel}
\email{hinich@math.haifa.ac.il}

\date{7 May 2025}




\begin{abstract}
We give closed formulas for the abelian Galois cohomology groups
$H^1_{\rm ab}(F,G)$ and $H^2_{\rm ab}(F,G)$
of a connected reductive group $G$ over a global field $F$ in terms of the algebraic fundamental group $\pi_1(G)$ introduced earlier by one of us (M.B.).
We further give closed formulas for the effects of restriction, corestriction, and localization,
 in terms of these formulas and the analogous known formulas in the case of local fields.
 Building on this, we give formulas, suitable for computer computations,
 for the first nonabelian Galois cohomology set $H^1(F,G)$ of $G$
 and for the second Galois cohomology group $H^2(F,T)$ of an $F$-torus $T$.

As a preparation for the derivation of our formulas,
we review the interpretation of Tate cohomology of a finite group
in terms of the stable derived category of ${\mathbb Z}[\Gamma]$-modules due to Buchweitz,
and relate it to the explicit definition via cochains due to Kottwitz--Shelstad.
We use this to construct the Tate--Nakayama isomorphisms for bounded complexes of tori over local and global fields,
whose specialization to complexes of length $2$ is then applied to obtain the desired formulas.
\end{abstract}

\maketitle

\tableofcontents

\chapter*{Acknowledgements}
The authors thank Dylon Chow, Cristian D.~Gonz\'alez-Avil\'es,
Willem A.~de Graaf,  Robert E.~Kottwitz, Boris \`E.~Kunyavski\u{\i}, and Zev Rosengarten
for helpful comments and email correspondence.
We thank Kasper Andersen,  Keith Conrad, and Will Sawin for answering Borovoi's questions in MathOverflow.
Special thanks to Vladimir Hinich for writing Appendix \ref{app:Hinich}.
We thank the anonymous referees for thorough refereeing and helpful comments and suggestions.

\mainmatter

\chapter{Introduction}

\addtocontents{toc}{\protect\setcounter{tocdepth}{0}}

\section{}
Let $F$ be a local or global field, $\ov F$ a fixed algebraic closure of $F$,
\,$F^s$ the separable closure of $F$ in $\ov F$, and $\G\sub F=\tx{Gal}(F^s\hm/F)$ the absolute Galois group of $F$.
Let $G$ be a (connected) reductive $F$-group (we follow the convention of SGA3, where reductive groups are assumed to be connected).
We refer to Serre's book \cite{SerGC} for the definition of the first Galois cohomology set
\[H^1(F,G)\coloneqq H^1\big(\tx{Gal}(F^s\hm/F),\,G(F^s)\big). \]
In this paper we give a combinatorial description of $H^1(F,G)$ when $F$ is a  global field,
in terms of the algebraic fundamental group $\pi_1(G)$ and the sets $H^1(F_v, G)$
for the real places $v$ of $F$ (when $F$ has no real places, the description only involves $\pi_1(G)\hs$).
The algebraic fundamental group $\pi_1(G)$, which we recall in \S\ref{ss:fund-gp},
was introduced earlier by one of us in \cite{Brv98},
and also by Merkurjev \cite{Merkurjev98} and Colliot-Th\'el\`ene \cite{CT08}.
It is a finitely generated abelian group equipped with an action of $\GF$.
We also rederive the known description of $H^1(F,G)$ in terms of $\pi_1(G)$
when $F$ is a non-archimedean local field, a result that goes back to work of Kottwitz.

\section{}
In the case when $F$ is a $p$-adic field
(a finite extension of the field of $p$-adic numbers $\Q_p$),
Kottwitz \cite{Kot84} noticed that
the set $H^1(F,G)$ has a canonical structure of an abelian group,
and he computed this abelian group in \cite[Proposition 6.4]{Kot84}
in terms of the center of a Langlands dual group for $G$.
One can rephrase Kottwitz's formula in terms of $\pi_1(G)$,
see \cite[Corollary 5.5(i)]{Brv98} or Labesse \cite[\S1.6]{Labesse99}.
The result, which is essentially \cite[Theorem 1.2]{Kot86},
is a functorial bijection $H^1(F,G) \to \pi_1(G)_{\G\sub F,{\Tors}}$,
where $\pi_1(G)_{\G\sub F,{\Tors}}$ denotes the torsion subgroup
of the group of $\G\sub F$-coinvariants of $\pi_1(G)$.
This bijection induces a functorial abelian group structure on $H^1(F,G)$, namely,
the unique one that turns this bijection into an isomorphism of groups.

\section{}
In the case  $F=\R$, the set $H^1(F,G)$ was computed in \cite{Borovoi88}
(see also \cite{Borovoi22-CiM})
using the method of Borel and Serre \cite[Theorem 6.8]{Borel-Serre},
and  in \cite[Theorem 7.14]{BorTim24}
using the method of Kac \cite{Kac69}.
There is a computer program computing $H^1(\R,G)$
(even when $G$ is not necessarily reductive or connected); see \cite{BG23}.
Note that over $F=\R$,  the set $H^1(F,G)$ for a reductive $F$-group $G$
cannot be computed in terms of $\pi_1(G)$ only.
While Kottwitz's map $H^1(F,G) \to \pi_1(G)_{\G\sub F,{\Tors}}$ from \cite[Theorem 1.2]{Kot86} still exists, in general
it is neither injective nor surjective. Instead, there is an exact sequence of pointed sets
\[ H^1(F,G_\tx{sc}) \to H^1(F,G) \to \pi_1(G)_{\G\sub F,{\Tors}} \to \pi_1(G)_{\Tors}, \]
where the last map is induced by the norm homomorphism for the action of $\G\sub{F}$ on $\pi_1(G)$.
Here $G_\ssc$ is the universal cover of the commutator subgroup $[G,G]$ of $G$.
Contrary to the case of non-archimedean $F$, the set $H^1(F,G_\tx{sc})$ need not be a singleton for $F=\R$,
and there is no functorial group structure on $H^1(F,G)$;
see \cite[Theorem 1.2]{Borovoi24-Group-structure}.

\section{}
Let $F$ be a global field, that is, a finite extension either of the field $\Q$
of rational numbers, or of the field $\mb{F}_q(t)$ of rational functions over a finite field.
Write $\V_F$ for the set of places of $F$.

For each place $v\in \V_F$ there is the canonical localization map
$\tx{loc}_v\colon H^1(F,G)\to H^1(F_v,G)$ where $F_v$ denotes the completion of $F$ at $v$.
Consider the product of the localization maps
\[ H^1(F,G)\to \prod_{v\in V_F} H^1(F_v,G).\]
It is well known that this map actually takes image in
\[\bigoplus_{v\in \V_F}\!\! H^1(F_v,G)\,\coloneqq\,
    \bigg\{ \big(\eta_v\big)_{v\in \V_F}\in\!\! \prod_{v\in \V_F}\!\! H^1(F_v,G)\ \,
    \bigg|\ \,\eta_v=1\ \text{for almost all}\ v\bigg\};\]
see, for instance, \cite[Corollary 4.6]{Brv98}.
Moreover, this image is known;
see Kottwitz \cite[Proposition 2.6]{Kot86} or \cite[Theorem 5.15]{Brv98}.

Consider the Tate--Shafarevich kernel
\[\Sha^1(F,G)\coloneqq \ker\bigg[ H^1(F,G)\to \bigoplus_{v\in \V_F} H^1(F_v,G)\bigg].\]
Sansuc \cite[Theorem 8.5]{Sansuc81} showed that the finite set $\Sha^1(F,G)$
has a functorial structure of an abelian group.
Kottwitz \cite[(4.2.2)]{Kot84} computed the group $\Sha^1(F,G)$
in terms of the center of a Langlands dual group for $G$.
Another, more functorial version of the formula of Kottwitz,
again in terms of $\pi_1(G)$, was given in
\cite[Proposition 4.13 and Corollary 5.13]{Brv98}.
Using twisting, one can show that $\Sha^1(F,G)$
naturally acts on the non-empty fibers of the localization  map
\[\loc\colon H^1(F,G)\to\bigoplus_{v\in\V_F} H^1(F_v,G),\]
and this action is simply transitive.
Thus one knows the image and the fibers of this localization map.
In this paper we compute functorially  the set $H^1(F,G)$.
Moreover, for an $F$-torus $T$, we compute functorially $H^2(F,T)$.

\section{}
\label{ss:E/F}
Let $G$ be a reductive group over a global field $F$.
We write $M=\pi_1(G)$; then the absolute Galois group $\GF$ acts on $M$ via a finite quotient group.
We state our main results in terms of a sufficiently large {\em finite} Galois extension $E/F$.

Let $E$ be a finite Galois extension of $F$ in $F^s$ having no real places
and such that $\G(F^s/E)$ acts trivially on $M$.
Write $\GEF=\Gal(E/F)$, the Galois group of $E/F$, which is a finite group naturally acting on $M$.
Let $\V_E$ denote the set of places of $E$;
the Galois group $\G$ naturally acts on $\V_E$.
We compute $H^1(F,G)$ in terms of the action of $\G$ on the group $M$ and on the set $\V_E$,
using also the Galois cohomology sets $H^1(F_v,G)$ for the real places $v$ of $F$.

Consider the group of finite formal linear combinations
\[ M[\V_E]\coloneqq \Big\{\sum_{w\in \V_E}\! m_w\cdot w\ \Big|\ m_w\in M\Big\}. \]
The group $\GEF$ naturally  acts on $M[\V_E]$. We have a $\GEF$-equivariant homomorphism
\[ M[\V_E]\to M,\qquad \sum m_w\cdot w\,\mapsto \sum m_w \]
and we let $M[\V_E]_0$ denote its kernel.
We write $(M[\V_E]_0)_{\GEF}$ for the {\em group of coinvariants} of $\GEF$ in $M[\V_E]_0$\hs,
and we write
\[ (M[\V_E]_0)_\Gt\coloneqq \big(\hs(M[\V_E]_0)_\GEF\big)_\tor \]
for the torsion subgroup of $(M[\V_E]_0)_{\G}$\hs.
Similarly, we write
\[ \big(M[\V_E]_0\otimes\Q/\Z\big)_\G \]
for the group of coinvariants of $\G$ in $M[\V_E]_0\otimes\Q/\Z$
where the tensor product is taken over $\Z$.
Note that there is a canonical isomorphism
$\big(M[\V_E]_0\big)_\G\otimes\Q/\Z=\big(M[\V_E]_0\otimes\Q/\Z\big)_\G$\hs.
For each place $v$ of $F$ we choose a place $\bv$ of $E$ over $v$,
and we denote by  $\Gv$ the stabilizer in $\G$ of $\bv$;
it can be identified with the Galois group $\tx{Gal}(E_\bv/F_v)$.
If $v$ is an archimedean place, then $E_\bv\simeq \C$,
because by assumption $E$ has no real places.
If $v$ is a real place of $F$, then $\Gv$ is of order 2,
and if $v$ is a complex place, then $\Gv=\{1\}$.

We recall that a commutative  diagram
\[
\xymatrix@C=5mm@R=5mm{
A\ar[r]\ar[d] & B\ar[d]\\
C\ar[r] & D
}
\]
is called Cartesian if it identifies $A$ with the fibered product of $B$ and $C$ over $D$.
We prove the following theorem:

\begin{theorem}\label{t1:main}\ \\[-10pt]
\begin{enumerate}

\item[\rm(1)] With the notation above, for a reductive group $G$ over a global field $F$
with fundamental group  $M=\pi_1(G)$,
the pointed set $H^1(F,G)$ fits into a functorial Cartesian diagram
\[ \xymatrix@C=16mm{
	H^1(F,G)\ar[r]\ar[d]&\prod\limits_{v|\infty} H^1(F_v,G)\hskip -18mm \ar[d]\\
	(M[V_E]_0)_\Gt\ar[r]&\prod\limits_{v|\infty} M_\Gvt\hskip -16.5mm
}\]
where the products on the right run over the set of infinite places $v$ of $F$.

\item[\rm(2)] For a torus $T$ over a global field $F$, the abelian group $H^2(F,T)$
fits into  a functorial Cartesian diagram
\[ \xymatrix@C=16mm{
	H^2(F,T)\ar[r]\ar[d]&\prod\limits_{v|\infty}  H^0(\Gv, M)\hskip-20.5mm\ar[d]\\
	\hskip-2.5mm(M[V_E]_0 \otimes \Q/\Z)_{\G}\ar[r]&\prod\limits_{v|\infty} (M \otimes \Q/\Z)_\Gv\hskip-24mm
}\]
where $M=\pi_1(T)\coloneqq X_*(T)$ is the cocharacter group of $T$,
and  $ H^0(\Gv, M)$ denotes the {\em Tate cohomology} of the finite group $\Gv$ of order dividing 2.

\item[\rm(3)] If $F$ has no real places, that is, $F$ is a global function field or a totally imaginary number field,
 then we have a functorial bijection
\[H^1(F,G)\isoto (M[V_f(E)]_0)_\Gt\quad\text{with}\ \, M=\pi_1(G)\]
and a functorial isomorphism
\[H^2(F,T)\isoto (M[V_f(E)]_0 \otimes \Q/\Z)_\G\quad\text{with}\ \, M=\pi_1(T)=X_*(T)\]
where $V_f(E)$ denotes the set of finite (that is, non-archimedean) places of $E$.
\end{enumerate}
\end{theorem}

Assertion (1) is Theorem \ref{thm:main}. It follows from Theorem \ref{t1:H1-H2-ab}  and  Theorem \ref{t:Brv-GA}.
Assertion (2) follows immediately from Theorem \ref{t1:H1-H2-ab}.
Assertion (3) follows from Corollary \ref{c:tot-imaginary}.
For description of the maps we refer the reader to the body of the paper,
where we show that the diagrams above do not depend on the choice of the extension $E/F$
and that they are compatible with restriction (corresponding to a finite separable extension $F'/F$), localization, and connecting homomorphisms.

\begin{remark} When $G=T$ is a torus, assertion (1) of  Theorem \ref{t1:main}
follows from the result of Tate \cite{Tate66}.
Indeed, since $T$ splits over $E$, we may identify $H^1(F,T)=H^1(E/F, T)$.
By Tate  \cite{Tate66} we have a canonical isomorphism
$$H^1(E/F,T)\cong H^{-1}(\G,M[V_E]_0)$$
where $M=X_*(T)=\pi_1(T)$.
Finally, since $M[V_E]_0$ is torsion-free, we have
$$H^{-1}(\G,M[V_E]_0)=(M[V_E]_0)_\Gt\hs,$$
and we obtain a canonical isomorphism
\[ H^1(F,T)=H^1(E/F, T)\isoto H^{-1}(\G,M[V_E]_0)= (M[V_E]_0)_\Gt\hs. \]
\end{remark}

\begin{remark}
For a general connected reductive group $G$ over a global field $F$, assertion (1) of Theorem \ref{t1:main}
can be also deduced from a result of Kottwitz's preprint  \cite{KotBG}.
Using Galois gerbes,
Kottwitz defines pointed sets $B(F,G)_\tx{bsc}$ and $B(F_v,G)_\tx{bsc}$, equipped with maps to $\big(M[\V_E]_0\big)_\G$ and $M_\Gv$, respectively.
He identifies $B(F,G)_\tx{bsc}$ with the fiber product
of $\big(M[\V_E]_0\big)_\G$ and $\prod_{v|\infty} B(F_v,G)_\tx{bsc}$
over $\prod_{v|\infty} M_\Gv$  (\hs\cite[Proposition 15.1]{KotBG}\hs).
The Galois cohomology set $H^1(F,G)$ embeds into $B(F,G)_\tx{bsc}$ and is equal to the preimage of the torsion subgroup
$\big(M[V_E]_0\big)_\Gt\subseteq \big(M[V_E]_0\big)_\G$\hs. This identifies $H^1(F,G)$ with the fiber product
of $\big(M[\V_E]_0\big)_\Gt$ and $\prod_{v|\infty} H^1(F_v,G)$
over $\prod_{v|\infty} M_\Gvt$ (our Theorem \ref{t1:main}(1)\hs).
The proof we give here is different from Kottwitz's and does not involve Galois gerbes.
Rather, it proceeds by computing the abelianized Galois cohomology
$H^i_\tx{ab}(F,G)$ for all $i>0$; see \S\ref{ss:abelian} and \S\ref{ss:comp-abelian}.
Moreover, our construction of the map $H^1(F,G) \to (M[V_E]_0)_{\Gamma,\tx{Tors}}$ is more direct,
as it does not involve taking $z$-extensions.
Instead, we systematically use Galois hypercohomology of bounded complexes.
\end{remark}

\begin{remark}
When $G$ is a simply connected semisimple group, we have $M=\pi_1(G)=0$,
and Theorem \ref{t1:main}(1) gives a bijection
\[ H^1(F,G)\isoto \prod\limits_{v|\infty} H^1(F_v,G),\]
which is the celebrated Hasse principle of Kneser, Harder, and Chernousov.
We do not give a new proof of this result; we use it in our proof of Theorem \ref{t1:main}.
See Remark \ref{r:HPs5}.
\end{remark}

\begin{remark}
By Remark \ref{r:xi},
the right-hand vertical arrow in the diagram of assertion (2) of  Theorem \ref{t1:main}
is injective. Therefore, the diagram shows that $H^2(F,T)$ is isomorphic to a subgroup of
$(M[V_E]_0 \otimes \Q/\Z)_{\G}$.
\end{remark}

\section{}
For a reductive group $G$ over a global field $F$, Theorem \ref{t1:main}(1)
describes the infinite set $H^1(F,G)$ in terms of the finite
real Galois cohomology sets $H^1(F_v,G)$ and the infinite abelian group $(M[\V_E]_0)_\Gt$.
From Theorem  \ref{t1:main}(1) we deduce the following description of  $H^1(F,G)$
in terms of a finite set and finite groups.

Let $S_F\subset V_F$ be a finite subset,
and let $S_E\subset V_E$ denote the preimage of $S_F$ in $V_E$.
Write $S_F^\cmp=V_F\smallsetminus S_F$, $S_E^\cmp=V_E\smallsetminus S_E$.
We can and will choose $S_F$ in such a way
that $S_F$ contains all infinite places of $F$ and that
for any (finite) place $v\in S_F^\cmp$
there exists a {\em finite} place $w_S(\bv)$ such that $\G(w_S(\bv))\supseteq \G(\bv)$,
where $\G(w_S(\bv))$ and $\G(\bv)$ are the corresponding
decomposition groups (stabilizers in $\G$).
Then we can construct a $\G$-equivariant map
\[ \phi\colon S_E^\cmp\to S_E,\quad\ \gamma\cdot\bv\mapsto \gamma\cdot  w_S(\bv)\quad \ \text{for}\ \,\gamma\in\G.\]

We denote by $\cH^1_S(G)$ the fiber product of the maps
\[\big([M[S_E]_0\big)_\Gt\lra \prod_{v|\infty} M_\Gvt\longleftarrow \prod_{v|\infty}H^1(F_v,G)\]
where the arrows are natural maps (see Chapter \ref{sec:explicit}).

\begin{theorem}[Theorem \ref{t:explicit}]
\label{t:non-can}
For $G$ and $S_F$ as above, there is a bijection
\[H^1(F,G)\isoto \cH^1_S(G)\times\bigoplus_{v\in S_F^\cmp}M_\Gvt\hs.\]
\end{theorem}

Note that the bijection in the theorem \ref{t:non-can} is not canonical:
it depends on the choice of $\phi$.

In Theorem \ref{t:non-can}, the finite set $\cH^1_S(G)$
can be computed using a computer,
and the finite groups $M_\Gvt$ for $v\in  S_F^\cmp$
can be easily computed by hand or using a computer.
This makes it possible to use a computer for computing the Galois cohomology
of reductive groups over global fields.
We use our results for computations with Tate--Shafarevich groups
regarded as subgroups of cohomology groups.
We give an  example of a semisimple group $G$ over a number field $F$,
for which with the help of computer we show that the Tate--Shafarevich group
$\Sha^1_\ab(F,G)$  is not a direct summand of $H^1_\ab(F,G)$ (Proposition \ref{p:using-computer}).
Moreover,  we prove the following theorem:
\begin{theorem}[Corollary \ref{c:not-direct}]
\label{t:not-direct-intro}
Let $T$ be a torus over a global field $F$.
If $\Sha^2(F,T)$ is nontrivial,
then it is not a direct summand of $H^2(F,T)$.
\end{theorem}

\section{}
\label{ss:abelian}
The proof of Theorem \ref{t1:main} is based on the computation of the abelian Galois cohomology $H^i_\tx{ab}(F,G)$,
the definition of which is recalled in \S\ref{ss:fund-gp}.
This cohomology is available for all $i \geq 0$ and is an abelian group.
We compute this abelian group in terms of $\pi_1(G)$ for all $i >0$.

\begin{theorem}[Theorem \ref{thm:tnglob}]
\label{t1:H1-H2-ab}
	For a reductive group $G$ over a global field $F$ with algebraic fundamental group $M=\pi_1(G)$,
there are functorial Cartesian diagrams
\[
\xymatrix@C=43pt{
	H^1_\tx{ab}(F,G)\ar[r]\ar[d]&\prod\limits_{v|\infty}  H^{-1}(\Gv,M)\hskip-67pt\ar[d]\\
	(M[\V_E]_0)_\Gt\ar[r]&\prod\limits_{v|\infty} M_\Gvt\hskip-45pt
}\hskip 35pt
\]
\[
\xymatrix@C=35pt{
	H^2_\tx{ab}(F,G)\ar[r]\ar[d]&\prod\limits_{v|\infty}  H^0(\Gv,M)\hskip-62
pt\ar[d]\\
	(M[\V_E]_0 \otimes \Q/\Z)_{\G}\ar[r]&\prod\limits_{v|\infty} (M \otimes \Q/\Z)_\Gv\hskip-70pt
}\hskip 40pt
\]
Moreover, for $i>2$ there are functorial isomorphisms
\[ H^i_\tx{ab}(F,G)\,  \isoto\, \prod_{v|\infty} H^{i-2}(\Gv,M).\]
Furthermore, if $F$ has no real places, then there are functorial isomorphisms
\[ H^1_\tx{ab}(F,G)\isoto \big(M[\V_f(E)]_0\big)_\Gt\quad   \text{and} \quad
   H^2_\tx{ab}(F,G)\isoto \big(M[\V_f(E)]_0 \otimes \Q/\Z\big)_{\G}\,.
\]
\end{theorem}

We also describe the compatibility of these Cartesian diagrams  with respect to restriction,
corestriction, localization, and connecting homomorphisms.

\section{}
\label{ss:comp-abelian}
The computation of the abelian cohomology of $G$ goes as follows.
Consider the composite homomorphism
$\rho\colon G_\ssc\onto [G,G]\into G$.
Then $G_\tx{sc} \to G$ is a {\em symmetrically braided crossed module}
(see \S\ref{ss:stable-crossed} for the definition), and
we set  $H^i_\tx{ab}(F,G)=H^i(F,G_\tx{sc} \to G)$.
Given a  maximal torus $T\subseteq G$, we set $T_\ssc=\rho^{-1}(T)$.
The natural map $[T_\tx{sc} \to T] \to [G_\tx{sc} \to G]$ is
a quasi-isomorphism of crossed modules,
and this leads to the natural identification $H^i_\tx{ab}(F,G)=H^i(F,T_\tx{sc} \to T)$.
This latter cohomology group  is the hypercohomology group
of the complex of tori $(T_\ssc\to T)$ placed in degrees $-1$ and $0$.
By construction we have
$$H^i(F,T_\tx{sc} \to T)=\varinjlim_E H^i(E/F,T_\tx{sc} \to T),$$
where the colimit is taken over the set of finite Galois extensions $E/F$ in $F^s$,
and $H^i(E/F,T_\tx{sc} \to T)\coloneqq H^i\big(\G\sub{E/F},T_\tx{sc}(E) \to T(E)\big)$.
The group $$H^i\big(\G\sub{E/F},T_\tx{sc}(E) \to T(E)\big)$$ is in turn identified,
by the global Tate--Nakayama isomorphism defined by Tate in \cite{Tate66} for tori and extended
in \S\ref{sub:tiso} of this paper to bounded complexes of tori, with the group
\[ H^{i-2}\big(\G\sub{E/F},\,X_*(T_\tx{sc})[\V_E]_0\! \to\! X_*(T)[\V_E]_0\big) \]
where by $X_*(T)$ we denote the cocharacter group of $T$,
and $ H^{i-2}$ denotes the Tate (hyper)cohomology.
Since the homomorphism $X_*(T_\tx{sc}) \to X_*(T)$
is injective with cokernel $M=\pi_1(G)$, we have
\[ H^{i-2}\big(\G\sub{E/F},\,X_*(T_\tx{sc})[\V_E]_0
   \! \to\! X_*(T)[\V_E]_0\big) =  H^{i-2}(\G\sub{E/F},M[\V_E]_0). \]
This reduces the computation of $H^i_\tx{ab}(F,G)$ to the computation
of the transition maps in the directed system $ H^{i-2}(\G\sub{E/F},M[\V_E]_0)$
and of the colimit of that system.

\section{}
The paper is structured as follows.
At the end of this chapter  we explain our notation and terminology.
In Chapter \ref{sec:tcbc} we review briefly the definition and properties
of Tate hypercohomology of bounded complexes of $\G$-modules, for an arbitrary finite group $\G$.
This is done abstractly via the stable derived category of $\Z[\G]$-modules,
following the ideas of Buchweitz \cite{Buchweitz86}, which have by now been expounded in various references, such as \cite{Krause22}.
This abstract approach has the advantage that many useful properties follow
almost immediately
from the well-known properties of Tate cohomology of $\Z[\G]$-modules.
We then provide explicit formulas that allow us to compute Tate hypercohomology and restriction and corestriction maps by means of cochains.
These explicit formulas actually apply not only to bounded, but also to \emph{unbounded} complexes, and can be used to define and compute Tate hypercohomology in this more general setting as well, although we will have no need for that in this paper.

In Chapter \ref{sec:tniso} we discuss the local and global Tate classes following \cite{Tate66} and review the local and global Tate--Nakayama isomorphisms in the context of bounded complexes of tori. We further discuss the compatibility of these isomorphisms with restriction, corestriction, localization, and connecting homomorphisms.

In Chapter \ref{s:ab-coho} we perform the key computations of this paper. We specialize the discussion of Chapter \ref{sec:tniso} to the case of
a complex of tori $T^{-1} \to T^0$ of length $2$, placed in degrees $-1$ and $0$,
such that the kernel $T^{-1} \to T^0$ is finite.
This is equivalent to the homomorphism $X_*(T^{-1}) \to X_*(T^0)$ being injective.
Let $M$ denote the cokernel of this injective homomorphism.
In this setting, when $K/E/F$ is a tower of finite Galois extensions and $i=1,2$,
we compute the homomorphism of Tate cohomology
\begin{equation}\label{e:not-inf}
H^{i-2}(\G\sub{E/F},M[\V_E]_0) \to  H^{i-2}(\G\sub{K/F},M[\V_K]_0)
\end{equation}
that corresponds to the inflation homomorphism
\[  H^i(E/F,\, T^{-1} \to T_0) \to H^i(K/F,\, T^{-1} \to T^0) \]
under the global Tate isomorphisms.
(Note that this homomorphism \eqref{e:not-inf} is not an inflation homomorphism
because there are no inflation homomorphisms in Tate cohomology $H^n$ for $n\le 0$. Its description is more subtle.)
\ Then we compute the colimit of the groups $H^{i-2}(\G\sub{E/F},M[\V_E]_0)$
over the finite Galois extensions $E/F$.

In Chapter \ref{sec:red} we deal with the Galois cohomology of reductive groups.
After reviewing the abelian Galois cohomology groups $H^n_\tx{ab}(F,G)$
and their relation to the usual Galois cohomology $H^1(F,G)$,
we use the results of Chapter \ref{s:ab-coho} to compute
$H^n_\ab(F,G)$ for $n\ge 1$
for a reductive group $G$ over a global field $F$
in terms of the algebraic fundamental group $\pi_1(G)$,
and to compute $H^1(F,G)$ in terms of
$\pi_1(G)$ and the real Galois cohomology sets $H^1(F_v,G)$ for the real places $v$ of $F$.

In  Chapter \ref{sec:explicit} we give explicit versions
(Theorems \ref{t:FS} and \ref{t:explicit}) of the results of Chapter  \ref{sec:red}.
In  Chapter \ref{sec:Shafarevich} we use the results of Chapter \ref{sec:explicit}
for computations with Tate--Shafarevich kernels.

In Appendix \ref{app:Hinich}, for
a {\em finite} group $\G$ and a short exact sequence of $\G$-modules
$0\to B'\to B\to B''\to 0,$
Vladimir Hinich defines and investigates  an   exact sequence in group homology
containing $B_\Gt$ and $(B\otimes\Q/\Z)_\G$\hs.

In Appendix \ref{app:crossed} we recall the definition of the first  group cohomology set  with coefficients in a crossed module.
In Appendix \ref{app:duality} we prove some duality assertions, and in Appendix \ref{app:fiber}
we prove a lemma dealing with fiber products.
Appendix \ref{app:cheat-sheet} is a cheat sheet for Tate cohomology in low degrees.
The arXiv version \cite{BorKal23} of this paper contains Appendix F, which is the listing of a computer program and its output;
 see \S\ref{ss:computer}.

In the body of the paper, besides dealing with global fields $F$,
we also deal with non-archimedean local fields, and rederive classical results of Kottwitz.

\addtocontents{toc}{\protect\setcounter{tocdepth}{1}}

\section*{Notation and terminology}
\label{sec:notation}

\renewcommand{\thefootnote}{\ensuremath{\fnsymbol{footnote}}}

\begin{itemize}

\item
$\Z,\,\Q,\,\R,\,\C$ denote the ring of integers and the fields
of rational numbers, of real numbers, and of complex numbers, respectively.

\item For a field $F$, we denote by $\overline F$ a fixed algebraic closure of $F$,
and by $F^s$ a separable closure of $F$ in $\overline F$.
When considering a finite separable extension $F'/F$, a Galois extension $E/F$,
or a tower of Galois extensions $K/E/F$,
we tacitly assume that $F',E,K\subset F^s$.

\item For a Galois extension $E/F$, we write $\G(E/F)$ for the Galois group of $E/F$.

\item When $F$ is a global field, we denote by $V_F$ or $V(F)$ the set of places of $F$.
We denote by $V_f(F)$, $V_\infty(F)$, and $V_\R(F)$
the sets of finite (non-archimedean) places,
of infinite (archimedean) places, and of real places, respectively.
In particular, when $F$ is a global function field,
the sets $V_\infty(F)$ and $V_\R(F)$ are empty.

\item When $E/F$ is a Galois extension of global fields with
Galois group $\G=\G(E/F)$, and when $w\in V_E$,
we write $\G\sss{w}=\G(E/F,w)$ for the stabilizer of $w$ in $\G$
with respect to the natural action of $\G$ on $\V_E$.
\footnote{ We write $\G\sss{w}$  and  $\G\sub{E/F,w}$
rather than  $\G\hm_w$ and $\G_{E/F,w}$ in order to avoid
numerous double subscripts in our formulas and commutative diagrams.}

\item When $E/F$ is a finite Galois extension of global fields,
we choose a section $v\mapsto \bv\colon V_F\to V_E$
of the natural surjective map $V_E\to V_F$.
Similarly, we denote $\overline V\coloneqq V\sss{F^s}$
and choose a section $v\mapsto \vbar\colon V_F\to \overline V$
of the natural surjective map $\overline V\to V_F$.
Our results do not depend on these choices.

\item Let $A$ be an abelian group. We write $A_\Tors$
for the torsion subgroup of $A$, and we write  $A_\tf=A/A_\Tors$\hs.

\item Let $\G$ be a profinite group
and let $A$ be a $\G$-module, that is,
an abelian group written additively on which $\G$ acts continuously
with respect to the discrete topology on $A$.
As usual, we denote by $A^\G$ the subgroup of $\G$-invariants in $A$,
and by
\[ A_\G=A/\langle\gamma\cdot a-a\ |\ \gamma\in\G,\, a\in A\hs\rangle \]
the group of $\G$-coinvariants.
We write $A_\Gt=(A_\G)_\Tors$ and $A_\Gtf=(A_\G)_\tf$\hs.
We write $A^\vee$ for the $\Gamma$-module $\tx{Hom}_\Z(A,\Z)$.

\item If $\Gamma=\Gamma(E/F)$ is the Galois group of a finite extension of global fields
and $A$ is a $\Gamma$-module, we follow \S\ref{ss:E/F} and write $\Z[V_E]$
for the free abelian group on the set of places $V_E$ of $E$
equipped with the $\Gamma$-action induced from the $\Gamma$-action on $V_E$.
We write $A[V_E]=A\otimes_\Z\Z[V_E]$, and we write  $A[V_E]_0$
for the kernel of the augmentation map $A[V_E] \to A$
sending $\sum a_w\cdot w$ to  $\sum a_w$\hs.

\item $A\otimes B$ denotes $A\otimes_\Z B$,
the tensor product of two abelian groups $A$ and $B$ over $\Z$.

\item For a subgroup $A$ of an abelian group $B$,
we say that $A$ is a {\em direct summand} of $B$
if there exists a subgroup $C$ of $B$ such that $B=A\oplus C$.
(the internal direct sum of $A$ and $C$).

\item By an {\em exact} commutative diagram we mean a commutative diagram with exact rows and columns.

\end{itemize}

\chapter{Tate cohomology of bounded complexes} \label{sec:tcbc}

Let $\G$ be a finite group. It is well-known that the (usual, that is, not Tate) cohomology
of the group $\Gamma$, which can be constructed explicitly in terms of cochains,
can also be expressed using the (bounded) derived category $\mc{D}^b(\Gamma)$
of the category $\tx{Mod}(\Gamma)$ of all $\Z[\Gamma]$-modules.
Namely, the functor of invariants $\tx{Mod}(\Gamma) \to \tx{Mod}(\{1\})=\mc{A}$
that takes a $\Z[\Gamma]$-module $M$ to the abelian group $M^\Gamma$ of $\Gamma${\hyp}invariants
induces a functor between the derived categories $\mc{D}^b(\Gamma) \to \mc{D}^b(\mc{A})$.
The latter functor is better behaved than the former -- while the former is not exact,
in the sense that it does not map exact sequences to exact sequences, the latter is exact,
in the sense that it maps distinguished triangles to distinguished triangles.
The derived category $\mc{D}^b(\mc{A})$ comes equipped with cohomology functors
$H^i : \mc{D}^b(\mc{A}) \to \mc{A}$, and composing the derived invariants functor with $H^i$
leads to the $i$-th group cohomology functor for $\Gamma$.
The derived category $\mc{D}^b(\Gamma)$ comes equipped with an embedding
$\tx{Mod}(\Gamma) \to \mc{D}^b(\Gamma)$, and an equivalent description of the $i$-th group cohomology functor
is $\tx{Hom}_{\mc{D}^b(\Gamma)}(\Z,M[i])$, where $\Z$ is the trivial $\Gamma$-module,
and $[i]$ is the shift-functor on $\mc{D}^b(\Gamma)$.

A very useful byproduct of this interpretation of group cohomology is that it extends immediately from the case of individual $\Gamma$-modules to the case of bounded complexes of $\Gamma$-modules, which is classically called ``hypercohomology''.

In this chapter we will recall that a similar interpretation can be given to the Tate cohomology of $\G$. This originates in the work of Buchweitz \cite{Buchweitz86}, but has since been covered in other sources, such as \cite{CTVEZ} or \cite{Krause22}, and is based on the notion of the ``stable module category'', equivalently the ``stable derived category'', a certain quotient of the derived category. The $i$-th Tate cohomology group of the $\Gamma$-module $M$ is then given by $\tx{Hom}_{\ul{\mc{D}}^b(\Gamma)}(\Z,M[i])$, where $\ul{\mc{D}}^b(\Gamma)$ is the stable module category of $\Gamma$. As with usual cohomology, one of the benefits is that one obtains at once an extension to bounded complexes of $\Gamma$-modules.

For an introduction to derived categories we refer the reader to \cite[Chapter III]{GelfandManin} and for the relation to (usual) group cohomology we refer to \cite[Chapters 6 and 10, esp. \S10.7]{Weibel94}.

Since in this paper we will almost exclusively use Tate cohomology, our notation for Tate cohomology will be $H^i(\G,\,\cdot\, )$.

\section{The stable module category}

We give here three presentations of the stable derived category of $\Z[\G]$-mod\-ules, which we will denote by $\ul{\mc{D}}^b(\G)$. While the original treatment is given in  \cite{Buchweitz86}, we follow here the presentation of \cite{BeilinsonCFT}.

Let $\tx{Mod}(\G)$ denote the category of $\Z[\G]$-modules, and $\tx{mod}(\G)$ the subcategory of finitely generated (equivalently, finitely presented, since $\Z[\G]$ is Noetherian) modules. We write $\tx{Mod}_\tx{tf}(\G)$ for the subcategory of torsion-free modules, and $\tx{mod}_\tx{tf}(\G)=\tx{Mod}_\tx{tf}(\G) \cap \tx{mod}(\G)$.

Let $\ul{\mc{D}}^b_1(\G)$ denote the category whose objects are the torsion-free $\Z[\G]$-modules,
and given two such $M,N$, the set of morphisms is the abelian group quotient of the abelian group $\tx{Hom}_{\Z[\G]}(M,N)$ by the subgroup consisting of those morphisms that factor through a projective $\Z[\G]$-module, that is, homomorphisms of the form $M \to P \to N$ with $P$ projective. We could equivalently replace $P$ by a free $\Z[\G]$-module $F$. This has the effect of quotienting out the projective modules and leads to the ``stable module category'' of $\Z[\Gamma]$-modules.

Let $\ul{\mc{D}}^b_2(\G)$ denote the Verdier quotient (cf. \cite[13.6]{stacks-project})
of the bounded derived category $\mc{D}^b(\G)$ of $\Z[\G]$-modules by the full triangulated subcategory consisting of perfect complexes (bounded complexes of projective $\Z[\G]$-modules). This is the ``stable derived category'' of $\Z[\Gamma]$-modules.

Let $\ul{\mc{D}}^b_3(\G)$ denote the homotopy category of exact complexes of projective $\Z[\G]$-modules.

We now describe three functors
\[\ul{\mc{D}}^b_1(\G) \to \ul{\mc{D}}^b_2(\G) \to \ul{\mc{D}}^b_3(\G) \to \ul{\mc{D}}^b_1(\G).\]

The functor $\ul{\mc{D}}^b_1(\G) \to \ul{\mc{D}}^b_2(\G)$ is induced by the usual fully faithful embedding of an abelian category in its derived category. In other words, a torsion-free $\Z[\G]$-module $M$ is sent to the object in the derived category that is represented by $M$ considered as a complex concentrated in degree $0$.

To construct the functor $\ul{\mc{D}}^b_2(\G) \to \ul{\mc{D}}^b_3(\G)$, consider an object in $\ul{\mc{D}}^b_2(\G)$ represented by a bounded complex $M^\bullet$. Choose quasi-isomorphisms $M_l^\bullet \to M^\bullet$ and $M^\bullet \to M_r^\bullet$, where the complexes $M_l^\bullet$ and $M_r^\bullet$ consist of free $\Z[\G]$-modules, $M_l^\bullet$ is bounded above, and $M_r^\bullet$ is bounded below. The composition $M_l^\bullet \to M_r^\bullet$ is also a quasi-isomorphism, so its cone is an object of $\ul{\mc{D}}^b_3(\G)$.

The functor $\ul{\mc{D}}^b_3(\G) \to \ul{\mc{D}}^b_1(\G)$ sends an exact complex $M^\bullet$ of projective $\Z[\G]$-modules to $\tx{ker}(d^0_M\colon M^0 \to M^1)$.

One can impose the condition of finite generation on all modules, and thereby arrive at a full subcategory $\ul{\mc{D}}_i^{b,c} \subset \ul{\mc{D}}_i^b$.

The proofs that the above functors between the categories $\ul{\mc{D}}^{b,c}_i(\G)$ are equivalences of categories can be found in \cite[\S6.2]{Krause22}, in particular Theorem 6.2.5 and the sections following it. The arguments also work for $\ul{\mc{D}}^{b}_i(\G)$, and results are recorded in \cite{ChenXW11}.

We will regard these three categories as three different presentations of $\ul{\mc{D}}^b(\G)$ resp. $\ul{\mc{D}}^{b,c}(\G)$. The category $\ul{\mc{D}}^{b,c}(\G)$ coincides with the subcategory of compact objects in $\ul{\mc{D}}^b$, cf. \cite{ChenXW11} (note that a direct summand in $\tx{Mod}(\G)$ of a member of $\tx{mod}(\G)$ still lies in $\tx{mod}(\G)$).

The category $\ul{\mc{D}}^b(\G)$ is triangulated. This is evident from the presentation $\ul{\mc{D}}^b_2(\G)$.
In presentation $\ul{\mc{D}}^b_1(\G)$ the shift functor $M \to M[1]$ sends a torsion-free $\Z[\G]$-module $M$
to a $\Z[\G]$-module $N$ which sits in an exact sequence
\[ 0 \to M \to P \to N \to 0 \]
with $P$ a projective $\Z[\G]$-module.

The category $\ul{\mc{D}}^b(\G)$ has a contravariant idempotent endofunctor. In the presentation $\ul{\mc{D}}^b_1(\G)$ it sends a torsion-free $\Z[\G]$-module $M$ to $M^\vee=\tx{Hom}_\Z(M,\Z)$. We have $M[1]^\vee=M^\vee[-1]$. Note that $M^\vee$ is injective if and only if $M$ is projective, and vice versa.

\begin{fact} \label{fct:fsd}
The natural bilinear extension $\Z[\G] \otimes_\Z \Z[\G] \to \Z$
of the Kronecker delta symbol $\delta\colon \G \times \G \to \Z$ is a perfect pairing,
that is, it induces an isomorphism $\Z[\G] \to \tx{Hom}_\Z(\Z[\G],\Z)$.
In particular, every finitely generated free $\Z[\G]$-module is self-dual, projective, and injective.
\end{fact}

The category $\ul{\mc{D}}^b(\G)$ has a symmetric monoidal structure. In the presentation $\ul{\mc{D}}^b_1(\G)$ it is induced by the symmetric monoidal structure on the category of all $\Z[\G]$-modules, that is, the usual tensor product $M \otimes_\Z N$ of such modules.

\section{Tate cohomology}

We define for each $i \in \Z$ the functor
\[ H^i(\G,\,\cdot\, )\colon\, \ul{\mc{D}}^b(\G) \to \tx{Ab},\qquad M \mapsto \tx{Hom}_{\ul{\mc{D}}^b(\G)}(\Z,M[i]) \]
to the category of abelian groups $\Ab$.
If $M$ is a $\Z[\G]$-module, interpreted as an element of $\ul{\mc{D}}^b(\G)$ via the natural embedding $\tx{Mod}(\G) \to \mc{D}^b(\G)$ and the natural projection $\mc{D}^b(\G) \to \ul{\mc{D}}^b(\G)$, this coincides with the usual definition of Tate cohomology, cf. \cite[Proposition 2.6.2]{CTVEZ}.
Using the presentation $\ul{\mc{D}}_2^b(\G)$ for $\ul{\mc{D}}^b(\G)$ we obtain in this way an extension of Tate cohomology to bounded complexes of $\Z[\G]$-modules.
From this definition it is clear that Tate cohomology is compatible with shifting complexes, that an exact triangle induces a long exact cohomology sequence, and that quasi-isomorphisms of bounded complexes induce isomorphisms on Tate cohomology.

The symmetric monoidal structure on $\ul{\mc{D}}^b(\G)$ induces products on Tate cohomology, namely
\[ \tx{Hom}_{\ul{\mc{D}}^b(\G)}(\Z,M[i]) \otimes \tx{Hom}_{\ul{\mc{D}}^b(\G)}(\Z,N[j]) \to \tx{Hom}_{\ul{\mc{D}}^b(\G)}(\Z,(M\otimes N)[i+j]) \]
given by
\[ \Z = \Z \otimes_\Z \Z \to M[i] \otimes_\Z N[j] = (M \otimes N)[i+j]. \]
By uniqueness, these products coincide up to sign with the cup products
defined in terms of the standard resolution \cite[Chapter IV, \S7]{CasFro86}.

\section{Restriction and corestriction}
\label{ss:Res-Cor}
Consider a subgroup $\Delta \subset \G$. Given a $\Z[\G]$-module $M$ we have the homomorphisms
\[ \tx{Res}\colon H^i(\G,M) \longleftrightarrow H^i(\Delta,M)\, :\! \tx{Cor}. \]
They are functorial in $M$.
Restricting them to torsion-free $\Z[\G]$-modules we obtain morphisms between the functors
\[ H^i(\G,\,\cdot\, ) \colon\ul{\mc{D}}^b(\G) \to \tx{Ab}\quad\ \text{and}\quad\  H^i(\Delta,\,\cdot\, )\colon \ul{\mc{D}}^b(\Delta) \to \tx{Ab}. \]
Since the identities $\tx{Cor}\circ\tx{Res}=[\G:\Delta] \cdot $ and $\tx{Res} \circ \tx{Cor} = N_{\G/\Delta}$ (the second when $\Delta$ is normal) hold for all torsion-free $\Z[\G]$-modules, they carry over to all bounded complexes of $\Z[\G]$-modules.

\section{The standard resolution}
\label{ss:St-res}

We recall the standard homological homogeneous free resolution $P_n$ of the trivial $\G$-module $\Z$. Thus $P_n=\Z[\G^{n+1}]$ is the free $\Z$-module with basis $\G^{n+1}$ and $\partial_n\colon P_n \to P_{n-1}$ is the unique $\Z$-linear map that satisfies
\[ \partial_n([g_0,\dots,g_n]) = \sum_{i=0}^n (-1)^i [g_0,\dots,\hat g_i,\dots,g_n], \]
where the ``hat'' symbol signifies that $g_i$ has been omitted.
Letting $\G$ act on $P_n$ by $g \cdot [g_0,\dots,g_n] = [gg_0,\dots,gg_n]$
turns $P_n$ into a free $\Z[\G]$-module.
If we interpret $P_{-1}=\Z[\G^0]=\Z$, then $\partial_0\colon \Z[\G] \to \Z$ becomes the augmentation map,
and this leads to the exact sequence of $\Z[\G]$-modules
\[ \dots \to P_3 \to P_2 \to P_1 \to P_0 \to \Z \to 0. \]

For any $\G$-module $M$ we have the cohomological complex $\tx{Hom}_{\Z[\G]}(P_n,M)=\tx{Maps}_\G(\G^{n+1},M)$
with differential
\[ d^n\colon \tx{Maps}_\G(\G^n,M) \to \tx{Maps}_\G(\G^{n+1},M) \]
sending $c \in \tx{Maps}_\G(\G^n,M)$ to $d^nc=c \circ \partial_{n+1}$.
Here $\tx{Hom}_{\Z[\G]}$ denotes the group of homomorphisms between $\Z[\G]$-modules, and $\tx{Maps}_\G$ denotes the group of maps between $\G$-sets.
The cohomology groups of this complex are by definition the (usual, not Tate) cohomology groups of $\G$ with values in $M$.
Thus, the $n$-th cohomology group is the quotient of the group of $n$-cocycles,
that is $c \in \tx{Maps}_\G(\G^{n+1},M)$ such that $d^{n+1}c=0$, by the group of $n$-coboundaries,
which are elements in the image of $d^n$.

For any $\G$-module $M$ we have the homological complex $(P_n \otimes_{\Z[\G]} M)_{n=0}^\infty$ with differential $\partial_n \otimes \tx{id}_M$, whose homology is by definition the homology of $\G$ with values in $M$. Thus, the $n$-th homology group is the quotient of the group of $n$-cycles, which are finite sums $\sum [g_0,\dots,g_n] \otimes m_{g_0,\dots,g_n}$ killed by the differential $\partial_n \otimes \tx{id}_M$, by the group of $n$-boundaries, which are elements in  the image of $\partial_{n+1} \otimes \tx{id}_M$.

An element of $P_n \otimes_{\Z} M$, which is a finite sum $\sum [g_0,\dots,g_n] \otimes m_{g_0,\dots,g_n}$,
can also be represented by the function $(g_0,\dots,g_n) \mapsto m_{g_0,\dots,g_n}$,
which is an element of $\tx{Hom}_\Z(P_n,M)$.
More formally, we have the perfect $\G$-equivariant pairing of $\Z$-modules $P_n \otimes P_n \to \Z$
sending $[g_0,\dots,g_n] \otimes [h_0,\dots,h_n]$ to $\prod_{i=0}^n \delta_{g_i,h_i}$,
where $\delta_{g_i,h_i}$ is the Kronecker symbol. This gives the identification  $P_n = \tx{Hom}_\Z(P_n,\Z)$
as $\Z[\G]$-modules, where $g \in \G$ acts on $f \in \tx{Hom}_\Z(P_n,\Z) = \tx{Maps}(\G^{n+1},\Z)$
by the rule
\[ (gf)(h_0,\dots,h_n)=f(g^{-1}h_0,\dots,g^{-1}h_n). \]
From this we obtain in turn the identification
\[ P_n \otimes_\Z M = \tx{Hom}_\Z(P_n,\Z) \otimes M = \tx{Hom}_\Z(P_n,M),\]
 and hence
\[ P_n \otimes_{\Z[\G]} M = (P_n \otimes_\Z M)_\G = \tx{Hom}_\Z(P_n,M)_\G \]
where the subscript $\G$ denotes coinvariants.

One checks that the differential \[ \partial_n \colon P_n = \tx{Hom}_\Z(P_n,\Z) \to \tx{Hom}_\Z(P_{n-1},\Z) = P_{n-1} \]
is given by the pull-back along the $\Z$-linear map $\partial^n \colon P_{n-1} \to P_n$ defined by
\[ \partial^n(g_0,\dots,g_{n-1}) = \sum_{i=0}^n(-1)^i \sum_{h \in \G} [g_0,\dots,g_{i-1},h,g_i,\dots,g_{n-1}].\]

We thus see that the standard {\em homological} homogeneous free resolution of $\Z$
is the $\Z$-dual of the standard {\em cohomological}  homogeneous free resolution of $\Z$,
namely the cohomological complex $\Z \to P_0 \to P_1 \to \dots$ with differential $\partial^n$. In order to streamline notation, we will write $P^n$ instead of $P_n$ in this case. Note that $\partial^0 \colon \Z \to \Z[\G]$ is the diagonal map $1 \mapsto \sum_{h \in \G}[h]$.

We now have
\[ P_n \otimes_{\Z[\G]}M = (P_n \otimes_\Z M)_\G = \tx{Hom}_\Z(P^n,M)_\G \,\isoto\,\tx{Hom}_\Z(P^n,M)^\G, \]
where the final arrow  is given by the norm map for the action of $\G$, and it is an isomorphism because $P_n$ is a free $\Z[\G]$-module, which makes $\tx{Hom}_\Z(P_n,M)$ into an induced, hence cohomologically trivial, $\G$-module.
This final isomorphism commutes with pull-back along $\partial^n$.
Therefore, we can think of an $n$-chain as either a $\G$-equivariant map $\G^{n+1} \to M$,
or as an arbitrary map $\G^{n+1} \to M$ that is given up to maps of the form $c-gc$.
In either interpretation, an $n$-cycle is such $c$ with the condition $c \circ \partial^n=0$,
and an $n$-boundary is an element of the form $c \circ \partial^{n+1}$ for $c \colon \G^{n+2} \to M$.

The two resolutions $P_\bullet \to \Z$ and $\Z \to P^\bullet$ can be spliced to a bi-infinite resolution
\begin{equation} \label{eq:biinf}
\dots P_2 \to P_1 \to P_0 \to P^0 \to P^1 \to P^2 \to \dots,
\end{equation}
where the differential $P_0 \to P^0$ is given by $\partial^0 \circ \partial_0$. This is an exact complex of $\Z[\G]$-modules. We can write it in homological notation by setting $P_{-n}=P^{n-1}$ and $\partial_{-n}=\partial^{n-1}$ for $n<0$. Then $\tx{Hom}_\Z(P_n,M)^\G$ is a cohomological complex, whose cohomology groups are the Tate cohomology groups of $\G$ with values in $M$.

\section{Tate cohomology in terms of the standard resolution}

Let $A$ be a complex of $\G$-modules. We will write the complex in cohomological notation and denote its differential
by $f_A^k \colon A^k \to A^{k+1}$. We will define $C^n(\G,A)$ as the $n$-th term of the Hom-complex $\tx{Hom}_{\Z[\G]}(\tilde P_\bullet,A^\bullet)$, where $\tilde P_\bullet$ is the bi-infinite resolution \eqref{eq:biinf} but with one small change: we replace the differential $\partial_n$ in that resolution with the differential $\tilde\partial_n = (-1)^{n+1}\partial_n$. Of course, $\tilde P_\bullet$ is again an acyclic complex by free $\Z[\G]$-modules. Thus, more explicitly,
\[ C^n(\G,A) = \prod_k \tx{Hom}_{\Z[\G]}(P_{n-k},A^k), \]
with differential $d^n\colon C^n(\G,A) \to C^{n+1}(\G,A)$ that  is given on the direct factor $\tx{Hom}_{\Z[\G]}(P_{n-k},A^k)$ of $C^n(\G,A)$  by the $\Z$-linear map
\[ \tx{Hom}_{\Z[\G]}(P_{n-k},A^k) \to \tx{Hom}_{\Z[\G]}(P_{n-k},A^{k+1}) \oplus \tx{Hom}_{\Z[\G]}(P_{n+1-k},A^k) \]
defined by
\[ d^n c^k= f^k_A \circ c^k  + (-1)^k c^k \circ \partial_{n+1-k}. \]
With this differential, $C^n(\G,A)$ becomes a cohomological complex.

\begin{proposition} Let $A$ be a bounded complex of $\Z[\G]$-modules. The cohomology groups of the complex $C^n(\Gamma,A)$ coincide with the Tate cohomology groups of $A$ defined in terms of the stable module category, functorially in $A$.
\end{proposition}
\begin{proof}
	As we have already remarked, for complexes concentrated in one degree this follows from \cite[Proposition 2.6.2]{CTVEZ}.

	Write temporarily $\tilde H^n(\Gamma,A)$ for the $n$-cohomology group of the complex $C^n(\Gamma,A)$. We first check that the functor $A \mapsto \tilde H^n(\G,A)$ descends to the derived category of $\Z[\G]$-modules. For this it is enough to show that it sends quasi-isomorphisms of complexes to isomorphisms of abelian groups \cite[Chap. III,\S2]{GelfandManin}. Thus we want to show that $A \mapsto C^n(\G,A)$ maps quasi-isomorphisms of $\Z[\G]$-modules to quasi-isomorphisms of $\Z$-modules.
	
	A morphism $f\colon A \to B$ of complexes is a quasi-isomorphism if and only if its cone $\tx{hcok}(f)$ is acyclic. Now $C^n(\G,A)$ is defined as the Hom-complex for the pair $\tilde P_\bullet$ and $A^\bullet$, and since the Hom-complex respects cones, we have $C^n(\G,\tx{hcok}(f))=\tx{hcok}(C^n(\G,f))$. Note that the boundedness of $A$ and $B$ implies the boundedness of $\tx{hcok}(f)$. This reduces the proof to showing that $C^\bullet(\G,A)$ is acyclic provided $A$ is acyclic. Since $A$ is bounded, we may replace $P_\bullet$ by any finite truncation when checking the acyclicity of $C^\bullet(\G,A)$ complex at a particular degree. But a bounded complex of projective $\Z[\G]$-modules is a projective complex of $\Z[\G]$-modules, and the claim follows.

	Next one checks that $\tilde H(\Gamma,-)$ is a cohomological functor. The arguments are analogous to the case of usual Tate cohomology and are left to the reader (cf. \cite[VII.5]{Brown}).

	Next one checks that if $A$ is a perfect complex, then $\tilde H(\Gamma,A)=0$. This is done by induction on the length of the complex $A$, using the fact that $\tilde H(\Gamma,-)$ is a cohomological functor, and the fact that for a projective $\Z[\G]$-module $A$ (seen as a complex concentrated in a single degree) we have $\tilde H(\G,A)=0$.

	It follows from the universal property of the Verdier quotient (\cite[Lemma I.13.6.8]{stacks-project}) that $\tilde H(\G,-)$ descends to $\ul{\mc{D}}_2^b(\G)$. Composing this functor with the isomorphism $\ul{\mc{D}}_1^b(\G) \to \ul{\mc{D}}_2^b(\G)$ reduces the proof to showing that $\tilde H(\G,-)$ and $H(\G,-)$ coincide for complexes concentrated in a single degree, which has been established.
\end{proof}

The usage of the modified differential $\tilde\partial_n$ was suggested to us by Kottwitz. It has the effect that, when $A$ is a complex concentrated in degree $0$, the cochain complex $C^n(\G,A)$ defined above coincides with the usual cochain complex in the definition of Tate cohomology of $\G$-modules.

It is useful to record that, for $c \in C^n(\G,A)$, we have
\begin{equation} \label{eq:hyperdif}
(dc)^k = f^{k-1}_A \circ c^{k-1}+(-1)^k\circ\partial_{n+1-k} \in \tx{Hom}_{\Z[\G]}(P_{n+1-k},A^k).	
\end{equation}

\begin{remark}
	The above formula for $C^n(\G,A)$ and its differential can also be used to compute Tate cohomology for an unbounded complex $A$.
\end{remark}

\section[Restriction and corestriction]{Restriction and corestriction in terms of the standard resolution}
\label{sub:rescor}

In this section we will give explicit formulas for the  restriction and corestriction maps in terms of the standard resolution
(in its homogeneous form; for the inhomogeneous form, see Appendix \ref{app:cheat-sheet}).

Let $\Delta \subset \G$ be a subgroup and let $s\colon \Delta \lmod \G \to \G$ be a section of the natural projection.

We begin with corestriction. For a $\Z[\G]$-module $M$ and $n \in \Z$ define
\[ \tx{cor}^n\colon C^n(\Delta,M) \to C^n(\G,M) \]
as follows. If $n \geq 0$, then
\[ \tx{cor}^n(c)(g_0,\dots,g_n) = \sum_{x \in \Delta \lmod \G} s(x)^{-1}c(s(x)g_0s(xg_0)^{-1},\dots,s(x)g_ns(xg_n)^{-1}). \]
If $n<0$, then
\[ \tx{cor}^n(c)(g_0,\dots,g_{-n-1}) =\! \begin{cases}
gc(h_0,\dots,h_{-n-1}) &\!\!\text{if }\, \exists g \in \G,h_i \in\Delta\,\text{ such that }\, g_i=gh_i\\
0 &\!\!\text{otherwise}.
\end{cases}
\]

\begin{lemma} \label{lem:cor1}
For every $n \in \Z$ we have $\tx{cor}^{n+1} \circ d_\Delta^n = d_\G^n \circ \tx{cor}^n$. Moreover, $\tx{cor}^n$ is functorial in $M$.
\end{lemma}
\begin{proof}
This is a simple exercise using the definitions.
\end{proof}

We now turn to restriction. For  a $\Z[\G]$-module $M$ and $n \in \Z$ define
\[ \tx{res}^n\colon C^n(\G,M) \to C^n(\Delta,M) \]
as follows. If $n < 0$, then
\[ \tx{res}^n(c)(h_0,\dots,h_{-n-1}) = \sum_{x \in (\Delta\lmod \G)^{-n}} c(h_0s(x_0),\dots,h_{-n-1}s(x_{-n-1})). \]
If $n \geq 0$, then
\[ \tx{res}^n(c)(h_0,\dots,h_n) = c(h_0,\dots,h_n). \]

\begin{lemma} \label{lem:res1}
For every $n \in \Z$ we have $\tx{res}^{n+1} \circ d_\G^n = d_\Delta^n \circ \tx{res}^n$. Moreover, $\tx{res}^n$ is functorial in $M$.
\end{lemma}

\begin{proof}
This is a simple exercise using the definitions.
\end{proof}

Let $A$ be a cohomological complex of $\Z[\G]$-modules. Define
\[ \tx{cor}^n\colon C^n(\Delta,A) \to C^n(\G,A) \quad\tx{and}\quad \tx{res}^n\colon C^n(\G,A) \to C^n(\Delta,A) \]
by
\[ \tx{cor}^n(c)^k = \tx{cor}^{n-k}(c^k) \quad\tx{and}\quad \tx{res}^n(c)^k = \tx{res}^{n-k}(c^k). \]

\begin{corollary}
For every $n \in \Z$ we have $\tx{cor}^{n+1} \circ d_\Delta^n = d_\G^n \circ \tx{cor}^n$. The same holds for $\tx{res}$ in place of $\tx{cor}$. Both are functorial in $A$.
\end{corollary}

\begin{proof}
This follows from Lemmas \ref{lem:cor1} and \ref{lem:res1}, and \eqref{eq:hyperdif}.
\end{proof}

\begin{lemma} \label{lem:corres}
Let $A$ be a bounded complex of $\Z[\G]$-modules.
	The maps
\[ \tx{cor}^n\colon H^n(\Delta,A) \longleftrightarrow H^n(\G,A)\hs:\hm \tx{res}^n \]
obtained from the above explicitly constructed cochain maps coincide with the maps defined abstractly in terms of the stable category.
\end{lemma}

\begin{proof}
This is well-known for complexes that are concentrated in degree $0$, see e.g. \cite[I.5]{NSW08}. Since these maps are morphisms between functors that originate in the stable category, and the functor from $\tx{Mod}_\tx{tf}(\Z[\G])$ to the stable category is essentially surjective, we obtain the result for all objects in the stable category.
\end{proof}

\begin{remark}
	The above definitions also define restriction and corestriction for unbounded complexes.
\end{remark}

\section{Explicit formulas in degrees $-1$ and $0$}  \label{sub:explicit0-1}

For any $\Z[\G]$-module $A$, the norm map $N_\G \colon A \to A^\G$ descends to a map $A_\G \to A^\G$,
which we will again denote by $N_\G \colon A_\G \to A^\G$. This construction has the following relative variant.
Let $\G' \subset \G$ be a subgroup.
Given a section $s \colon \G' \lmod \G \to \G$, we can consider $\sum_{x \in \G' \lmod \G} s(x)\in\Z[\G]$,
which we regard as a map $\sum_{x \in \G' \lmod \G} s(x) :\, A \to A$.
Composing this map with the natural projection $A \to A_{\G'}$, we obtain a map $A \to A_{\G'}$,
which is independent of the choice of section $s$ and descends to a map $A_{\G} \to A_{\G'}$.
We will call this map $N_{\G' \lmod \G}$.
It is clearly functorial in $A$. When $\G'$ is normal in $A$, $N_{\G' \lmod \G}$ takes values in $(A_{\G'})^{\G/\G'}$, and equals the absolute map
\[ N_{\G/\G'}\colon A_\G = (A_{\G'})_{\G/\G'} \lra (A_{\G'})^{\G/\G'}. \]
We also consider the map $\sum_{x \in \G' \lmod \G} s(x)^{-1}\! : \, A^{\G'} \to A^\G$,
which is independent of the choice of section $s$  and which we denote by $N_{\G/ \G'}$.

\begin{fact} \label{fct:h0}
	The augmentation map $\epsilon\colon A[\G] \to A$ induces an isomorphism
\[ H^0(\G,A)=Z^0(\G,A)/B^0(\G,A)\, \isoto\, A^\G/N_\G(A), \]
which is functorial in $A$.
\end{fact}

\begin{fact} \label{fct:372}
Let $A$ be a $\G$-module and $B$ be an abelian group
regarded as a $\G$-module with trivial $\G$-action.
Then we have a canonical isomorphism $A_\G\otimes B\hs\isoto\hs (A\otimes B)_\G\hs$.
\end{fact}

\begin{fact} \label{fct:Q}
Let $B$ be an abelian group and $B_\tf = B/B_{\Tors}$ its torsion-free quotient. The natural map $B \to B_\tf$ induces an isomorphism $B\otimes\Q \hs\isoto\hs B_\tf \otimes\Q$.
\end{fact}

\begin{fact} \label{fct:Q/Z}
Let $B$ be a {\em torsion-free} abelian group. The inclusion $B \to B\otimes\Q$
induces an isomorphism $B\otimes(\Q/\Z) \hs\isoto\hs (B\otimes\Q)/B$.
\end{fact}

Let $A$ be a $\G$-module.
The norm homomorphism $ A_\G \to A^\G$ is in general neither injective nor surjective;
its kernel and cokernel are torsion.
The norm homomorphism $ (A\otimes \Q)_\G\to (A\otimes \Q)^\G$ is an isomorphism,
and its inverse sends $x \in (A\otimes\Q)^\G$
to the element of $(A\otimes\Q)_\G$ represented by $|\G|^{-1} \cdot x$; \,the tensor products are taken over $\Z$.

The inclusion map $i\colon A_\tf\into A_\tf\otimes \Q$
induces a  not necessarily injective homomorphism
\[i_*\colon (A_\tf)_\G\to (A_\tf)_\G\otimes\Q=(A_\tf\otimes \Q)_\G\hs, \]
while the norm homomorphism
\[N_\G\colon A_\tf\otimes \Q\to (A_\tf\otimes \Q)^\G\]
induces an  isomorphism $(A_\tf\otimes \Q)_\G\to (A_\tf\otimes \Q)^\G$ and an isomorphism
\[N_{\G\!,*}\colon\,(A_\tf\otimes \Q)_\G/i_*\big((A_\tf)_\G\big)\,\longisoto\, (A_\tf\otimes\Q)^\G/N_\G(A_\tf).\]
We now define a homomorphism
\begin{equation} \label{eq:xi}
	\xi_\G\colon H^0(\G,A) \to (A\otimes\Q/\Z)_\G
\end{equation}
to be the composition
\begin{multline*}
H^0(\G,A)=A^\G/N_\G(A)\to (A\otimes \Q)^\G/N_\G(A) =(A_\tf\otimes \Q)^\G/N_\G(A_\tf)\\
\labeltoo{N_{\G\!,*}^{-1}} (A_\tf\otimes \Q)_\G/i_*\big((A_\tf)_\G\big)=\big ( (A_\tf\otimes\Q)/A_\tf\big)_\G\\
\cong (A_\tf\otimes\Q/\Z)_\G\cong(A\otimes \Q/\Z)_\G.
\end{multline*}

\begin{remark}\label{r:xi}
When $A$ is torsion-free, the homomorphism $\xi_\G$ is injective.
\end{remark}

\begin{fact} \label{fct:xig}
The composition of $\xi_\G$ with the natural map $A^\G \to A^\G/N_\G(A) = H^0(\G,A)$
sends $x \in A^\G$ to the class of $|\G|^{-1}\cdot x$.
\end{fact}

\begin{lemma} \label{lem:res0}
Let $\G' \subset \G$ be a subgroup. The following diagrams commute:
\[ \xymatrix{
	A^\G\ar[r]\ar[d]&H^0(\G,A)\ar[d]^{\tx{Res}}\ar[r]^-{\xi_\G}&(A \otimes\Q/\Z)_\G\ar[d]^{N_{\G' \lmod \G}}\\
	A^{\G'}\ar[r]&H^0(\G',A)\ar[r]^-{\xi_{\G'}}&(A \otimes\Q/\Z)_{\G'},
}\]
where the left-hand vertical arrow is the natural inclusion, and
\[ \xymatrix{
	A^\G\ar[r]&H^0(\G,A)\ar[r]^-{\xi_\G}&(A \otimes\Q/\Z)_\G\\
	A^{\G'}\ar[r]\ar[u]^{N_{\G / \G'}}&H^0(\G',A)\ar[r]^-{\xi_{\G'}}\ar[u]^{\tx{Cor}}&(A \otimes\Q/\Z)_{\G'}\ar[u],
}\]
where the right-hand vertical arrow is the natural projection.
\end{lemma}

\begin{proof}
The commutativity of the left-hand rectangles follows from Lemma \ref{lem:corres} and the explicit description of restriction and corestriction in terms of cochains. The commutativity of the right-hand rectangles will follow from that of the outer rectangles. That commutativity in turn is immediate from Fact \ref{fct:xig} and the identity $\sum_{\sigma \in \G/\G'}\sigma(x)=|\G/\G'|\cdot x$ in $(A\otimes\Q/\Z)_\G$, for any $x \in A^{\G'}$.
\end{proof}

Let ${}_{N_\G}A$
denote the kernel of the homomorphism $N_\G\colon A\to A^\G$ and let $I_\G$ be
the kernel of the augmentation map $\Z[\G] \to \Z$, that is,
the ideal in $\Z[\G]$ generated by the elements $\gamma-1$ for $\gamma\in\G$.

\begin{fact} \label{fct:h-1}
The augmentation map $\epsilon\colon A[\G] \to A$ induces an isomorphism
\[ H^{-1}(\G,A)=Z^{-1}(\G,A)/B^{-1}(\G,A) \to  {}_{N_\G}A/I_\G A, \]
which is functorial in $A$.
Furthermore, we have the natural inclusion ${}_{N_\G}A\,/I_\G A \subset (A_\G)_{\Tors}$\hs,
which is an isomorphism when $A$ is torsion-free.
\end{fact}

\begin{lemma} \label{lem:connect-10}
For a short exact sequence of $\G$-modules  $0 \to A' \to A \to A'' \to 0,$ the following diagram commutes:
\[ \xymatrix{
	H^{-1}(\G,A'')\ar[d]_-\cong\ar[r]&H^0(\G,A')\ar[d]^-\cong\\
	{}_{N_\G}A''/I_\G(A'')\ar[r]&(A')^\G/N_\G(A')
}\]
where the top horizontal arrow is the connecting homomorphism,
the bottom horizontal arrow sends  the class of $a''\in {}_{N_\G}A''$
to the class of $N_\G(a)\in (A')^\G$
for an arbitrary lift $a \in A$ of $a''$,
and the vertical isomorphisms are those of  Facts \ref{fct:h-1} and \ref{fct:h0}.
\end{lemma}

\begin{proof}
This follows at once from the description of the differential in the standard resolution.
\end{proof}

  Let $0 \to A_1 \to A_2 \to A_3 \to 0$ be a short exact sequence of $\Z[\G]$-modules. We recall the map
\begin{equation} \label{eq:connect-delta}
	\delta\colon (A_3)_{\G,{\Tors}} \to (A_1 \otimes\Q/\Z)_\G
\end{equation}
defined in Theorem \ref{t:Hinich} in Appendix \ref{app:Hinich} where it is denoted by $\delta_0$.
Given $a_3 \in A_3$ and $n \in \Z$
such that $na_3 \in I_\G A_3$, we have $na_3 = \sum_\gamma (\gamma a_{3,\gamma}-a_{3,\gamma})$
for some $a_{3,\gamma} \in A_3$.
Lift $a_3,\hs a_{3,\gamma}$ arbitrarily to $a_2,\hs a_{2,\gamma} \in A_2$.
Then $a_1\coloneqq na_2-\sum_\gamma (\gamma a_{2,\gamma}-a_{2,\gamma}) \in A_1$, and we set $\delta[a_3]=[n^{-1}a_1]$;
see Proposition \ref{p:delta0}.
One checks easily that $\delta[a_3]$ depends only on the class $[a_3]$  of $a_3$ in $(A_3)_\G$.

\begin{lemma} \label{lem:connect-delta}
The following diagram commutes:
\[ \xymatrix{
	H^{-1}(\G,A_3)\ar[r]\ar[d]&H^0(\G,A_1)\ar[d]^{\xi_\G}\\
	(A_3)_{\G,{\Tors}}\ar[r]^-\delta&(A_1 \otimes \Q/\Z)_\G
}\]
where the top horizontal arrow  is the connecting homomorphism,
the left-hand vertical arrow  is that of Fact \ref{fct:h-1}
and the homomorphism $\delta$ is defined above.
\end{lemma}

\begin{proof}
Using Lemma \ref{lem:connect-10}, the proof is reduced to showing that,
with the notation of the definition of the map $\delta$ given above, $N_{E/F}(a_2)=N_{E/F}(n^{-1}a_1)$.
This, however, is obvious.	
\end{proof}

\begin{lemma} \label{lem:res-1}
Let $A$ be a $\G$-module and let $\G' \subset \G$. The following diagrams commute:
\[ \xymatrix{
	H^{-1}(\G,A)\ar[r]\ar[d]^{\tx{Res}}&A_\Gt\ar[d]^{N_{\G' \lmod \G}}&\\
	H^{-1}(\G',A)\ar[r]&A_{\G'\hm,\Tors}
}
\quad
\xymatrix{
	H^{-1}(\G,A)\ar[r]&A_\Gt&\\
	H^{-1}(\G',A)\ar[r]\ar[u]^{\tx{Cor}}&A_{\G'\hm,\Tors}\ar[u]
}
\]
where the right-hand vertical arrow in the rectangle at right is the natural projection.
\end{lemma}

\begin{proof}
One can check this directly using the explicit description of the restriction and corestriction maps given in \S\ref{sub:rescor}. We give here an alternative proof. Take a resolution $A^{-1} \to A^0$ of $A$ such that $A^0$ is a free $\Z[\G]$-module. The connecting homomorphism $H^{-1}(\G,A) \to H^0(\G,A^{-1})$ is then an isomorphism, and the claim follows from Lemmas \ref{lem:res0} and \ref{lem:connect-delta}, and the fact that the map \eqref{eq:connect-delta} is injective due to Theorem \ref{t:Hinich} and the freeness of the $\Z[\Gamma]$-module $A^0$.
\end{proof}

\begin{remark} \label{rem:ncomp}
It is useful to observe that, given $\G'' \subset \G' \subset \G$ we have $N_{\G'' \lmod \G'} \circ N_{\G' \lmod \G} = N_{\G'' \lmod \G}$. Indeed, note that giving a section $s\colon \G' \lmod \G \to \G$ is equivalent to giving a retraction $r\colon \G \to \G'$, that is, a set-theoretic map that satisfies $r(\gamma' \gamma)=\gamma' r(\gamma)$ for $\gamma' \in \G'$ and $\gamma \in \G$. The equivalence is given by $\gamma = r(\gamma) \cdot s(\gamma)$. Now, given $r\colon \G \to \G'$ and $r'\colon \G' \to \G''$, we can define $r''\colon \G \to \G''$ by $r'' = r' \circ r$. The associated section $s''\colon \G'' \lmod \G \to \G$ is then determined by $s''(\gamma)=s'(r(\gamma)) \cdot s(\gamma)$. One now checks at once the desired formula $N_{\G'' \lmod \G'} \circ N_{\G' \lmod \G} = N_{\G'' \lmod \G}$.
\end{remark}

\chapter[The Tate--Nakayama isomorphism]{The Tate--Nakayama isomorphism for bounded complexes of tori} \label{sec:tniso}

In this chapter $E/F$ is a finite Galois extension of local or global fields with Galois group $\G=\G\sub{E/F}$.
We will extend the famous Tate--Nakayama isomorphism defined originally in \cite{Tate66}
in the case of tori to the case of bounded complexes of tori.

\section{The local Tate class.}

Consider the case when $F$ is local. A fundamental role in class field theory is played by the \emph{fundamental class} $\alpha_{E/F} \in H^2(\G,E^\times)$. Its construction is a major achievement of the theory, and can be summarized in the following steps:
\begin{enumerate}
	\item If $E/F$ is unramified, then $H^i(\G,O_E^\times)=\{1\}$ for all $i$, and hence the normalized valuation map $E^\times \to \Z$ induces an isomorphism
$H^i(\G,E^\times) \to H^i(\G,\Z)$.
On the other hand, the exact sequence $0 \to \Z \to \Q \to \Q/\Z \to 0$ leads to an edge map $H^1(\G,\Q/\Z) \to H^2(\G,\Z)$ that is an isomorphism.
Finally, $H^1(\G,\Q/\Z)=\tx{Hom}(\G,\Q/\Z)$ has the canonical element sending the Frobenius element of $\G$ to $(\#\G)^{-1}$.
 	\item The constructions in (1) are compatible with inflation and provide an isomorphism $H^2(\tx{Gal}(F^\tx{ur}/F),(F^\tx{ur})^\times) \to \Q/\Z$, where $F^\tx{ur}$ is the maximal unramified extension of $F$ contained in a fixed separable extension $F^s$.
  	\item The inflation map $H^2(\tx{Gal}(F^\tx{ur}/F),(F^\tx{ur})^\times) \to H^2(\tx{Gal}(F^s/F),(F^s)^\times)$ is an isomorphism. Hence, for any finite Galois extension $E/F$, the subgroup $H^2(\Gamma,E^\times) \subset H^2(\tx{Gal}(F^s/F),(F^s)^\times)$ is cyclic of order $[E:F]$ with canonical generator called the fundamental class,
  which we denote by $\alpha_{E/F}$.
\end{enumerate}
We refer the reader to \cite[Chapter  VI]{CasFro86} for more details.

Given a tower $K/E/F$ of finite Galois extensions, we have
\[ \tx{res}(\alpha_{K/F})=\alpha_{K/E},\quad \tx{cor}(\alpha_{K/E})=[E:F]\hs\alpha_{K/F},\quad \tx{inf}(\alpha_{E/F})=[K:E]\hs\alpha_{K/F}.\]

\section{The global Tate class $\alpha_3$\hs.}

Consider now the case when $F$ is global.
Let $S_F\subseteq \V_F$ be a subset, finite or infinite.
Let $S_E$ denote the preimage of $S_F$ in $\V_E$.
Following  Tate \cite{Tate66}, we assume that $S_F$ satisfies the following conditions:

\begin{enumerate}
\item[\rm (i)] $S_F$ contains all archimedean places of $F$;

\item[\rm (ii)] $S_F$ contains all places ramified in $E$;

\item[\rm (iii)] $S_F$ is large enough with respect to $E$
so that every ideal class of $E$ contains an ideal with support in $S_E$.
\end{enumerate}

Observe that the set $S_F=\V_F$ of all places of $F$ satisfies conditions (i-iii).

Following \cite{Tate66}, we consider the short exact sequence of $\G$-modules
\begin{equation} \label{eq:es1}
0\to O_{E,S}^\times\to \A_{E,S}^\times\to C_E\to 0
\end{equation}
where we use the following notations:
\begin{itemize}
\item $O_{E,S}^\times\coloneqq O_{E,S_E}^\times$ is the group of invertible elements  of the ring $O_{E,S_E}$ of elements of $E$
     that are integral outside $S_E$.
\item $\A_{E,S}^\times\coloneqq\A_{E,S_E}^\times$ is the group of invertible elements
   of the ring $\A_{E,S_E}$ of elements of the ad\`ele ring $\A_E$ of $E$ that are integral outside $S_E$.
\item $C_E$ is the group $\A_{E,S}^\times/O_{E,S}^\times$ of $S$-id\`ele classes of $E$, which in view of condition (iii)
     is isomorphic to the group of all id\`ele classes $\A_E^\times/E^\times$ of $E$.
\end{itemize}
Writing $\Z[S_E]_0$ for the kernel of the augmentation map $\Z[S_E] \to \Z$, we have the exact sequence
\begin{equation} \label{eq:es2}
0 \to \Z[S_E]_0 \to \Z[S_E] \to \Z \to 0.
\end{equation}
We recall from \cite{Tate66} the construction of the three classes
\begin{align*}
\alpha_1 \in &\ H^2(\G, C_E)=H^2\big(\G, \tx{Hom}(\Z,\A_{E,S}^\times/O_{E,S}^\times)\big), \\
\alpha_2 \in &\ H^2\big(\G,\tx{Hom}(\Z[S_{E}],\A_{E,S}^\times)\big),\\
\alpha_3 \in &\ H^2\big(\G,\tx{Hom}(\Z[S_E]_0,O_{E,S}^\times)\big)
\end{align*}
where $\tx{Hom}$ denotes the group of abelian group homomorphisms.
The class $\alpha_1$ is simply the fundamental class for the global extension $E/F$, for whose construction we refer the reader to \cite[Chapter VII, \S11]{CasFro86}. The class $\alpha_2$ is a composite of the fundamental classes of all local extensions arising from $E/F$.
More precisely,
choosing a section $v\mapsto \bv$  of the projection $S_E\to S_F$, we have the isomorphism
\[ H^2\big(\G,\tx{Hom}(\Z[S_{E}],\A_{E,S}^\times)\big) \to \prod_{v \in S_F} H^2(\Gv,\A_{E,S}^\times), \]
which sends a class $\beta$ on the left to the collection $(\beta_{w})_{w}$ where $w=\bv$ and
$\beta_{w}$ is obtained from $\beta$ by restricting to $\G\sss{w}$ and then evaluating at $[w] \in \Z[S_E]$.
Then $\alpha_2$ is the class on the left that corresponds
to the collection $(\alpha_{2,w})_{w}$ where $w=\bv$ and  $\alpha_{2,w}$ is the image of the local fundamental class in $H^2(\G\sss{w},(E_{w})^\times)$ under the natural inclusion $(E_{w})^\times \to \A_{E,S}^\times$.

To construct the class $\alpha_3$, consider the group $X$ of homomorphisms from the exact sequence \eqref{eq:es2} to the exact sequence \eqref{eq:es1},
that is, the group of compatible triples of elements of
\[ \tx{Hom}(\Z[S_E]_0,O_{E,S}^\times),\ \tx{Hom}(\Z[S_E],\A_{E,S}^\times),\
\text{and}\  \tx{Hom}(\Z,\A_{E,S}^\times/O_{E,S}^\times).\]
For $i=1,2,3$ let $u_i$ be the map that selects the $(4-i)$th component of such a triple.
Thus we have the map
\[ (u_2,u_1)\colon X \to \tx{Hom}(\Z[S_E],\A_{E,S}^\times) \times \tx{Hom}(\Z,\A_{E,S}^\times/O_{E,S}^\times). \]
Tate proves that there exists a unique class $\alpha \in H^2(\G,X)$ whose image under $(u_2,u_1)$ is the tuple $(\alpha_2,\alpha_1)$. Then $\alpha_3$ is defined to be $u_3(\alpha)$.

From the definition the following properties are immediate.

\begin{fact} \label{fct:tate-123}
\begin{enumerate}
	\item[\rm(1)] The images of $\alpha_3$ and $\alpha_2$ in $H^2(\G,\tx{Hom}_\Z(\Z[S_E]_0,\A_{E,S}^\times))$ coincide, where we have used the natural inclusions $\Z[S_E]_0 \to \Z[S_E]$ and $O_{E,S}^\times \to \A_{E,S}^\times$.
	\item[\rm(2)] Let $w \in S_E$ and let $v \in S$ be its image. The image of $\alpha_2$  in the group $H^2(\G\sub{E/F,w},E_w^\times)$,
where we first restrict along the inclusion $\G\sub{E/F,w} \to \G\sub{E/F}$ and then project via $\A_{E,S}^\times \to E_w^\times$, equals $\alpha_{E_w/F_v}$.
\end{enumerate}
\end{fact}

\begin{lemma} \label{lem:tate-res}
Let $K/E/F$ be a tower of finite Galois extensions. The image of $\alpha_{3,K/F}$ under
\[ \tx{Res}\colon H^2(\G\sub{K/F},\tx{Hom}_\Z(\Z[S_K]_0,O_{K,S}^\times)) \to H^2(\G\sub{K/E},\tx{Hom}_\Z(\Z[S_K]_0,O_{K,S}^\times)) \]
is the Tate class $\alpha_{3,K/E}$.
\end{lemma}

\begin{proof}
By construction of $\alpha_3$ it is enough to prove that the restriction of $\alpha_{K/F}$ equals $\alpha_{K/E}$.
By construction of $\alpha$ it is enough to prove the analogous statement for $\alpha_1$ and $\alpha_2$.
For $\alpha_1$ this is well-known.
For $\alpha_2$ this follows from the corresponding statement for the local classes $\alpha_{2,w}$\hs, which is again well-known.
\end{proof}

\section{Definition of the isomorphism} \label{sub:tiso}

We continue with $E/F$ being a finite Galois extension of local or global fields with Galois group $\G$. In the local case, we let $\alpha_{E/F} \in H^2(\G,E^\times)$ be the fundamental class.
In the global case, we consider a set $S_F\subseteq \V_F$ satisfying conditions (i-iii) of the previous section, and let $\alpha_{3,E/F} \in \tx{Hom}(\Z[S_E]_0,O_{E,S}^\times)$ be the Tate class.

Let $M \in \tx{mod}_\tx{tf}\big(\G\sub{E/F}\big)$. When $E/F$ is local, we consider for $i \in \Z$ the homomorphism
\begin{equation*}
 - \cup \alpha_{E/F} :\,  H^i(\G\sub{E/F},M) \to H^{i+2}(\G\sub{E/F},M \otimes_\Z E^\times).
\end{equation*}
When $E/F$ is global, we consider for $i \in \Z$ the homomorphism
\[ - \cup \alpha_{3,E/F} :\, H^i(\G\sub{E/F},M \otimes_\Z \Z[S_E]_0) \to H^{i+2}(\G\sub{E/F},M \otimes_\Z O_{E,S}^\times). \]
In both cases, Tate \cite{Tate66} has shown that this homomorphism is an isomorphism. It is evidently functorial in $M$.

Tate's isomorphism extends from the case of
torsion-free finitely generated $\Z[\G\sub{E/F}]$-modules to bounded complexes of such. Abstractly this can be seen as follows.
We can present the isomorphism as an isomorphism between the two functors
from $\tx{mod}_\tx{tf}(\G\sub{E/F})$ to the category of abelian groups $\Ab$, defined by
\[ M \mapsto M \otimes_\Z \Z[S_E]_0 \mapsto H^i(\G\sub{E/F},M \otimes_\Z \Z[S_E]_0) \]
and
\[ M \mapsto M \otimes_\Z O_{E,S}^\times \mapsto H^{i+2}(\G\sub{E/F},M \otimes_\Z O_{E,S}^\times) \]
in the global case, and by
\[ M \mapsto M \otimes_\Z \Z \mapsto H^i(\G\sub{E/F},M \otimes_\Z \Z) \]
and
\[ M \mapsto M \otimes_\Z E^\times \mapsto H^{i+2}(\G\sub{E/F},M \otimes_\Z E^\times) \]
in the local case. We then observe that these functors descend along the essentially surjective functor $\tx{mod}_\tx{tf}(\G\sub{E/F}) \to \ul{\mc{D}}^{b,c}(\G\sub{E/F})$. For this, we use the symmetric monoidal structure on this category to note that we can perform the tensor products internally in $\ul{\mc{D}}^{b,c}(\G\sub{E/F})$.

From now on we will abbreviate $M \otimes_\Z \Z[S_E]_0$ by $M[S_E]_0$.

\section{Compatibility with restriction and corestriction}

Let $M$ be a bounded complex of finitely generated torsion-free $\Z[\G\sub{E/F}]$\-modules. Consider a tower of finite Galois extensions $K/E/F$.

\begin{lemma} \label{lem:tnrescor}
\begin{enumerate}
	\item Assume that $F$ is local. We have the commutative diagrams
	\[ \xymatrix@C=15mm{
		H^i(K/F,M)\ar[d]_-{\tx{Res}}\ar[r]^-{\cup \alpha_{K/F}}&H^{i+2}(K/F,M\otimes K^\times)\ar[d]^-{\tx{Res}}\\
		H^i(K/E,M)\ar[r]^-{\cup \alpha_{K/E}}&H^{i+2}(K/E,M\otimes K^\times) } \]
	and
	\[ \xymatrix@C=15mm{
		H^i(K/F,M)\ar[r]^-{\cup \alpha_{K/F}}&H^{i+2}(K/F,M\otimes K^\times)\\
		H^i(K/E,M)\ar[u]^-{\tx{Cor}}\ar[r]^-{\cup \alpha_{K/E}}&H^{i+2}(K/E,M\otimes K^\times)\ar[u]_-{\tx{Cor}} } \]
	
\item Assume that $F$ is global. We have the commutative diagrams
	\[ \xymatrix@C=15mm{
		H^i(K/F,M[S_{K}]_0)\ar[d]_-{\tx{Res}}\ar[r]^-{\cup \alpha_{3,K/F}}&H^{i+2}(K/F,M\otimes O_{K,S}^\times)\ar[d]^-{\tx{Res}}\\
		H^i(K/E,M[S_{K}]_0)\ar[r]^-{\cup \alpha_{3,K/E}}&H^{i+2}(K/E,M\otimes O_{K,S}^\times) } \]
	and
	\[ \xymatrix@C=15mm{
		H^i(K/F,M[S_{K}]_0)\ar[r]^-{\cup \alpha_{3,K/F}}&H^{i+2}(K/F,M\otimes O_{K,S}^\times)\\
		H^i(K/E,M[S_{K}]_0)\ar[u]^-{\tx{Cor}}\ar[r]^-{\cup \alpha_{3,K/E}}&H^{i+2}(K/E,M\otimes O_{K,S}^\times)\ar[u]_-{\tx{Cor}} } \]
\end{enumerate}
\end{lemma}

\begin{proof}
Since the horizontal isomorphisms are between functors $\ul{\mc{D}}^{b,c}(\G\sub{E/F}) \to \Ab$, we can use the presentation $\ul{\mc{D}}^{b,c}_1(\G\sub{E/F})$ and see that it is enough
to prove the statements when $M$ is concentrated in degree $0$, that is, $M$ is a $\Z[\G\sub{E/F}]$-module.

The local and global diagram for restriction are proved in the same way, namely via the formulas
$\tx{Res}(\lambda \cup \alpha_{K/F}) = \tx{Res}(\lambda) \cup \tx{Res}(\alpha_{K/F})$ and $\tx{Res}(\alpha_{K/F})=\alpha_{K/E}$,
where we are using $\alpha_{K/F}$ to denote either the local fundamental class or the global Tate class $\alpha_{3,K/F}$.
The first equality is a basic property of cup products (cf. \cite[Chapter IV, \S7, Proposition 9(iii)]{CasFro86}),
while the second equality is well-known in the local case, and is Lemma \ref{lem:tate-res} in the global case.

The commutativity of the  local and global diagrams for corestriction are also proved in the same way,
namely via the formula
\[ \tx{Cor}(\mu \cup \alpha_{K/E}) = \tx{Cor}\big(\mu \cup \tx{Res}(\alpha_{K/F})\big) =
   \tx{Cor}(\mu) \cup \alpha_{K/F}. \]
Here the second equality is   \cite[Chapter IV, \S7, Proposition 9(iv)]{CasFro86}.
\end{proof}

\section{Compatibility with localization}

Let $F$ be a global field, and $E/F$ a finite Galois extension with Galois group $\G\sub{E/F}$.
Let $w$ be a place of $E$, and $v$ the place of $F$ below $w$.
We define the map
\[ l^i_w \!:\, H^i(\G\sub{E/F},M[S_E]_0) \to H^i(\G\sub{E/F,w},M)\]
to be the composition of the restriction map
\[ H^i(\G\sub{E/F},M[S_E]_0) \to H^i(\G\sub{E/F,w},M[S_E]_0)\]
and the evaluation map $M[S_E]_0 \to M$
that extracts the $w$-component.
We further consider the usual localization map
\[ \tx{loc}^i_w \!: H^i(\G\sub{E/F},M \otimes_\Z O_{E,S}^\times) \to H^i(\G\sub{E/F,w},M \otimes E_w^\times)\]
given by composing the restriction map
\[ H^i(\G\sub{E/F},M \otimes_\Z O_{E,S}^\times) \to H^i(\G\sub{E/F,w},M \otimes_\Z O_{E,S}^\times)\]
with the natural inclusion $O_{E,S}^\times \to E_w^\times$.

\begin{proposition} \label{pro:tnfinloc}
Let $M \in \ul{\mc{D}}^{b,c}$ and $i \in \Z$. The following diagram commutes:
\[ \xymatrix@C=15mm{
	H^i(\G\sub{E/F},M[S_E]_0)\ar[r]^-{\cup \alpha_{3,E/F}}\ar[d]^{l^i_w}&H^{i+2}(\G\sub{E/F},M \otimes_\Z O_{E,S}^\times)\ar[d]^{\tx{loc}^i_w}\\
	H^i(\G\sub{E/F,w}\hs,M)\ar[r]^-{\cup \alpha_{E_w/F_v}}&H^{i+2}(\G\sub{E/F,w}\hs,M \otimes E_w^\times)
}\]
\end{proposition}

\begin{proof}
The inclusion $O_{E,S}^\times \to E_w^\times$ is the composition of the inclusion $O_{E,S}^\times \to \A_{E,S}^\times$,
the projection $\A_{E,S}^\times \to (E \otimes_F F_v)^\times$
and the projection $(E \otimes_F F_v)^\times \to E_w^\times$.
Let $S_{E}(v)$ denote the set of places of $E$ lying over $v$.
We claim that the following diagram commutes:
\[ \xymatrix@C=15mm{
H^i(\G\sub{E/F},M[S_E]_0)\ar[r]^-{\cup \alpha_{3,E/F}}\ar[d]&H^{i+2}(\G\sub{E/F},M \otimes_\Z O_{E,S}^\times)\ar[d]\\
H^i(\G\sub{E/F},M[S_E])\ar[r]^-{\cup \alpha_{2,E/F}}\ar[d]&H^{i+2}(\G\sub{E/F},M \otimes_\Z \A_{E,S}^\times)\ar[d]\\
H^i(\G\sub{E/F},M[S_{E}(v)])\ar[r]^-{\cup \alpha_{2,E/F,v}}\ar[d]&H^{i+2}(\G\sub{E/F},M \otimes_\Z (E \otimes_F F_v)^\times)\ar[d]\\
H^i(\G\sub{E/F,w}\hs,M)\ar[r]^-{\cup \alpha_{E_w/F_v}}&H^{i+2}(\G\sub{E/F,w}\hs,M \otimes_\Z E_w^\times).
}\]
The top left-hand vertical arrow comes from the inclusion $M[S_E]_0 \to M[S_E]$,
and the top right-hand vertical arrow comes from the inclusion $O_{E,S}^\times \to \A_{E,S}^\times$.
The commutativity of the top rectangle is due to Fact \ref{fct:tate-123}(1).

We have the decomposition of the $\G\sub{E/F}$-modules
$M[S_E]=\bigoplus_{v' \in S} M[S_{E}(v')]$ and $\A_{E,S}^\times=\prod_{v' \in S_F} (E \otimes_F F_{v'})^\times$.
By the construction of $\alpha_{2,E/F}$\hs, the isomorphism $\cup\alpha_{2,E/F}$ decomposes accordingly
as the product of the maps
\[ H^i(\G\sub{E/F},M[S_{E}(v')]) \to H^{i+2}(\G\sub{E/F},M \otimes_\Z (E \otimes_F F_{v'})^\times),\]
each of which is necessarily an isomorphism.
This establishes the commutativity of the middle rectangle,
in which the left-hand vertical arrow is the projection $M[S_E] \to M[S_{E}(v)]$,
the right-hand vertical arrow is the projection $\A_{E,S}^\times \to (E \otimes_F F_v)^\times$, and
\[ \alpha_{2,E/F,v} \in H^2\big(\G\sub{E/F},\tx{Hom}_\Z(\Z[S_{E}(v)],\A_{E,S}^\times)\big)\]
is the $v$-component of $\alpha_{2,E/F}$.

In the bottom rectangle, the vertical arrows are the Shapiro isomorphisms
(given by restriction to $\G\sub{E/F,w}$ followed by the map
induced by the homomorphism  $M[S_{E}(v)] \to M$ evaluating at $w$)  on the left
 and the homomorphism $(E \otimes_F F_v)^\times \to E_w^\times$ on the right.
The commutativity of the rectangle is due to Fact \ref{fct:tate-123}(2).
\end{proof}

\section{Compatibility with connecting homomorphisms}

Let $F$ be either local or global, and  $E/F$ be a finite Galois extension with Galois group $\G\sub{E/F}$.

\begin{proposition}
Let \, $M' \to M \to M''\to M'[1]$ \, be a distinguished triangle in $\ul{\mc{D}}^{b,c}(\G\sub{E/F})$. The following diagrams commute:
\[ \xymatrix@C=15mm{
	H^i(\G\sub{E/F},M''[S_E]_0)\ar[r]^-{\cup\alpha_{3,E/F}}\ar[d]&H^{i+2}(\G\sub{E/F},M'' \otimes_\Z O_{E,S}^\times)\ar[d]\\
	H^{i+1}(\G\sub{E/F},M'[S_E]_0)\ar[r]^-{-\cup\alpha_{3,E/F}}&H^{i+3}(\G\sub{E/F},M' \otimes_\Z O_{E,S}^\times)\\
}\]
when $F$ is global, and
\[ \xymatrix@C=16mm{
	H^i(\G\sub{E/F},M'')\ar[r]^-{\cup\alpha_{E/F}}\ar[d]&H^{i+2}(\G\sub{E/F},M'' \otimes_\Z E^\times)\ar[d]\\
	H^{i+1}(\G\sub{E/F},M')\ar[r]^-{-\cup\alpha_{E/F}}&H^{i+3}(\G\sub{E/F},M' \otimes_\Z E^\times)\\
}\]
when $F$ is local. Here the vertical arrows are the connecting homomorphisms.
\end{proposition}

\begin{proof}
This follows at once from the fact that cup products anticommute with the shift operator.
\end{proof}

\chapter{Absolute Galois cohomology in a special case}
\label{s:ab-coho}

In this chapter we will apply the results of the previous chapters in a very specific case.
Let $F$ be a local or global field, $F^s$ a separable closure of $F$, and  $\G=\tx{Gal}(F^s\hm/F)$.
Let $M^\bullet=(M^{-1} \to M^0)$ be a complex of torsion-free
finitely generated abelian groups equipped with a continuous $\G$-action
(with respect to the discrete topology on $M^{-1}$ and  $M^0$).
We assume that the map $M^{-1} \to M^0$ is injective.
Hence we have the quasi-isomorphism $M^\bullet \to M$ where $M=\tx{coker}(M^{-1} \to M^0)$.
Let $\Tbul=(T^{-1} \to T^0)$ be the associated complex of tori, that is, $X_*(T^\bullet)=M^\bullet$.
We are interested in computing the absolute Galois cohomology $H^i\big(\G,T^{-1}(F^s) \to T^0(F^s)\big)$ in degrees $i=1,2$.
Since this cohomology is defined as the injective limit
of the cohomologies $H^i\big(\G\sub{E/F},T^{-1}(E) \to T^0(E)\big)$
over all finite Galois extensions $E/F$, we can apply the Tate--Nakayama isomorphism
and reduce to computing the injective limit of $H^{i-2}(\G\sub{E/F},A)$, where $A=M$ when $F$ is local, and $A=M[V_E]_0$ when $F$ is global. In the computation of this injective limit, a key step is the identification of the transition maps.

\section{Helpful quasi-isomorphisms} \label{sub:help}

Let $E/F$ be a finite Galois extension such that the action of $\G$ on $M^{-1} \to M^0$ factors through $\G\sub{E/F}$.
In this section we will show that one can construct a complex $\tilde M^{-1} \to \tilde M^{0}$
that is quasi-isomorphic to $M^{-1} \to M^0$ and such that
in addition we can choose any one of $\tilde M^0$ or $\tilde M^{-1}$ to be a free $\Z[\G\sub{E/F}]$-module of finite rank.

If we wish $\tilde M^0$ to be such, the construction is very simple.
We choose $\tilde M^0$ to be finite rank free $\Z[\G\sub{E/F}]$-module
equipped with a surjection $\tilde M^0 \to M$,
and let $\tilde M^{-1}$ be the fiber product of $\tilde M^0 \to M^0 \from M^{-1}$.
The obvious map
\[ [\tilde M^{-1} \to \tilde M^{0}] \to [M^{-1} \to M^{0}]\]
is a quasi-isomorphism.

If we wish $\tilde M^{-1}$ to be a finite rank free $\Z[\G\sub{E/F}]$-module,
the construction is based on the previous one
and double dualization. For this we need the following lemma.

\begin{lemma} \label{lem:dual-q-iso}
Let $[M^{-1} \to M^0] \to [N^{-1} \to N^0]$ be a quasi-isomorphism
of complexes of finitely generated torsion-free $\G$-modules.
Then so is $\big[(N^0)^\vee \to (N^{-1})^\vee\big] \to \big[(M^0)^\vee \to (M^{-1})^\vee\big]$, where we recall that 
$(-)^\vee=\tx{Hom}_\Z((-),\Z)$.
\end{lemma}

\begin{proof}
Writing $K_M$ and $C_M$ for the kernel and cokernel of $M^{-1} \to M^0$,
we have the exact sequence $0 \to K_M \to M^{-1} \to M^0 \to C_M \to 0$.
With the similar notation for $N$, we are given that
the maps $K_M \to K_N$ and $C_M \to C_N$ are isomorphisms.
Writing $L_M$ and $D_M$ for the kernel and cokernel of $(M^0)^\vee \to (M^{-1})^\vee$
and using the similar notation for $N$, we want to check
that $L_N \to L_M$ and $D_N \to D_M$ are isomorphisms.

The exact sequence
\[ 0 \to C_M^\vee \to (M^0)^\vee \to (M^{-1}/K_M)^\vee \to \tx{Ext}^1_\Z(C_M,\Z) \to 0\]
identifies $L_M$ with $C_M^\vee$,
and we see that $L_N \to L_M$ is the dual of the isomorphism
$C_M \to C_N$, and hence itself an isomorphism.
At the same time, we have
\[ \tx{Ext}^1_\Z(C_M,\Z)\cong\tx{Hom}_\Z\big((C_M)_{\Tors},\Q/\Z\big)\]
(see formula \eqref{e:Ext-Hom} in Appendix \ref{app:duality}\hs),
whence $\tx{Ext}^1_\Z(C_M,\Z)$ is finite.
Noting that $M^{-1}/K_M$ is torsion-free,
from the above exact sequence we obtain the short exact sequence
\[ 0 \to \tx{Hom}_\Z((C_M)_{\Tors}\hs,\Q/\Z) \to D_M \to K_M^\vee \to 0.\]
The map $D_N \to D_M$ induces the maps $K_N^\vee \to K_M^\vee$ and
$\tx{Hom}_\Z((C_N)_{\Tors},\Q/\Z) \to \tx{Hom}_\Z((C_M)_{\Tors},\Q/\Z)$
obtained by duality from the isomorphisms $K_M \to K_N$ and $C_M \to C_N$.
We conclude that $D_N \to D_M$ is also an isomorphism.
\end{proof}

With this in mind, we apply the above construction to the complex $(M^0)^\vee \to (M^{-1})^\vee$ to obtain a quasi-isomorphic complex $N^0 \to N^1$ with $N^1$ a finite rank free $\Z[\G\sub{E/F}]$-module. By Lemma \ref{lem:dual-q-iso} the complex $(N^1)^\vee \to (N^0)^\vee$ is quasi-isomorphic to $[M^{-1} \to M^0]$. By Fact \ref{fct:fsd} the $\Z[\G\sub{E/F}]$-module $(N^1)^\vee$ is free of finite rank.

For another construction of a desired complex $$\tilde M^{-1} \to \tilde M^{0}$$
with $\tilde M^{-1}$ being $\Z[\G\sub{E/F}]$-free,
see Milne and Shih, Chapter V of \cite{DMOS82}, Lemma 3.2 on page 297.

\section{The local case}

Consider the case when $F$ is local. We wish to compute the  injective limit of $H^{i-2}(\G\sub{E/F},M)$ for $i=1,2$. If $F$ is archimedean, then $F^s/F$ is finite, and there is nothing to do. For completeness, we record the result nonetheless:

\begin{proposition}[\hs{\cite[Proposition 8.21]{BorTim24}}\hs]
	\label{p:R}
	Let $F=\R$. There are functorial isomorphisms
	\[ H^{n-2}(\G\sub{\C/\R}, M)\isoto H^n(\R,T^{-1}\to T^0)\qquad\text{for}\ \, n\in\Z.\]
	\end{proposition}

For the rest of this section, we assume that $F$ is non-archimedean.
Consider a tower $K/E/F$ of finite Galois extensions
and assume that the action of $\G$ on $M^{-1} \to M^0$
factors through $\G\sub{E/F}$.
For an arbitrary $\Z[\G\sub{E/F}]$-module $A$ we define maps
\[ ! \colon H^{-1}(E/F,A) \to H^{-1}(K/F,A)\]
and
\[ p \colon H^0(E/F,A) \to H^0(K/F,A)\]
that are functorial in $A$, as follows.
The map $!$ is given by the inclusion
\[ \tx{ker}[N_{E/F}\colon A \to A] \subset \tx{ker}[N_{K/F}\colon A \to A]\]
and the identification $A_{\G\sub{E/F}} = A_{\G\sub{K/F}}$.
The map $p$ is given by multiplication by $[K:E]$ on $A$
(which is equal to the action of the norm map $N_{K/E}$ on $A$
and the identification $A^{\G\sub{E/F}}=A^{\G\sub{K/F}}$.

\begin{lemma} \label{lem:!ploc}
The following diagram commutes:
\[ \xymatrix{
H^{-1}(E/F,M)\ar[d]^{!}\ar[r]&H^0(E/F,M^{-1})\ar[d]^p\\
H^{-1}(K/F,M)\ar[r]&H^0(K/F,M^{-1})
}\]
where the horizontal arrows are the connecting homomorphisms.
\end{lemma}

\begin{proof}
Direct computation using the standard resolution.
\end{proof}

\begin{lemma} \label{lem:tnfinloc}
The following diagrams commute:
\begin{align*}
\begin{aligned}
\xymatrix@C=15mm{
	H^{-1}(E/F,M)\ar[d]^{!}\ar[r]^-{\cup \alpha_{E/F}}&H^1(E/F,T^{-1} \to T^0)\ar[d]^{\tx{Inf}}\\
	H^{-1}(K/F,M)\ar[r]^-{\cup \alpha_{K/F}}&H^1(K/F,T^{-1} \to T^0)
}
\end{aligned}\tag{1}
\\
 \notag  \\
\begin{aligned}
 \xymatrix@C=15mm{
	H^0(E/F,M)\ar[d]^p\ar[r]^-{\cup \alpha_{E/F}}&H^2(E/F,T^{-1} \to T^0)\ar[d]^{\tx{Inf}}\\
	H^0(K/F,M)\ar[r]^-{\cup \alpha_{K/F}}&H^2(K/F,T^{-1} \to T^0)
}
\end{aligned}\tag{2}
\end{align*}
\end{lemma}

\begin{proof}
(1) According to Lemma \ref{lem:!ploc} and the functoriality of $\tx{Inf}$,
the diagram in (1) for $M$ maps to the diagram in (2) for $M^{-1}$.
If we knew that the connecting homomorphism $H^1(K/F,T^{-1} \to T^0) \to H^2(K/F,T^{-1})$ is injective,
the commutativity of (1) would follow from the commutativity of (2) applied to $M^{-1}$.
But we can place ourselves in the situation where this connecting homomorphism is injective
by replacing $[M^{-1} \to M^0]$ with a quasi-isomorphic complex
$[\tilde M^{-1} \to \tilde M^0]$ with $\tilde M^0$ a free $\Z[\G\sub{E/F}]$-module,
using the construction described in \S\ref{sub:help}.
In other words, we may assume that $M^0$ is such, from which it follows that $T^0$ is an induced torus,
and hence that $H^1(K/F,T^0)$ vanishes
(see \cite[Lemma 1.9]{Sansuc81}\hs), which implies the desired injectivity.

(2) Assume first that $M$ is torsion-free.
This is equivalent to $T^{-1} \to T^0$ being injective, in which case this complex is quasi-isomorphic
to the quotient $T=T^0/T^{-1}$, and $M=X_*(T)$.
In this setting the cup product of $\lambda \in M^{\G\sub{E/F}}$,
representing a class of $H^0(\G\sub{E/F},M)$, with $\alpha_{E/F} \in H^2(\G\sub{E/F},E^\times)$,
is given by the image of $\alpha_{E/F}$ under the map $H^2(\G\sub{E/F},E^\times) \to H^2(\G,T(E))$
induced by $\lambda \colon\mb{G}_m \to T$.
The claim then follows from the equality $\tx{Inf}(\alpha_{E/F}) = [K:E] \cdot \alpha_{K/F}$.
This completes the proof of (2) in the case when $M$ is torsion-free, and hence also the proof of (1).

The proof of (2) in the general case can be reduced to the case of $M$ torsion-free
if we can find a quasi-isomorphism $[M^{\prime\,-1} \to M^{\prime\,0}] \sim [M^{-1} \to M^0]$
for which $H^0(E/F,M^{\prime\,0}) \to H^0(E/F,M)$ is surjective, because then the diagram (2)
for $M^{\prime\,0}$ maps onto the diagram  (2) for $M$,
and the surjectivity of $H^0(E/F,M^{\prime\,0}) \to H^0(E/F,M)$
and the commutativity of the diagram for $M^{\prime\,0}$ implies the commutativity for $M$.
The surjectivity of $H^0(E/F,M^{\prime\,0}) \to H^0(E/F,M)$ would be implied
by the vanishing of $H^1(E/F,M^{\prime\,-1})$, so it is enough to find $[M^{\prime\,-1} \to M^{\prime\,0}]$ with $M^{\prime\,-1}$
a free $\Z[\G\sub{E/F}]$-module of finite rank.
The existence of such a complex $[M^{\prime\,-1} \to M^{\prime\,0}]$ was proved in \S\ref{sub:help}.
\end{proof}

The rest of this section is devoted to computing the injective limits of the groups $H^{-1}(K/F,M)$ and $H^0(K/F,M)$ along the maps $!$ and $p$, respectively, where $K$ runs over the finite Galois extensions of $F$ in $F^s$.

\begin{lemma} \label{lem:h-1loc}
The family of natural maps $H^{-1}(K/F,M) \to M_{\G}$ induces an isomorphism
\[ \varinjlim_K H^{-1}(K/F,M) \isoto (M_{\G})_{\Tors}. \]
\end{lemma}

\begin{proof}
The map $H^{-1}(K/F,M) \to M_{\G}$ is injective by construction,
and its image  lies in $(M_\G)_{\Tors}$\hs, since $H^{-1}(K/F,M)$ is $[K:F]$-torsion.
This map is also compatible with the transition maps $!$,
so we only need to check that every element of $(M_\G)_{\Tors}$
is in the image of $H^{-1}(K/F,M)$ for some $K$.
But this is the case as soon as  the order of $(M_\G)_{\Tors}$  divides $[K:E]$,
because
\[N_{K/F}=N_{E/F}\circ N_{K/E} = N_{E/F} \cdot [K:E].\qedhere\]
\end{proof}

\begin{lemma}\label{l:H0H0}
\ \\[-15pt] \label{lem:h0loc}
\begin{enumerate}
	\item[\rm(1)] The natural map $\varinjlim_K H^0(K/F,M) \to \varinjlim_K H^0(K/F,\Mtf)$ is an isomorphism.
	\item[\rm(2)] For every finite Galois extension $K/F$ containing $E$,
the map
\[ [K:F]^{-1}N_{K/F} \colon \Mtf\to (\Mtf\otimes \Q)^\G \]
induces an isomorphism
\[ N^\natural\colon (\Mtf\otimes\Q)_\G \to (\Mtf\otimes\Q)^\G, \]
and this isomorphism is independent of $K$.
	\item[\rm(3)] The family of injective maps
\begin{align*}
[K:F]^{-1}\colon\, &H^0(K/F,\Mtf)=(\Mtf)^\G/N_{K/F}(\Mtf)\, \into\, (\Mtf\otimes\Q)^\G/N^\natural(\Mtf),\\
                                 &x+N_{K/F}(\Mtf)\,\mapsto\, [K:F]^{-1}x+N^\natural(\Mtf)\quad\,\text{for}\ \,x\in M^\G
\end{align*}
induces an isomorphism
\begin{equation}\label{e:(3)}
\varinjlim_K H^0(K/F,\Mtf)\hs \isoto\hs (M_\tf\otimes\Q)^\G/N^\natural(\Mtf).
\end{equation}
\end{enumerate}
\end{lemma}

\begin{proof}
(1) Consider the commutative diagram with exact columns
\[ \xymatrix@R=5.5mm{
	H^0(E/F,M_{\Tors})\ar[r]^p\ar[d]&H^0(K/F,M_{\Tors})\ar[d]\\
	H^0(E/F,M)\ar[r]^p\ar[d]&H^0(K/F,M)\ar[d]\\
	H^0(E/F,\Mtf)\ar[r]^p\ar[d]&H^0(K/F,\Mtf)\ar[d]\\
	H^1(E/F,M_{\Tors})\ar[r]&H^1(K/F,M_{\Tors}).
}\]	
In this diagram, the bottom horizontal arrow is given by multiplication by $[K:E]$ on $M_\tx{Tors}$\hs, followed by inflation. The other three horizontal arrows are given by the map $p$ defined above.
The groups $H^i(E/F,M_{\Tors})$ are killed by multiplication by $|M_{\Tors}|$.
Taking $K$ large enough so that $|M_{\Tors}|$ divides $[K:E]$,
we see that the top and bottom horizontal maps vanish, hence the claim.

Alternatively, let
\[ 0\to M^{-1}\labelt\vk M^0\to M\to 0\]
be a resolution of $M$
where $M^{-1}$ and $M^0$ are torsion-free finitely generated (over $\Z$) $\G(F)$-modules.
Let $T^{-1}$ and $T^0$ be the $F$-tori with cocharacter groups $M^{-1}$ and $M^0$, respectively.
Consider the induced homomorphism $\vk_T\colon T^{-1}\to T^0$,
and put $T^\bullet=(T^{-1}\to T^0)$,  $T=\coker \vk_T$, and $\KK=\ker\vk_T$;
then $T$ is an $F$-torus with cocharacter group $\Mtf$\hs, and $\KK$ is a finite
$F$-group of multiplicative type (not necessarily smooth)
with character group $\tx{Hom}(M_{\Tors},\Q/\Z)$.
We have a short exact sequence of complexes
\begin{equation*}
 1\to (T^{-1}\to \im\vk_T)\to T^\bullet \to (1\to T)\to 1.
\end{equation*}
Note that we may identify $H^n(F,1\to T)=H^n(F,T)$ and
$H^n(F, T^{-1}\to \im\vk_T)=H^{n+1}(F,\KK)$,
where we write $H^n(F,\KK)$ for $H_{\rm fppf}^n(F,\KK)$.
Since fppf, \'etale, and Galois cohomology agree for tori,
see \cite[Proposition 71, Theorem 43]{Shatz},
the short exact sequence above gives rise to a cohomology exact sequence
\[H^3(F,\KK)\to H^2(F,T^\bullet)\to H^2(F,T)\to H^4(F,\KK)\]
Since $F$ is a non-archimedean local field, we have $H^3(F,\KK)=1$ and $H^4(F,\KK)=1$;
see  \cite[Proposition 3.1.2]{Rosengarten-Mem}.
We obtain an isomorphism $H^2(F,T^\bullet)\to H^2(F,T)$, which gives (1).

(2) is  immediate.

(3) The homomorphism \eqref{e:(3)} is well-defined and injective.
We have $(\Mtf\otimes\Q)^\G=(\Mtf)^\G\otimes \Q$, so each element of its target
is represented by $y \otimes q$ for some $y \in (\Mtf)^\G$ and  $q\in\Q$.
Choose $K$ so that $[K:F]\cdot q \in \Z$.
Then the element represented by $y \otimes q$ is in the image of the map of (3) for this $K$.
Thus \eqref{e:(3)} is surjective, as required.
\end{proof}

\begin{corollary} \label{cor:h0loc}
The composition
\begin{multline*}
 \varinjlim_K H^0(K/F,M) \to \varinjlim_K H^0(K/F,\Mtf) \to (\Mtf\otimes\Q)^\G/N^\natural(\Mtf)\\
\labeltoo{\,(N^\natural)^{-1}} \big((\Mtf\otimes\Q)/\Mtf\big)_\G= (\Mtf \otimes \Q/\Z)_\G = (M\otimes\Q/\Z)_\G
\end{multline*}
is an isomorphism.
\end{corollary}

\begin{proof}
The first, second, and third maps are isomorphisms by Lemma \ref{lem:h0loc} and Fact \ref{fct:372}.
The final two equalities follow from Facts \ref{fct:Q} and \ref{fct:Q/Z}.
\end{proof}

\begin{theorem}[\hs\cite{Brv98}, Proposition 4.1]
\label{thm:tnloc}
For a non-archimedean local field $F$, there are functorial isomorphisms
\begin{align}
&(M_\G)_{\Tors} \isoto H^1(F,T^{-1} \to T^0),\tag{1}\\
&(M\otimes\Q/\Z)_\G \isoto H^2(F,T^{-1} \to T^0),\tag{2}\\
&H^n(F,T^{-1} \to T^0)=\{0\}\ \, \text{for}\ \, n\ge 3.\tag{3}
\end{align}
\end{theorem}

\begin{proof}
We obtain (1) from Lemmas \ref{lem:tnfinloc}(1) and \ref{lem:h-1loc}.
We obtain (2) from Lemma \ref{lem:tnfinloc}(2) and Corollary \ref{cor:h0loc}.
We obtain (3) from the fact that
the strict cohomological dimension of $F$ is 2, cf. \cite[Corollary 7.2.5]{NSW08}.
\end{proof}

\section{The global case}
\label{ss:3-global}

Consider now the case when $F$ is global.
Recall that $M=\coker[M^{-1}\to M]$.
Let $K/E/F$ be a tower of finite Galois extensions
and assume that the $\G$-action on $M^{-1} \to M^0$ factors
through $\G\sub{E/F}$ \and that $E$ has no real places.
Let $\V_F$ denote the set of all places of $F$.
We will apply results of Chapter \ref{sec:tniso} in the case $S_F=\V_F$; then $S_E=\V_E$.

For an arbitrary $\Z[\G\sub{E/F}]$-module $A$, we follow \S\ref{ss:E/F} and write $\Z[V_E]$
for the free abelian group on the set of places $V_E$ of $E$
equipped with the $\Gamma$-action induced from the $\Gamma$-action on $V_E$.
We write $A[V_E]=A\otimes_\Z\Z[V_E]$, and we write  $A[V_E]_0$
for the kernel of the augmentation map $A[V_E] \to A$ sending $\sum a_w[w] \mapsto \sum a_w$.

We now define three maps:\vskip-20pt
\begin{gather*}
    \beta\colon A[\V_K] \to A[\V_E],
\\
     p\colon H^0(E/F,A[\V_E]_0) \to H^0(K/F,A[\V_K]_0),
\\
     ! \colon (A[\V_E]_0)_{\G\sub{E/F}} \to (A[\V_K]_0)_{\G\sub{K/F}}.
\end{gather*}\vskip-5pt

The map $\beta$ is the tensor product of the identity on $A$ and the natural projection $\V_K \to \V_E$. It is $\G\sub{K/F}$-equivariant.

The map $p$ is defined in the special case of the trivial module $A=\Z$ as
\[ p\colon \Z[\V_E] \to \Z[\V_K],\qquad [w]\hs \mapsto \sum_{u|w} |\G\sub{K/E,u}|\cdot[u]. \]
For general $A$ it is given by tensoring this special case with the identity on $A$.

To check that $p$ induces a map on Tate cohomology, and to describe the map $!$,
we fix a section $s \colon \V_E \to \V_K$ and let $s! \colon A[\V_E]_0 \to A[\V_K]_0$ be the tensor product
of the identity on $A$ and the linear extension $\Z[\V_E] \to \Z[\V_K]$ of $s$.
One checks that $p$ is $\Gamma_{K/F}$-equivariant, that $\epsilon_{K/F} \circ p = [K:E] \cdot \epsilon_{E/F}$
where $\epsilon_{K/F} : \Z[V_K] \to \Z$ and $\epsilon_{E/F} : \Z[V_E] \to \Z$ are the augmentation maps,
and that $p \circ N_{E/F} = N_{K/F} \circ s!$.
This implies that $p$ induces a map $H^0(E/F,A[\V_E]_0) \to H^0(K/F,A[\V_K]_0)$ as desired. The following lemma describes the map $!$.

\begin{lemma} \label{lem:!betaglob}
The map $s!$ induces a homomorphism
$$(A[\V_E]_0)_{\G\sub{E/F}} \to (A[\V_K]_0)_{\G\sub{K/F}}\hs,$$
which is independent of the choice of the section $s$ and will be denoted by $!$.
It is an isomorphism whose inverse is given by $\beta$.
Moreover, $!$ restricts to a homomorphism
\[ H^{-1}(\G(E/F),A[\V_E]_0) \to H^{-1}(\G(K/F),A[\V_K]_0).\]
\end{lemma}

\begin{proof}
In the proof of \cite[Lemma 3.1.7]{KalGRI} it is shown that the map $s!$ induces a well-defined homomorphism $! \colon (A[\V_E]_0)_{\G\sub{E/F}} \to (A[\V_K]_0)_{\G\sub{K/F}}$ which is independent of the choice of a section $s$ and furthermore restricts to a homomorphism $H^{-1}(\G(E/F),A[\V_E]_0) \to H^{-1}(\G(K/F),A[\V_K]_0)$. The fact that the composition $\beta\circ !$ is the identity on $(A[\V_E]_0)_{\G\sub{E/F}}$ is immediate from the definition. That $! \circ \beta$ is the identity on $(A[\V_K]_0)_{\G\sub{K/F}}$ follows from the argument in the first paragraph of the proof of \cite[Lemma 3.1.7]{KalGRI}.
\end{proof}

\begin{fact} \label{fct:p!n}
The map $p$ is equal to the composition of the map $s!$ and the norm map for the action of $\G\sub{K/E}$. Consequently, the identity $p \circ N_{E/F} = N_{K/F} \circ !$ holds.
\end{fact}

\begin{lemma} \label{lem:!pglob}
For $M$ as in the beginning of this section,
the following diagram commutes:
\[ \xymatrix{
H^{-1}(E/F,M[\V_E]_0)\ar[d]^{!}\ar[r]^-\delta  &H^0(E/F,M^{-1}[\V_E]_0)\ar[d]^p\\
H^{-1}(K/F,M[\V_K]_0)\ar[r]^-{\delta'}    &H^0(K/F,M^{-1}[\V_K]_0),
}\]
where the horizontal arrows are the connecting homomorphisms.
\end{lemma}

\begin{proof}
As in the local case, this is a direct computation using the standard resolution,
albeit a tad longer. For completeness, we include this computation.
The connecting homomorphisms are given by the norm maps
for the action of $\G\sub{E/F}$ (resp. $\G\sub{K/F}$\hs) on $M^0[V_E]_0$ (resp. $M^0[V_K]_0$\hs).
Consider an element $\sum_w m_w[w] \in M[\V_E]_0$ representing a class in $H^{-1}(\G,M[\V_E]_0)$.
For each $m_w \in M$ we choose a lift $m_w^0 \in M^0$.
We abbreviate $\Delta=\tx{Gal}(K/E)$, $\G=\tx{Gal}(E/F)$, $\G'=\tx{Gal}(K/F)$.
Then
\begin{eqnarray*}
\delta's_!\Big(\sum_w m_w[w]\Big)&=&\delta'\Big(\sum_w m_w [s(w)]\Big)\\
&=&\sum_{\sigma' \in \G'}\sigma'\Big(\sum_w m^0_w [s(w)]\Big)\\
&=&\sum_{\sigma \in \G} \sum_{\tau \in \Delta} \sum_w \sigma(m^0_w) [\sigma\tau s(w)]\\
&=&\sum_w  \sum_{\sigma \in \G} \sigma(m^0_{\sigma^{-1}w}) \sum_{w'|w} |\Delta_{w'}|[w']\\
&=&\sum_{w' \in \V_{K}}\sum_{\sigma \in \G} \sigma(m^0_{\sigma^{-1}p(w')}) |\Delta_{w'}|[w']\\
&=&p_{K/E}\Big( \sum_w \sum_{\sigma\in\G}\sigma(m^0_{\sigma^{-1}w})[w] \Big)\\
&=&p_{K/E}\Big( \sum_{\sigma\in\G} \sigma \sum_w m^0_w[w] \Big)\\
&=&p_{K/E}\delta\Big(\sum_w m_w[w]\Big).
\end{eqnarray*}
\end{proof}

\begin{lemma} \label{lem:tnfinglob}
The following diagrams commute:
\begin{equation}\tag{1}
\begin{aligned}
 \xymatrix@C=15mm{
	H^{-1}(E/F,M[\V_E]_0)\ar[d]^{!}\ar[r]^-{\cup \alpha_{3,E/F}}&H^1(E/F,T^{-1} \to T^0)\ar[d]^{\tx{Inf}}\\
	H^{-1}(K/F,M[\V_K]_0)\ar[r]^-{\cup \alpha_{3,K/F}}&H^1(K/F,T^{-1} \to T^0)\\
}\end{aligned}
\end{equation}

\begin{equation}\tag{2}
\begin{aligned}
\xymatrix@C=15mm{
	H^0(E/F,M[\V_E]_0)\ar[d]^p\ar[r]^-{\cup \alpha_{3,E/F}}&H^2(E/F,T^{-1} \to T^0)\ar[d]^{\tx{Inf}}\\
	H^0(K/F,M[\V_K]_0)\ar[r]^-{\cup \alpha_{3,K/F}}&H^2(K/F,T^{-1} \to T^0).\\
}\end{aligned}
\end{equation}

\end{lemma}

\begin{proof}
The proof proceeds as in the local case, that is, as the proof of Lemma \ref{lem:tnfinloc}, but with Lemma \ref{lem:!ploc} replaced by Lemma \ref{lem:!pglob}. Upon replacing $[M^{-1} \to M^0]$ with a quasi-isomorphic complex $[\tilde M^{-1} \to \tilde M^0]$ with $\tilde M^0$ a finite rank free $\Z[\G\sub{E/F}]$-module, the proof of (1) is reduced to the proof of (2) for $M$ torsion-free.

On the other hand, the proof of (2) in the general case can also be reduced to the proof of (2) for $M$ torsion-free, by replacing $[M^{-1} \to M^0]$ with a quasi-isomorphic complex $[M^{\prime\,-1} \to M^{\prime\,0}]$ with $M^{\prime\,-1}$ a finite rank free $\Z[\G\sub{E/F}]$-module using the construction of \S\ref{sub:help}. Then $M^{\prime\,-1}[\V_E]_0$ is an induced $\Z[\G\sub{E/F}]$-module and hence $H^1(E/F,M^{\prime\,-1}[\V_E]_0)$ vanishes.

Finally, the validity of (2) when $M$ is torsion-free is \cite[Lemma 3.1.4]{KalGRI}.
\end{proof}

The rest of this section is devoted to computing the injective limits of the systems $H^{-1}(K/F,M[V_K]_0)$ and $H^0(K/F,M[V_K]_0)$ along the maps $!$ and $p$, respectively.

Given a $\G\sub{E/F}$-module $A$, let us write $(A[\Vbar]_0)_\G$ for the limit of the system of groups $(A[\V_K]_0)_{\G\sub{K/F}}$ with transition maps given by $\beta$, equivalently its inverse $!$. Since these transition maps are isomorphisms, we can identify $(A[\Vbar]_0)_\G$ with any one of the terms $(A[\V_K]_0)_{\G\sub{K/F}}$ in that system.

Next we define a ``localization map'' $l_w \colon A[\V_E]_{\G\sub{E/F}} \to A_{\G\sub{E_w/F_v}}$
for any $w \in \V_E$ with image $v \in V_F$, by
\[ l_w\Big(\sum_{w'\in V_E} a_{w'}[w']\Big) = \sum_\sigma \sigma(a_{\sigma^{-1}w}), \]
where the right-hand sum runs over a set of representatives $\sigma$ in $\G\sub{E/F}$
for the coset space $\G\sub{E_w/F_v} \lmod \G\sub{E/F}$.
We can also express $l_w$ as the composition of
$$N_{\G\sub{E_w/F_v} \lmod \G\sub{E/F}} :\, A[\V_E]_{\G\sub{E/F}} \to A[\V_E]_{\G\sub{E_w/F_v}}$$
and the map $A[\V_E]_{\G\sub{E_w/F_v}} \to A_{\G\sub{E_w/F_v}}$ induced by the
 evaluation-at-$w$ map $A[\V_E] \to A$.

\begin{fact} \label{fct:locfin}
Let $u \in \V_K$ lie above $w \in \V_E$. Then the following diagram commutes:
\[ \xymatrix{
	A[\V_E]_{\G\sub{E/F}}\ar[r]^{l_w}\ar[d]^{!}&A_{\G\sub{E_w/F_v}}\ar@{=}[d]\\
	A[\V_K]_{\G\sub{K/F}}\ar[r]^{l_u}&A_{\G\sub{K_u/F_v}}
}\]
\end{fact}

For a place $\vbar $ of $F^s$ over a place $v$ of $F$, we obtain from Fact \ref{fct:locfin} a localization map
\begin{equation} \label{eq:loc}
l_{\vbar} \colon (A[\Vbar]_0)_\G \to A_{\G\sss{\vbar}}.	
\end{equation}

\begin{definition}\label{d:F1}
Let $\mc{F}^1(A/F)$, or just $\mc{F}^1(A)$ when $F$ is understood, be the fiber product of
\[ (A[\Vbar]_0)_{\G,{\Tors}} \lra \prod_{v|\infty}A_{\G\sss{\vbar},{\Tors}} \lla \prod_{v |\infty} H^{-1}(\G\sss{\vbar},A), \]
where the arrow at left is the product of $l_{\vbar}$ and the arrow at right is that of Fact \ref{fct:h-1}.
\end{definition}

\begin{remark}\label{r:injective}
Since the arrow at right in Definition \ref{d:F1} is injective by Fact \ref{fct:h-1}, so is the homomorphism
$\mc{F}^1(A) \to (A[\Vbar]_0)_{\G,{\Tors}}$\hs,
and we identify $\mc{F}^1(A)$ with its image in $(A[\Vbar]_0)_{\G,{\Tors}}$\hs.
When $A$ is torsion-free,  the arrow at right is an isomorphism by Fact \ref{fct:h-1},
and therefore  the inclusion $\mc{F}^1(A) \subset (A[\Vbar]_0)_{\G,{\Tors}}$  is in fact an equality, that is, we have
$ \mc{F}^1(A) = (A[\Vbar]_0)_{\G,{\Tors}}$\hs. Moreover, when $F$ is a function field, we have $ \mc{F}^1(A) = (A[\Vbar]_0)_{\G,{\Tors}}$\hs.
\end{remark}

\begin{lemma} \label{lem:h-1glob}
For $M$ as in the beginning of this section,
the family of natural maps $H^{-1}(K/F,M[\V_K]_0) \to (M[\V_K]_0)_\G$ induces an isomorphism
\begin{equation}\label{e:lim^-1}
\varinjlim_K H^{-1}(K/F,M[\V_K]_0) \isoto \mc{F}^1(M)
\end{equation}
where the injective limit is taken along the maps $!$\hs.
\end{lemma}

\begin{proof}
It follows from Lemma \ref{lem:!betaglob} that we have a well-defined injective map
$H^{-1}(K/F,M[\V_K]_0) \to (M[\Vbar]_0)_{\G}$.
One checks directly that its image lands in $\mc{F}^1(M)$. We must prove that the constructed injective map \eqref{e:lim^-1} is surjective.
Consider an element $x=\sum_w m_w\cdot w \in M[\V_E]_0$ that lies in $\mc{F}^1(M)$.
We want to find $K$ such that $!(x) \in M[\V_K]_0$ is killed by $N_{K/F}$.
This is equivalent to $\sum_{\sigma \in \G\sub{K/F}}\sigma(m_{\sigma^{-1}w})=0$ for all $w \in \V_E$.
Since $x$ is supported at only finitely many $w$, we can choose $K$
so that at each such $w$ that is non-archimedean, the number $|\G\sub{K/E,w}|$
is a multiple of the order of $[x]\in(M[\Vbar]_0)_{\G,{\Tors}}$.
The desired identity $\sum_{\sigma \in \G\sub{K/F}}\sigma(m_{\sigma^{-1}w})=0$ then holds for all non-archimedean $w$.
Since the image of $x$ in $\prod_{v|\infty}M_{\G\sss{\vbar}}$
is contained in the image of the arrow at right in Definition \ref{d:F1}, this identity  holds also for all archimedean (complex) places $w$ of $E$
(recall that by assumption $E$ has no real places).
\end{proof}

Note that for a $\G(E/F)$-module $A$,
the evaluation-at-$w$ map $\tx{ev}_w \colon A[\V_E] \to A$ satisfies
$\sigma \circ \tx{ev}_w \circ \sigma^{-1}=\tx{ev}_{\sigma w}$ for $\sigma \in \G\sub{E/F}$
and hence descends to a map $\tx{ev}_w \colon H^0(E/F,A[\V_E]) \to H^0(E_w/F_v,A)$.

Recall from \eqref{eq:xi} the homomorphism
$\xi_\G \colon H^0(\Gamma,A) \to (A\otimes\Q/\Z)_\G$ for any $\Gamma$-module $A$.
When $\Gamma=\Gamma(E/F)$ we will write $\xi_{E/F}$ in place of $\xi_{\Gamma(E/F)}$.

\begin{fact} \label{fct:evloc}
The following diagram commutes:
\[ \xymatrix@C=13mm{
H^0(E/F,M[\V_E])\ar[r]^{\xi_{E/F}}\ar[d]^{\tx{ev}_w}&(M[\V_E]_0 \otimes \Q/\Z)_{\G\sub{E/F}}\ar[d]^{l_w}\\
H^0(E_w/F_v,M)\ar[r]^{\xi_{E_w/F_v}}&(M\otimes\Q/\Z)_{\G\sub{E_w/F_v}}
} \]
\end{fact}

\begin{definition}\label{d:F2}
Let $\mc{F}^2(A/F)$, or just $\mc{F}^2(A)$ when $F$ is understood, be the fiber product of
\[ (A[\Vbar]_0 \otimes \Q/\Z)_\G \lra \prod_{v|\infty} (A \otimes \Q/\Z)_{\G\sss{\vbar}}
    \lla \prod_{v|\infty} H^0(\G\sss{\vbar},A), \]
where the arrow at left is the product of $l_{\vbar}$ for all $v|\infty$, and the arrow at right is the product of $\xi_{\G\sss{\vbar}}$.
\end{definition}

\begin{remark}\label{r:F^2}
When $A$ is torsion-free, by Remark \ref{r:xi} the arrow at right in the definition above is injective,
and therefore so is the homomorphism $\cF^2(A)\to (A[\Vbar]_0 \otimes \Q/\Z)_\G$\hs. Moreover, when $F$ is a function field, we have $\cF^2(A)=(A[\Vbar]_0 \otimes \Q/\Z)_\G$\hs.
\end{remark}

For $M$ as in the beginning of this section,
according to Fact \ref{fct:evloc},
the map
\[ \xi_{K/F} \colon H^0(K/F,M[\V_K]_0) \to (M[\V_K]_0 \otimes \Q/\Z)_{\G\sub{K/F}} \cong (M[\Vbar]_0 \otimes \Q/\Z)_\G \]
together with the evaluation maps $\tx{ev}_{\vbar} \colon H^0(K/F,M[\V_K]_0) \to H^0(\G\sss{\vbar},M)$
for $v|\infty$ induce a map
\[ H^0(K/F,M[\V_K]_0) \to \mc{F}^2(M).\]

\begin{lemma} \label{lem:h0glob}
The family of these maps  induces an isomorphism
\[ \varinjlim_K H^0(K/F,M[\V_K]_0) \isoto \mc{F}^2(M) \]
where the  injective limit is taken along the maps $p$.
\end{lemma}

We will need the following known lemma:

\begin{lemma}[\hs{\cite[Chapter III, \S4, Exercise 8(a)]{MacLane71}}\hs]
\label{l:MacLane}
In a diagram
\[ \xymatrix{
{*}\ar[r]\ar[d]&{*}\ar[r]\ar[d]&{*}\ar[d]\\
{*}\ar[r]&{*}\ar[r]&{*}}
\]
if the two inner squares are Cartesian, then so is the outer rectangle.
\end{lemma}

\begin{proof}[Proof of Lemma \ref{lem:h0glob}]
We split the proof into two parts. The first part is the claim that the following diagram
\begin{equation}\label{e:lim-lim}
\begin{aligned}
\xymatrix{
\varinjlim_K H^0(K/F,M[\V_K]_0)\ar[r]\ar[d]&\varinjlim H^0(K/F,\Mtf[\V_K]_0)\ar[d]\\
\prod_{v|\infty}H^0(K_{\vbar}/F_v,M)\ar[r]&\prod_{v|\infty}H^0(K_{\vbar}/F_v,\Mtf)
}
\end{aligned}
\end{equation}
is Cartesian, where  $\Mtf=M/M_{\Tors}$.
Lemma \ref{l:MacLane} reduces the proof of our lemma to the case when $M$ is torsion free.
This is then the second part of the proof.

We commence with the first part.
Let
\[ 0\to M^{-1}\labelt\vk M^0\to M\to 0\]
 be a resolution of $M$
where $M^{-1}$ and $M^0$ are torsion-free finitely generated (over $\Z$) $\G$-modules.
Let $T^{-1}$ and $T^0$ be the $F$-tori with cocharacter groups $M^{-1}$ and $M^0$, respectively.
Consider the induced homomorphism $\vk_T\colon T^{-1}\to T^0$,
and put $T^\bullet=(T^{-1}\to T^0)$,  $T=\coker \vk_T$, and $\KK=\ker\vk_T$;
then $T$ is an $F$-torus with cocharacter group $\Mtf$ and $\KK$ is a finite
$F$-group of multiplicative type (not necessarily smooth)
with character group $\tx{Hom}(M_{\Tors},\Q/\Z)$.
We have a short exact sequence of complexes
\begin{equation*}
 1\to (T^{-1}\to \im\vk_T)\to T^\bullet \to (1\to T)\to 1.
\end{equation*}
Note that we may identify $H^n(F,1\to T)=H^n(F,T)$ and
$H^n(F, T^{-1}\to \im\vk_T)=H^{n+1}(F,\KK)$, where we write $H^n(F,\KK)$ for $H_{\rm fppf}^n(F,\KK)$.
Using the fact that fppf, \'etale, and Galois cohomology agree for tori, see \cite[Proposition 71, Theorem 43]{Shatz},
the short exact sequence above gives rise to a commutative diagram with exact rows
\[
\xymatrix@C=5.6mm{
H^1(F,T)\ar[r]\ar@{->>}[d] &H^3(F,\KK)\ar[r]\ar[d]^-\cong &H^2(F,T^\bullet)\ar[r]\ar[d]\ar@{}[dr]|{\fstar} &H^2(F,T)\ar[r]\ar[d] &H^4(F,\KK)\ar[d]^-\cong\\
\prod\limits_{\infty}\!H^1(F_v,T)\ar[r] &\prod\limits_{\infty}\!H^3(F_v,\KK)\ar[r] &\prod\limits_{\infty}\!H^2(F_v,T^\bullet)\ar[r] &\prod\limits_{\infty}\!H^2(F_v,T)\ar[r] &\prod\limits_{\infty}\!H^4(F_v,\KK).
}
\]
in which the products in the bottom row are over the set of infinite places of $F$.
In this diagram, the vertical arrows labelled with \,$\scriptstyle \cong$
\,are isomorphisms: when $F$ is a number field, this follows
from the Tate--Nakayama theory, see, for instance, \cite[(8.6.10)(ii)]{NSW08},
and when $F$ is a function field, this follows from the fact that  $H^n(F,\KK)=0$ for $n>2$
by \cite[Proposition 3.1.2]{Rosengarten-Mem}.
The leftmost vertical arrow is a surjection; see, for instance, Sansuc \cite[(1.8.2)]{Sansuc81}.
Now a diagram chase shows that rectangle {\SMALL $\fstar$} is Cartesian.
Using the generalization of Tate's isomorphism \cite{Tate66} to complexes of tori
constructed in \S\ref{sub:tiso} and its compatibility with localization due to Proposition \ref{pro:tnfinloc},
we can identify this rectangle with \eqref{e:lim-lim},
which shows that rectangle \eqref{e:lim-lim} is Cartesian as well.
With this, the first part of the proof of this lemma,
namely the reduction to the case when $M$ is torsion-free,  is complete.

We now come to the second part, namely the proof of the lemma under the assumption that $M$ is torsion-free. We need to show that the following diagram is Cartesian:
\begin{equation}\label{e:tf-Cartesian}
\begin{aligned}
 \xymatrix@C=14mm{
	\varinjlim\limits_E H^0(E/F,M[\V_E]_0)\ar[r]^-{(\xi_{E/F})_E}\ar[d]^{(\tx{ev}_{\vbar})_{\vbar}}
&(M[\Vbar]_0\otimes\Q/\Z)_\G\ar[d]^{(l_{\vbar})_{\vbar}}\\
	\prod\limits_{v|\infty}H^0(\G\sss{\vbar},M)\ar[r]^-{\xi_{K_{\vbar}/F_v}}&\prod\limits_{v|\infty}(M\otimes\Q/\Z)_{\G\sss{\vbar}},
}
\end{aligned}
\end{equation}
where the colimit on the top left is taken with respect to the maps $p$.

For this recall that the map
\[ \xi_{E/F} \colon H^0(E/F,M[\V_E]_0) \to (M[\Vbar]_0 \otimes \Q/\Z)_{\G}\]
was defined in \eqref{eq:xi} as the composition of following three maps:
first, the natural inclusion
\[ H^0(E/F,M[\V_E]_0) \to (M \otimes \Q)[\V_E]_0^{\G\sub{E/F}}/N_{E/F}(M[\V_E]_0),\]
second, the inverse of the isomorphism
\[ N_{E/F}\colon (M[\V_E]_0 \otimes \Q/\Z)_{\G\sub{E/F}} \to
    \big((M \otimes \Q)[\V_E]_0\big)^{\G\sub{E/F}}/N_{E/F}(M[\V_E]_0),\]
and third, the identification
\[ (M[\V_E]_0 \otimes \Q/\Z)_{\G\sub{E/F}} \to (M[\Vbar]_0 \otimes \Q/\Z)_{\G}\]
(coming from the definition of the latter).
Accordingly, it would be enough to show that
the following three diagrams are well-defined, commutative, and Cartesian:
\[
\xymatrix{	\varinjlim\limits_E (M[\V_E]_0\otimes\Q/\Z)_{\G\sub{E/F}}\ar[r]\ar[d]^{(l_{\vbar})_{\vbar}}& (M[\Vbar]_0 \otimes \Q/\Z)_\G\ar[d]^{(l_{\vbar})_{\vbar}}\\
\prod\limits_{v|\infty}(M\otimes\Q/\Z)_{\G\sss{\vbar}}\ar@{=}[r]&\prod\limits_{v|\infty}(M\otimes\Q/\Z)_{\G\sss{\vbar}},
}\]
where the colimit on the top left is taken with respect to the maps $!$,
\[
\xymatrix@C=15mm{	\varinjlim\limits_E \big(M[\V_E]_0\otimes\Q/\Z\big)_{\G\sub{E/F}}\!\ar[r]^-{(N_{E/F})_E}\ar[d]^{(l_{\vbar})_{\vbar}}
          & \varinjlim\limits_E\big((M\otimes\Q)[\V_E]_0\big)^{\G\sub{E/F}}/N_{E/F}(M[\V_E]_0)\ar[d]^{(\tx{ev}_{\vbar})_{\vbar}}\\
\prod\limits_{v|\infty}(M\otimes\Q/\Z)_{\G\sss{\vbar}}\ar[r]^-{(N_{K_{\vbar}/F_v})_{\vbar}}
          &\prod\limits_{v|\infty}(M\otimes\Q)^{\G\sss{\vbar}}/N_{\G\sss{\vbar}}(M),
}\]
where the colimit on the top left is taken with respect to the maps $!$, and on the top right with respect to the maps $p$,
and
{\small
\[ \xymatrix@C=7mm{
\varinjlim_E \big(M[\V_E]_0\big)^{\G\sub{E/F}}\!/N_{E/F}(M[\V_E]_0) \ar[r]\ar[d]^-{(\tx{ev}_{\vbar})_{\vbar}}
    &\varinjlim_E\big((M\otimes\Q)[\V_E]_0\big)^{\G\sub{E/F}}\!/N_{E/F}(M[\V_E]_0)\ar[d]^-{(\tx{ev}_{\vbar})_{\vbar}}\\
\prod\limits_{v|\infty} M^{\G\sss{\vbar}}/N_{\G\sss{\vbar}}(M)\ar[r]
    &\prod\limits_{v|\infty} (M \otimes \Q)^{\G\sss{\vbar}}/N_{\G\sss{\vbar}}(M),
}\]
}
where the colimits on the top left and right are taken with respect to the maps $p$.

We first consider the question of being  well-defined, which means that the top horizontal maps for each level $E/F$ are compatible with the transition maps from $E/F$ to $K/F$ and this induces maps on the colimits. For the first of the three diagrams, this is immediate from the definition of $M[\bar V]_0$. For the second diagram, this becomes the equation $p \circ N_{E/F} = N_{K/F} \circ !$, which is Fact \ref{fct:p!n}. For the third diagram the claim is clear since the transition maps in both top corners are given by $p$.

We now consider commutativity and being  Cartesian. For the first diagram the claim is a tautology.
For the second diagram, we note that the two horizontal arrows are both isomorphisms,
so it is enough to show that the diagram commutes, which is an immediate computation.
For the third diagram, it is clear that it commutes, but we need to show that it is Cartesian.
For this, we note that the horizontal arrows are injective because $M$ is torsion-free,
and therefore, the homomorphism  from
$\varinjlim_E M[\V_E]_0^{\G\sub{E/F}}/N_{E/F}(M[\V_E]_0)$ to the fiber product is injective.
We must prove that it is surjective.
Consider the diagram for fixed $E$:
\[ \xymatrix@C=12mm{
 \big(M[\V_E]_0\big)^{\G\sub{E/F}}\!/N_{E/F}(M[\V_E]_0) \ar[r]\ar[d]
    &\big((M\otimes\Q)[\V_E]_0\big)^{\G\sub{E/F}}\!/N_{E/F}(M[\V_E]_0)\ar[d]\\
\prod\limits_{v|\infty} M^{\G\sss{\vbar}}/N_{\G\sss{\vbar}}(M)\ar[r]
    &\prod\limits_{v|\infty} (M \otimes \Q)^{\G\sss{\vbar}}/N_{\G\sss{\vbar}}(M).
}\]
An element of the top right corner whose image in the bottom right corner comes from the bottom left corner is represented by
\[ y=\sum_w m_w\!\cdot w\in \big((M \otimes \Q)[\V_E]_0\big)^{\G\sub{E/F}} \]
with $m_w \in M \otimes \Q$ and $\sum m_w=0$, such that, for each $v|\infty$, $m_{\bar v} \in M^{\G\sub{\bar v}}$.
The $\G\sub{E/F}$-invariance of $y$ implies that $m_w \in M^{\G\sub w}$ for all $w|\infty$.
Now we choose a large enough finite Galois extension $K/F$ containing $E$
so that $[K_u:E_w]\cdot m_w \in M$ for the finitely many places $w$
that are non-archimedean and for which $m_w \neq 0$. Thus after passing to the colimit, the homomorphism to the fiber product becomes surjective,
which shows that diagram \eqref{e:tf-Cartesian} is Cartesian and completes the proof of Lemma \ref{lem:h0glob}.
\end{proof}

\begin{theorem} \label{thm:tnglob}
There are functorial isomorphisms
\begin{align}
&H^1(F, T^{-1}\to T^0)
\,\longisot\, \mc{F}^1(M) ,\tag{1}\\
&H^2(F, T^{-1}\to T^0)
\,\longisot\, \mc{F}^2(M),\tag{2}\\
&H^i(F,T^{-1}\to T^0)
\,\longisot\!\!\!\! \prod_{v\in\V_\R(F)}\!\!\!\! H^{i-2}(F_v,M)\quad\text{for all}\ \, i\ge 3.\tag{3}
\end{align}
\end{theorem}

\begin{proof}
We obtain isomorphisms (1) and (2) from  Lemmas \ref{lem:h-1glob} and \ref{lem:h0glob}, respectively,
and the definition of the Tate--Nakayama isomorphism from \S\ref{sub:tiso}.
We prove (3).
We abbreviate
\[ H^i(F,T^\bullet)\coloneqq H^i(F,T^{-1} \to T^0) = H^i\big(F,\, (M^{-1}\to M^0)\otimes F^{s\,\times}\big), \]
and
\[ H^i(F_\infty,\,\cdot\, )\coloneqq\!\! \prod_{v\in\V_\infty(F)}\!\! H^i(F_v,\,\cdot\, )
        =\!\!\prod_{v\in\V_\R(F)}\!\! H^i(F_v,\,\cdot\, ).\]
We have an exact commutative diagram
{\small
\[
\xymatrix@C=6mm{
H^i(F,T^{-1})\ar[r]\ar[d]^-\cong &H^i(F,T^{0})\ar[r]\ar[d]^-\cong  &H^i(F,T^\bullet)\ar[r]\ar[d]
       &H^{i+1}(F,T^{-1})\ar[r]\ar[d]^-\cong   &H^{i+1}(F,T^{0})\ar[d]^-\cong\\
 H^i(F_\infty,T^{-1})\ar[r] & H^i(F_\infty,T^{0})\ar[r] & H^i(F_\infty,T^\bullet)\ar[r]
       & H^{i+1}(F_\infty,T^{-1})\ar[r]   & H^{i+1}(F_\infty,T^{0})
}
\]
}
where the vertical arrows are localization maps.
In this diagram, for $i\ge 3$  the four vertical arrows labeled with $\scriptstyle\cong$
are isomorphisms by \cite[Corollary I.4.21]{Milne-ADT},
and by the five lemma the middle (unlabeled) vertical arrow is an isomorphism as well.
It remains to observe that by Proposition \ref{p:R} we have a canonical isomorphism
\[ \prod_{v\in\V_\R(F)}\!\! H^{i-2}(F_v,M )\,\isoto\!\!
    \prod_{v\in\V_\R(F)}\!\! H^i(F_v,T^\bullet ) = H^i(F_\infty,T^\bullet).\qedhere\]
\end{proof}

\section{$H^2$ for a finite  $F$-group scheme of multiplicative type}
\label{ss:B-finite}
Let $B$ be a finite group scheme of multiplicative type, not necessarily smooth, over a field $F$.
By this we mean a finite $F$-group scheme that  becomes
diagonalizable over $\Fbar$ in the sense of \cite[Exp. VIII, Definition 1.1]{SGA3}.
Thus there exists a surjective homomorphism $T^{-1}\to T^0$ of tori over a field $F$ whose kernel is $B$.
Set
\[M^{-1}=X_*(T^{-1}),\quad\ M^{0}=X_*(T^{0}),\quad\ M=\coker[M^{-1}\to M^0].\]
By Lemma \ref{l:M-via-K} in Appendix \ref{app:duality}, the finite abelian group $M$ is canonically isomorphic to $\Hom(X^*(B),\Q/\Z)$
where $X^*(B)$ is the character group of $B$.

\begin{corollary}
When $F$ is a global field,
for $B$ and $M$ as above there is  a canonical isomorphism
\[H^2_{\rm fppf}(F,B)\cong \cF^1(M).\]
\end{corollary}

\begin{proof}
We have $H^2_{\rm fppf}(F,B)\cong H^1(F,T^{-1}\to T^0)$,
and by Theorem \ref{thm:tnglob}(1) we have $H^1(F,T^{-1}\to T^0)\cong \cF^1(M)$.
\end{proof}

\begin{remark}
Similarly, when $F$ is a non-archimedean local field, by Theorem \ref{thm:tnloc}(1) we have
\[H^2_{\rm fppf}(F,B)\cong H^1(F,T^{-1}\to T^0)\cong M_\Gt=M_\G.\]
This can be also deduced from the Tate--Nakayama duality for $B$.
Indeed, write $B'=X^*(B)$. Then by the local duality theorem we have
\[ H^2_{\rm fppf}(F, B)\cong \Hom(H^0_{\rm fppf}(F,B'),\Q/\Z)= \Hom(B^{\prime\,\G}, \Q/\Z);\]
see \cite[Theorem (7.2.6)]{NSW08} or \cite[Theorem 10.9]{Harari20} in characteristic 0,
and \cite[Theorem 1.2.2]{Rosengarten-Mem} in positive characteristic.
Moreover, by Lemma \ref{l:A^*}(2) in Appendix \ref{app:duality} we have
\[\Hom(B^{\prime\,\G}, \Q/\Z)\cong\Hom(B',\Q/\Z)_\G\cong M_\G,\]
because $\Hom(B',\Q/\Z)\cong M$ by Lemma  \ref{l:M-via-K}.
\end{remark}

\section{Compatibility with restriction and corestriction}
\label{ss:com-res}

Let $T^{-1}\to T^0$ and $M$ be as in the beginning of this chapter, and
let $F'/F$ be a finite separable extension, not necessarily Galois.
We will discuss how the isomorphisms of Theorems \ref{thm:tnloc} and \ref{thm:tnglob} translate the restriction map
\[ H^i(F,T^{-1}(F^s) \to T^0(F^s)) \to H^i(F',T^{-1}(F^s) \to T^0(F^s)) \]
for $i=1,2$.

Write $\G'=\tx{Gal}(F^s/F') \subset \G=\tx{Gal}(F^s/F)$.
For a $\G$-module $A$, recall the map $N_{\G'\lmod\G} \colon A_\G \to A_{\G'}$ from \S\ref{sub:explicit0-1}.

\begin{proposition}
Consider a non-archimedean local field $F$. The following diagrams commute:
\[ \xymatrix{
	M_\Gt\ar[r]\ar[d]^{N_{\G' \lmod \G}}&H^1(F,T^{-1} \to T^0)\ar[d]^{\tx{Res}}\\
	M_{\G'\hm,\Tors}\ar[r]&H^1(F',T^{-1} \to T^0)
}
\qquad\quad
\xymatrix{(M \otimes\Q/\Z)_\G\ar[r]\ar[d]^{N_{\G' \lmod \G}}&H^2(F,T^{-1} \to T^0)\ar[d]^{\tx{Res}}\\
(M \otimes \Q/\Z)_{\G'}\ar[r]&H^2(F',T^{-1} \to T^0)}	
\]
and
\[
	\xymatrix{
		M_\Gt\ar[r]&H^1(F,T^{-1} \to T^0)\\
		M_{\G'\hm,\Tors}\ar[r]\ar[u]&H^1(F',T^{-1} \to T^0)\ar[u]_{\tx{Cor}}
	}
		\qquad\quad	
\xymatrix{(M \otimes\Q/\Z)_\G\ar[r]&H^2(F,T^{-1} \to T^0)\\
	(M \otimes \Q/\Z)_{\G'}\ar[r]\ar[u]&H^2(F',T^{-1} \to T^0)\ar[u]_{\tx{Cor}}}	
\]
where the top horizontal arrows in each diagram are the isomorphisms of Theorem \ref{thm:tnloc},
and the left-hand vertical arrow in each of the two bottom diagrams is the natural projection map.
\end{proposition}

\begin{proof}
Since inflation commutes with restriction and corestriction,
it is enough to show that for each Galois extension $K/F$ containing $F'$ and splitting $T^{-1}$ and $T^0$,
the similar diagrams commute that are obtained by replacing $H^i(F,\,\cdot\, )$ with $H^i(K/F,\,\cdot\, )$ on the right.
We then use Lemma \ref{lem:tnrescor}(1) to reduce the claim to the commutativity
of the diagrams in Lemmas \ref{lem:res-1} and \ref{lem:res0}.
\end{proof}

\begin{proposition}
\label{p:Res-gl}
Consider a global field $F$. Let $E/F$ be a finite Galois extension and $A$ be a $\G\sub{E/F}$-module.
\begin{enumerate}
\item[{\rm (1)}] The map $(A[\Vbar]_0)_\G \to (A[\Vbar]_0)_{\G'}$ induced by
	\[ N_{\G' \lmod \G} \colon (A[\V_E]_0)_{\G\sub{E/F}} \to (A[\V_E]_0)_{\G\sub{E/F'}} \]
	in turn induces maps $\mc{F}^1(A/F) \to \mc{F}^1(A/F')$ and $\mc{F}^2(A/F) \to \mc{F}^2(A/F')$.
\item[{\rm (2)}] The map $(A[\Vbar]_0)_{\G'} \to (A[\Vbar]_0)_{\G}$ induced by the natural projection
	\[ \tx{proj} \colon (A[\V_E]_0)_{\G\sub{E/F'}} \to (A[\V_E]_0)_{\G\sub{E/F}} \]
	in turn induces maps $\mc{F}^1(A/F') \to \mc{F}^1(A/F)$ and $\mc{F}^2(A/F') \to \mc{F}^2(A/F)$.

\item[{\rm (3)}]
For $T^{-1}\to T^0$ and $M$ be as in the beginning of this chapter,
the following diagrams commute:
\[ \xymatrix{
	\mc{F}^1(M/F)\ar[r]\ar[d]^{N_{\G' \lmod \G}}&H^1(F,T^{-1} \to T^0)\ar[d]_{\tx{Res}}\\
	\mc{F}^1(M/F')\ar[r]&H^1(F',T^{-1} \to T^0)
} \qquad
\xymatrix{\mc{F}^2(M/F)\ar[r]\ar[d]^{N_{\G' \lmod \G}}&H^2(F,T^{-1} \to T^0)\ar[d]_{\tx{Res}}\\
	\mc{F}^2(M/F')\ar[r]&H^2(F',T^{-1} \to T^0)}
\]
and
\[ \xymatrix{
	\mc{F}^1(M/F)\ar[r]&H^1(F,T^{-1} \to T^0)\\
	\mc{F}^1(M/F')\ar[r]\ar[u]_{\tx{proj}}&H^1(F',T^{-1} \to T^0)\ar[u]_{\tx{Cor}}
} \qquad
\xymatrix{\mc{F}^2(M/F)\ar[r]&H^2(F,T^{-1} \to T^0)\\
	\mc{F}^2(M/F')\ar[r]\ar[u]_{\tx{proj}}&H^2(F',T^{-1} \to T^0)\ar[u]_{\tx{Cor}},
}
\]
where the top horizontal arrows in each diagram are the isomorphisms of Theorem \ref{thm:tnglob}.
\end{enumerate}
\end{proposition}

\begin{proof}
(1) For a tower of finite Galois extensions $K/E/F$
with $F' \subset E$, the following diagram commutes:
\[ \xymatrix@C=13mm{
(A[\V_K]_0)_{\G\sub{K/F}}\ar[d]_-{\beta}\ar[r]^-{N_{\G' \lmod \G}}&(A[\V_K]_0)_{\G\sub{K/F'}}\ar[d]^-{\beta}\\
(A[\V_E]_0)_{\G\sub{E/F}}\ar[r]^-{N_{\G' \lmod \G}}&(A[\V_E]_0)_{\G\sub{E/F'}}
}\]
because $\beta$ is $\G\sub{K/F}$-equivariant. This shows that $N_{\G' \lmod \G}$
induces a map $(A[\Vbar]_0)_\G \to (A[\Vbar]_0)_{\G'}$.

That this map induces a map $\mc{F}^1(A/F) \to \mc{F}^1(A/F')$
follows from the commutativity of the following diagram:
\[ \xymatrix@C=11mm{
(A[\V_E]_0)_{\G\sub{E/F},{\Tors}}\ar[d]^{N_{\G' \lmod \G}}\ar[r]^{(l_{\vbar})_v}
    &\prod\limits_{v|\infty}A_{\G\sub{E/F,\vbar },{\Tors}}\ar[d]^{(N_{\G'\sss{\vbar '}\lmod\G\sss{\vbar}})_{v'}}
    &\ar[l]\prod\limits_{v|\infty}H^{-1}(\G\sss{\vbar},A)\ar[d]^{\tx{Res}}\\
(A[\V_E]_0)_{\G\sub{E/F'},{\Tors}}\ar[r]^{(l_{\vbar '})_{v'}}
    &\prod\limits_{v'|\infty}A_{\G\sub{E/F',\vbar '},{\Tors}}
    &\ar[l]\prod\limits_{v'|\infty}H^{-1}(\G'\sss{\vbar '},A)
} \]
In this diagram, the top products run over the infinite places of $F$,
and the bottom products run over the infinite places of $F'$.
Given infinite places $v$ of $F$ and $v'$ of $F'$ with $v'|v$,
the map $N_{\G'\sss{\vbar '} \lmod \G\sss{\vbar}}\, \colon A_{\G\sub{E/F,\vbar },{\Tors}} \to A_{\G\sub{E/F',\vbar '},{\Tors}}$ is obtained as the composition of the isomorphism $A_{\G\sub{E/F,\vbar },{\Tors}} \to A_{\G\sub{E/F,\vbar '},{\Tors}}$ effected by any
 $\sigma \in \G\sub{E/F}$ with $\sigma \vbar  = \vbar '$ and the map
$N_{\G'\sss{\vbar '} \lmod \G\sss{\vbar '}} \,\colon A_{\G\sub{E/F,\vbar '},{\Tors}} \to A_{\G\sub{E/F',\vbar '},{\Tors}}$.

The left-hand rectangle commutes because the top $l_{\vbar}$ is the composition of
$N_{\G\sub{E/F,\vbar } \lmod \G\sub{E/F}}$ and the evaluation at $\vbar $,
while the bottom $l_{\vbar '}$ has the same formula with $F$ replaced by $F'$, and one can use Remark \ref{rem:ncomp}.
The right-hand rectangle commutes by Lemma \ref{lem:res-1}.

To obtain from $N_{\G' \lmod \G}$ a map $\mc{F}^2(A/F) \to \mc{F}^2(A/F')$, we need the commutativity of
\[ \xymatrix@C=11mm{
(A[\V_E]_0 \otimes \Q/\Z)_{\G\sub{E/F}}\ar[d]^{N_{\G' \lmod \G}}\ar[r]^{(l_{\vbar})_v}
&\prod\limits_{v|\infty}(A \otimes \Q/\Z)_{\G\sub{E/F,\vbar}}\ar[d]^{(N_{\G'\sss{\vbar'}\lmod\G\sss{\vbar}})_{v'}}
&\ar[l]\prod\limits_{v|\infty}H^{0}(\G\sss{\vbar},A)\ar[d]^{\tx{Res}}\\
(A[\V_E]_0 \otimes \Q/\Z)_{\G\sub{E/F'}}\ar[r]^{(l_{\vbar '})_{v'}}&\prod\limits_{v'|\infty}(A \otimes \Q/\Z)_{\G\sub{E/F',\vbar '}}&\ar[l]\prod\limits_{v'|\infty}H^0(\G'\sss{\vbar '},A)
} \]
The left-hand rectangle commutes for the same reason as above, and the right-hand rectangle commutes by Lemma \ref{lem:res0}.

(2) We consider the analogs of the two diagrams from (1), which are now
\[ \xymatrix@C=11mm{
(A[\V_E]_0)_{\G\sub{E/F},{\Tors}}\ar[r]^{(l_{\vbar})_v}&\prod\limits_{v|\infty}A_{\G\sub{E/F,\vbar },{\Tors}}&\ar[l]\prod\limits_{v|\infty}H^{-1}(\G\sss{\vbar},A)\\
(A[\V_E]_0)_{\G\sub{E/F'},{\Tors}}\ar[u]_{\tx{proj}}\ar[r]^{(l_{\vbar '})_{v'}}&\prod\limits_{v'|\infty}A_{\G\sub{E/F',\vbar '},{\Tors}}\ar[u]_-\sum&\ar[l]\prod\limits_{v'|\infty}H^{-1}(\G'\sss{\vbar '},A)\ar[u]_-{\tx{Cor}}
} \]
and
\[ \xymatrix@C=11mm{
(A[\V_E]_0 \otimes \Q/\Z)_{\G\sub{E/F}}\ar[r]^{(l_{\vbar})_v}&\prod\limits_{v|\infty}(A \otimes \Q/\Z)_{\G\sub{E/F,\vbar }}&\ar[l]\prod\limits_{v|\infty}H^{0}(\G\sss{\vbar},A)\\
(A[\V_E]_0 \otimes \Q/\Z)_{\G\sub{E/F'}}\ar[u]_-{\tx{proj}}\ar[r]^{(l_{\vbar '})_{v'}}&\prod\limits_{v'|\infty}(A \otimes \Q/\Z)_{\G\sub{E/F',\vbar '}}\ar[u]_-\sum&\ar[l]\prod\limits_{v'|\infty}H^0(\G'\sss{\vbar '},A)\ar[u]_-{\tx{Cor}}
} \]
In these diagrams, the left-hand vertical arrows are the natural projections;
the middle arrows are made up by summing up, for each $v|\infty$,
the compositions of the natural projections $(A \otimes \Q/\Z)_{\G\sub{E/F',\vbar '}} \to (A \otimes \Q/\Z)_{\G\sub{E/F,\vbar '}}$
and the isomorphisms $(A \otimes \Q/\Z)_{\G\sub{E/F,\vbar '}} \to (A \otimes \Q/\Z)_{\G\sub{E/F,\vbar }}$ over  all $v'|v$;
the right-hand arrows are made by summing up, for each $v|\infty$,
the compositions of the corestrictions $H^0(\G'\sss{\vbar '},A) \to H^0(\G\sss{\vbar '},A)$
and the isomorphisms $H^0(\G\sss{\vbar '},A) \to H^0(\G\sss{\vbar},A)$
effected by any $\sigma \in \G$ with $\sigma\vbar'=\vbar$.
The commutativity of the left-hand rectangles is immediate,
and that of the right-hand rectangles follows from Lemmas \ref{lem:res0} and \ref{lem:res-1}.

(3) As in the local case, the proof can be reduced via Lemma \ref{lem:tnrescor}(2) to  Lemmas \ref{lem:res-1} and \ref{lem:res0}.
\end{proof}

\begin{remark}
For a finite separable field extension $F'/F$ and a complex of $F$-tori $T^\bullet=(T^{-1}\to T^0)$,
for $i>0$ consider the homomorphisms
\begin{align*}
\tx{Res}\colon &H^i(F,T^\bullet )\to H^i(F',T^\bullet)\quad\ \text{and}
\quad\ \tx{Cor}\colon H^i(F',T^\bullet)\to H^i(F,T^\bullet).
\end{align*}
Then $\tx{Cor}\circ\tx{Res}=[F':F]$.

Indeed,  it holds for $H^i(E/F,T^\bullet)$ and $H^i(E/F',T^\bullet)$
for all finite Galois extensions $E/F$ containing $F'$ and for all $i\in\Z$; see  \S\ref{ss:Res-Cor}.
Now both restriction and corestriction commute with inflation in degree $i>0$,
and so does multiplication by $[F':F]$, from which the result follows by passing to the colimit.
\end{remark}

\section{Compatibility with localization}

Let $F$ be a global field, and let $\vbar  \in\! \Vbar=V(F^s)$.
Recall the map $l_{\vbar} \colon (A[\Vbar]_0)_\G \to A_{\G\sss{\vbar}}$ of \eqref{eq:loc}. For a finite Galois extension $E/F$ that splits $T^{-1}$ and $T^0$ let $w$ be the place of $E$ under $\vbar $. Recall the maps $l_w^i \colon H^i(E/F,M[\V_E]_0) \to H^i(E_w/F_v,M)$ defined as the composition of the restriction along $\G\sub{E/F,w} \to \G\sub{E/F}$ and the evaluation at $w$.

\begin{lemma} \label{lem:loc1}
Let $A$ be a $\Z[\G\sub{E/F}]$-module. The map $l_{\vbar} \colon (A[\Vbar]_0)_\G \to A_{\G\sss{\vbar}}$ induces maps $\mc{F}^1(A) \to A_{\G\sss{\vbar},{\Tors}}$ and $\mc{F}^2(A) \to (A \otimes \Q/\Z)_{\G\sss{\vbar}}$. These maps fit into the commutative diagrams
\[ \xymatrix@C=5mm{
	H^{-1}(E/F,A[\V_E]_0)\ar[r]\ar[d]_-{l^{-1}_w}&\mc{F}^1(A)\ar[d]^{l_{\vbar}}\\
	H^{-1}(E_w/F_v,A)\ar[r]&A_{\G\sub{E/F,w},{\Tors}}
} \quad
\xymatrix@C=5mm{
	H^0(E/F,A[\V_E]_0)\ar[r]\ar[d]_-{l^0_w}&\mc{F}^2(A)\ar[d]^{l_{\vbar}}\\
	H^0(E_w/F_v,A)\ar[r]&(A \otimes \Q/\Z)_{\G\sub{E/F,w}}
} \]
\end{lemma}

\begin{proof}
We recall that $l^i_w$ is defined as the composition of the restriction map in degree $i$
along the inclusion $\G\sub{E/F,w} \to \G\sub{E/F}$ and the evaluation at $w$ map,
while $l_{\vbar}$ is defined on the level $E/F$
as the composition of the map $N_{\G\sub{E/F,w} \lmod \G\sub{E/F}}$
and the evaluation at $w$ map. The commutativity of the  diagram at left
now follows from Lemma \ref{lem:res-1}, and that of the one at right from Lemma \ref{lem:res0}.
\end{proof}

\begin{proposition}
\label{p:loc2}
For $T^{-1}\to T^0$ and $M$ be as in the beginning of this chapter and
for any place $v$ of $F$, we have commutative diagrams
\[ \xymatrix{
H^1(F,T^{-1} \to T^0)\ar[r]^-\sim \ar[d]_-{\tx{loc}_v}	&\mc{F}^1(M)\ar[d]^-{l_{\vbar}}   \\
H^1(F_v,T^{-1}\to T^0)\ar[r] 	                        &M_{\G\sss{\vbar},{\Tors}}
}\qquad \xymatrix{
H^2(F,T^{-1} \to T^0)\ar[r]^-\sim \ar[d]_-{\tx{loc}_v}	&\mc{F}^2(M) \ar[d]^-{l_{\vbar}}  \\
H^2(F_v,T^{-1}\to T^0)\ar[r]                       &(M \otimes \Q/\Z)_{\G\sss{\vbar}}
} \]
\end{proposition}

\begin{proof}
By Theorem \ref{thm:tnglob} we have commutative diagrams
\[ \xymatrix{
H^1(F,\Tbul)\ar[r]^-\sim \ar[d]_-{\tx{loc}_\infty}	&\mc{F}^1(M)\ar[d]^-{l_\infty}   \\
\prod\limits_{v|\infty}H^1(F_v,\Tbul)\ar[r] 	                        &\prod\limits_{v|\infty}M_{\G\sss{\vbar},{\Tors}}
}\qquad \xymatrix{
H^2(F,\Tbul)\ar[r]^-\sim \ar[d]_-{\tx{loc}_\infty}	&\mc{F}^2(M) \ar[d]^-{l_\infty}  \\
\prod\limits_{v|\infty}H^2(F_v,\Tbul)\ar[r]                       &\prod\limits_{v|\infty}(M \otimes \Q/\Z)_{\G\sss{\vbar}}
} \]
where $\Tbul=(T^{-1}\to T^0)$. From these diagrams we obtain the desired diagrams for $v$ infinite.

When $v$ is finite, it suffices to prove the commutativity of the following diagrams:
\[ \xymatrix{
	\mc{F}^1(M)\ar[r]^-\sim\ar[d]_-{l_{\vbar}}   &H^1(F,\Tbul)\ar[d]^-{\tx{loc}_v}\\
	M_{\Gamma_{\vbar},{\Tors}}\ar[r]^-\sim      &H^1(F_v,\Tbul)
}\qquad \xymatrix{
	\mc{F}^2(M)\ar[r]^-\sim\ar[d]_-{l_{\vbar}}       &H^2(F,\Tbul)\ar[d]^-{\tx{loc}_v}\\
	(M \otimes \Q/\Z)_{\Gamma_{\vbar}}\ar[r]^-\sim  &H^2(F_v,\Tbul)
} \]
We want to prove that in each diagram, the composition of the top and right-hand arrows
equals that of the left-hand and bottom arrows.
According to Lemmas \ref{lem:h-1glob} and \ref{lem:h-1loc}
as well as Lemma \ref{lem:h0glob} and Corollary \ref{cor:h0loc},
it is enough to prove this equality after composition
with the map $H^{i-2}(E/F,M[\V_E]_0) \to \mc{F}^i(M)$ for $i=1,2$.
Lemma \ref{lem:loc1} then reduces the proof to Proposition \ref{pro:tnfinloc}.
\end{proof}

\section{Possible injectivity and surjectivity in Proposition \ref{p:loc2}}
\label{ss:inj-surj}
When $v$ is finite, the bottom arrows in the diagrams of Proposition \ref{p:loc2}
are bijective. We discuss the possible injectivity and surjectivity when $v$ is real.

Let $F=\R$ and let $\Tbul=(T^{-1}\to T^0)$ be a complex of $\R$-tori with finite kernel.
We write $\G=\G(\C/\R)$ and consider the $\G$-module $M\coloneqq X_*(T^{-1})\to X_*(T^0)$.
Consider the composite homomorphisms
\begin{align*}
&H^1(\R,\Tbul)\,\longisot\, H^{-1}(\G,M)\labelto{\varphi^{-1}} M_\Gt\hs,\\
&H^2(\R,\Tbul)\,\longisot\, H^0(\G,M)\labelto{\varphi^0} M_\G\otimes\Q/\Z
\end{align*}
where the homomorphisms at left are bijective by Proposition  \ref{p:R}.
We discuss the possible injectivity and surjectivity of the homomorphisms $\varphi^{-1}$ and $\varphi^0$.

\begin{proposition}\label{p:R-Tbul}
With the above notation:
\begin{enumerate}
\item $\varphi^{-1}$ is injective, but not necessarily surjective;
\item $\varphi^0$ might be neither injective not surjective.
\item Assume that $M$ is torsion-free. Then $\varphi^0$ is injective, but not necessarily surjective.
\end{enumerate}
\end{proposition}

\begin{proof}
(1) The homomorphism $\varphi^{-1}$ is injective by Fact \ref{fct:h-1}.
If $M=\Z/3\Z$  (with trivial $\G$-action), then $H^{-1}(\G,M)=0$, while $M_\Gt\cong\Z/3\Z$, and $\varphi^{-1}$ is not surjective.

(2) If $M=\Z$, then $H^0(\G,M)\cong\Z/2\Z$, while $M_\G\otimes \Q/\Z\cong\Q/\Z$, and $\varphi^0$ is not surjective.
Moreover, if $M=\Z/2\Z$, then $H^0(\G,M)\cong\Z/2\Z$, while $M_\G\otimes \Q/\Z=0$, and $\varphi^0$ is not injective.
Furthermore, if we take $M=\Z\oplus\Z/2\Z$, then $\varphi^0$ is neither injective nor surjective.

(3)The classification of torsion-free $\G$-modules (see, for instance, \cite[Appendix A]{BorTim24}\hs)
immediately reduces the injectivity assertion to the case $M=\Z$ with trivial $\G$-action.
In this case  $H^0(\G,M)\cong\Z/2\Z$, $M_\G\otimes\Q/\Z\cong\Q/\Z$,
and $\varphi^0$ sends $1+2\Z\in \Z/2\Z=H^0(\G,M)$ to $\frac12+\Z\in\Q/\Z=M_\G\otimes\Q/\Z$;
see Facts \ref{fct:h0} and \ref{fct:xig}.
Thus $\varphi^0$ is injective, but not surjective.
\end{proof}

\section{Compatibility with connecting homomorphisms}

Consider an exact sequence
\[ 0 \to [M^{-1}_1 \to M^0_1] \to [M_2^{-1} \to M_2^0] \to [M^{-1}_3 \to M^0_3] \to 0 \]
where for $i=1,\, 2,\, 3$,  $M_i^{-1}$ and $M_i^0$ are torsion free $\G(F)$-modules,
and the homomorphisms $M_i^{-1}\to M_i^0$ are injective.
On the one hand, it leads to the exact sequence of complexes of tori
\begin{equation*}
1 \to [T^{-1}_1 \to T^0_1] \to [T_2^{-1} \to T_2^0] \to [T^{-1}_3 \to T^0_3] \to 1,	
\end{equation*}
which in turn leads to a long exact sequence in Tate cohomology. On the other hand, it also leads to the exact sequence
\[ 0 \to M_1 \to M_2 \to M_3 \to 0, \]
where as before $M_i=M_i^0/M_i^{-1}$.

Let $E/F$ be a finite Galois extension such that the action of $\G$ on each $M_i$ factors through $\G\sub{E/F}$.
Set $A_i=M_i$ if $F$ is local, and $A_i=M_i[\V_E]_0$ if $F$ is global.
By Theorem  \ref{t:Hinich} in Appendix \ref{app:Hinich},  we have an exact sequence
\[
\xymatrix{
(A_1)_{\G\sub{E/F},{\Tors}}\ar[r]&(A_2)_{\G\sub{E/F},{\Tors}}\ar[r]&(A_3)_{\G\sub{E/F},{\Tors}}\ar  `d[l]`[lld]^\delta[lld]\\
(A_1 \otimes \Q/\Z)_{\G\sub{E/F}}\ar[r]&(A_2 \otimes \Q/\Z)_{\G\sub{E/F}}\ar[r]&(A_3 \otimes \Q/\Z)_{\G\sub{E/F}},
}\]
where the horizontal arrows are induced by the maps between the $A_i$, and the map $\delta$ was defined in \eqref{eq:connect-delta}.

\begin{proposition}\label{p:F-connecting-local}
		Assume that $F$ is non-archimedean local. We have a commutative diagram
		\[ \xymatrix{
			(M_3)_{\G,{\Tors}}\ar[r]^\delta\ar[d]_-\cong&(M_1\otimes\Q/\Z)_{\G}\ar[d]^-\cong \\
			H^1(F,T_3^{-1} \to T_3^0)\ar[r]&H^2(F,T_1^{-1} \to T_1^0)
		}\]
		in which the bottom horizontal arrow is the connecting homomorphism,
        and the vertical arrows are the isomorphisms of Theorem \ref{thm:tnloc}.
\end{proposition}

\begin{proof}
According to Lemma \ref{lem:h-1loc}, it is enough to prove the commutativity
of this diagram after composing with the map $H^{-1}(E/F,M_3) \to (M_3)_{\G,{\Tors}}$
where $E/F$ is a large enough finite Galois extension.
Consider the diagram
\[ \xymatrix{
	H^{-1}(E/F,A_3)\ar[r]\ar[d]&H^0(E/F,A_1)\ar[d]\\
	H^1(E/F,T_3^{-1} \to T_3^0)\ar[r]&H^2(E/F,T_1^{-1} \to T_1^0)
}\]
Since cup products respect connecting homomorphisms, and the vertical arrows (isomorphisms)
are given by cup product with the fundamental class
$$\alpha_{E/F} \in H^2(\G\sub{E/F},E^\times),$$
this diagram commutes, where the top and bottom horizontal arrows
are the connecting homomorphisms in the corresponding exact sequences of $\G\sub{E/F}$-modules.
The proposition now follows from Lemma \ref{lem:connect-delta}.
\end{proof}

Let $F$ be global.
Note that $\delta\colon (M_3[\V_E]_0)_{\G\sub{E/F},{\Tors}} \to (M_1[\V_E]_0 \otimes \Q/\Z)_{\G\sub{E/F}}$
is compatible with transition from $E/F$ to $K/F$, that is, with the transition map $\beta$,
and hence induces a map $\delta\colon (M_3[\Vbar]_0)_{\G,{\Tors}} \to (M_1[\Vbar]_0 \otimes \Q/\Z)_{\G}$.
Then we check that the following diagram commutes:
\[ \xymatrix@C=16mm{
	\prod\limits_{v|\infty}H^{-1}(\G\sss{\vbar},M_3)\ar[r]\ar[d]_-{\delta_\infty}
    &\prod_{v|\infty} (M_3)_{\G\sss{\vbar},{\Tors}}\ar[d]^{(\delta)_{\vbar}}
    &\ar[l]_-{(l_{\vbar})_{\vbar}}(M_3)_{\G,{\Tors}}\ar[d]^\delta\\
\prod\limits_{v|\infty}H^0(\G\sss{\vbar},M_1)\ar[r]^-{(\xi_{F^s_{\vbar}/F_v})_{\vbar}}
     &\prod\limits_{v|\infty} (M_1 \otimes \Q/\Z)_{\G\sss{\vbar}}
     &\ar[l]_-{(l_{\vbar})_{\vbar}}(M_1 \otimes \Q/\Z)_{\G}
}\]
The commutativity of the left-hand rectangle was checked in the proof of Proposition \ref{p:F-connecting-local}.
The commutativity of the right-hand rectangle is immediate.
Now we can define  a map $\delta_\cF$ from the fiber product of the top row to the fiber product of the bottom row.

\begin{proposition}\label{p:F-connecting-global}	
Assume that $F$ is global.
We have a commutative diagram
\begin{equation}\label{e:delta-1-F}
\begin{aligned}
        \xymatrix{
			\mc{F}^1(M_3)\ar[r]\ar[d]_-\cong      &\mc{F}^2(M_1)\ar[d]^-\cong\\
			H^1(F,T_3^{-1} \to T_3^0)\ar[r]&H^2(F,T_1^{-1} \to T_1^0)
		}
\end{aligned}
\end{equation}
in which the vertical arrows are the isomorphisms of Theorem \ref{thm:tnglob} and
the bottom horizontal arrow is the connecting homomorphism.
Moreover, the top horizontal arrow is
the fiber product of the connecting homomorphism
\[ \delta \colon (M_3[\Vbar]_0)_{\Gt} \to (M_1[\Vbar]_0 \otimes \Q/\Z)_{\G}\]
of \eqref{eq:connect-delta} and the connecting homomorphism
\[\delta_\infty\colon \prod_{v|\infty} H^{-1}(\G\sub \vbar\hs,M_3)\to \prod_{v|\infty} H^{0}(\G\sub \vbar\hs,M_1)\]
over
\[\delta_\infty^M\colon\prod_{v|\infty} (M_3)_{\G(\vbar),\Tors}\to\prod_{v|\infty} (M_1\otimes \Q/\Z)_{\G(\vbar)}\]
where $\delta^M_\infty$ is the connecting homomorphism $\delta_0$
in the homology  exact sequence \ref{e:LG} in Theorem \ref{t:Hinich}.
\end{proposition}

\begin{proof}
The proof is similar to that of Proposition \ref{p:F-connecting-local}.
We use Lemma \ref{lem:h-1glob} instead of Lemma  \ref{lem:h-1loc}.
\end{proof}

\chapter{Galois cohomology of reductive groups}
\label{sec:red}

In this chapter we will study the Galois cohomology
of a connected reductive group $G$ defined over a local or global field $F$.
After reviewing the abelian cohomology groups $H^n_\tx{ab}(F,G)$
and their relation to the usual Galois cohomology $H^1(F,G)$,
we will use the results of Chapter \ref{s:ab-coho} to compute $H^1(F,G)$
in terms of the algebraic fundamental group $\pi_1(G)$
and the real Galois cohomology sets $H^1(F_v,G)$.
We will also compute $H^2(F,T)$ for an $F$-torus $T$ in terms of $\pi_1(T)=X_*(T)$.

{\em Notation.}
For a reductive group over an arbitrary field $F$, we denote by $G_\ssc$
the simply connected cover of the commutator subgroup $[G,G]$ of $G$;
see \cite[Proposition (2.24)(ii)]{BoT72} or \cite[Corollary A.4.11]{CGP15}.
We denote $\G=\tx{Gal}(F^s/F)$.

\section{$t$-Extensions}

\begin{definition}\label{d:toric}
Let $G$ be a reductive group over an arbitrary field $F$.

\begin{enumerate}
\item[\rm (1)] \cite[Definition 2.1]{BGA14} \
A {\em $t$-extension} (or $t$-resolution) of a reductive group $G$
is a surjection $G'\to G$ such that  $G'$ is a reductive $F$-group
with simply connected derived group $[G',G']$
and that $\ker[G'\to G]$ is an $F$-torus.

\item[\rm (2)] A $t$-extension of  a homomorphism of reductive $F$-groups
$\varphi\colon G\to H$  is a commutative diagram
\[
\xymatrix{
G'\ar[r]^-{\varphi'}\ar[d] &H'\ar[d]\\
G\ar[r] \ar[r]^-\varphi & H
}
\]
where $G'\to G$ and $H'\to H$ are $t$-extensions of $G$ and $H$, respectively.
\end{enumerate}
\end{definition}

\begin{lemma}\label{l:toric}
\begin{enumerate}
\item[\rm (1)] Any reductive $F$-group $G$ admits a $t$-extension.
\item[\rm (2)] Any homomorphism $\varphi\colon G\to H$ of reductive $F$-groups admits a $t$-extension.
\end{enumerate}
\end{lemma}

This lemma is well-known; it goes back to Langlands \cite[pp. 228-229]{LanCor}.
Note that the $z$-extensions of Kottwitz \cite{Kot82} and the flasque resolutions of Colliot-Th\'el\`ene \cite{CT08}
are special kinds of $t$-extensions.

\begin{proof}
For (1) see \cite[Proposition 2.2]{BGA14}. For (2) see \cite[Proposition 2.8]{BGA14}.
\end{proof}

\section{The fundamental group and abelian cohomology}
\label{ss:fund-gp}

For a reductive group over an arbitrary field $F$,
following Deligne \cite[\S2.0.11]{Deligne79} we consider the composite homomorphism
\[\rho\colon G_\ssc\onto[G,G]\into G,\]
which in general is neither injective nor surjective.
For a maximal torus  $T\subseteq G$, we denote
$T_\ssc=\rho^{-1}(T)\subseteq G_\ssc$\hs,
which is a maximal torus of $G_\ssc$.
The induced homomorphism on cocharacters $\rho_*\colon X_*(T_\ssc)\to X_*(T)$ is injective.
The {\em algebraic fundamental group} of $G$ is defined by
\[\pi_1(G)=X_*(T)/\rho_*X_*(T_\ssc).\]
The Galois group $\G=\tx{Gal}(F^s/F)$ naturally acts on $\pi_1(G)$,
and the  $\tx{Gal}(F^s/F)$-module $\pi_1(G)$
is well defined (does not depend on the choice of $T$
up to a transitive system of isomorphisms);
see \cite[Lemma 1.2]{Brv98} and \cite[Proposition 2.2.1]{Fu25}.
Consider  the complex of $\G$-modules $X_*(T_\ssc)\to X_*(T)$ in degrees $-1$ and $0$.
Its class in the derived category depends only on $\pi_1(G)$, and hence does not depend on the choice of $T$.
This also holds for the complex
\[ \big(X_*(T_\ssc)\otimes F^{s\hs\times} \to X_*(T)\otimes F^{s\hs\times}\big)\, =\,\big(T_\ssc(F^s)\to T(F^s)\big)\]
and for its hypercohomology groups
\[ H^n_\ab(F,G)\coloneqq H^n(F, T_\ssc\to T)\coloneqq H^n\big(\G,T_\ssc(F^s)\to T(F^s)\big)\]
for $n\ge 0$.
Thus the  abelian groups $H^n_\ab(F,G)$ are well-defined  (do not depend on the choice of $T$).
The results of Chapter \ref{s:ab-coho} compute $H^n_\ab(F,G)$ over a local or global field $F$
for $n\ge 1$ in terms of the $\G$-module  $M=\pi_1(G)$.

Let
\begin{equation}\label{e:G'-G-G''}
1 \to G' \labelt \iota G \labelt \vk G'' \to 1
\end{equation}
be a short exact sequence of connected reductive $F$-groups.
We choose a maximal torus $T\subseteq G$ and consider the maximal tori \,$T'=\iota^{-1}(T)\subseteq G'$,  \  $T''=\vk(T)\subseteq G''$.
\,We define the maximal tori \,$T_\ssc\subseteq G_\ssc,\ \, T'_\ssc\subseteq G'_\ssc,\ \, T''_\ssc\subseteq G''_\ssc$
\,to be the preimages of $T,\,T',\,T''$, respectively.
Then we obtain commutative diagrams with exact rows
\begin{equation}\label{e:tori}
\begin{aligned}
\xymatrix@R=6mm{
1\ar[r] &T'_\ssc\ar[r]\ar[d] &T_\ssc\ar[r]\ar[d] &T''_\ssc\ar[r]\ar[d] &1\\
1\ar[r] &T'\ar[r]            &T\ar[r]            &T''\ar[r]            &1
}
\end{aligned}
\end{equation}
\begin{equation}\label{e:cochar}
\begin{aligned}
\xymatrix@R=6mm{
0\ar[r] &X_*(T'_\ssc)\ar[r]\ar[d] &X_*(T_\ssc)\ar[r]\ar[d] &X_*(T''_\ssc)\ar[r]\ar[d] &0\\
0\ar[r] &X_*(T')\ar[r]            &X_*(T)\ar[r]            &X_*(T'')\ar[r]            &0
}
\end{aligned}
\end{equation}

\begin{lemma}[\hs{\cite[Lemma 1.5]{Brv98}, \cite[Proposition 6.8]{CT08}}\hs]
\label{l:G'-G-G''-pi1}
The following sequence is exact:
\begin{equation*}
0\to\pi_1(G')\labelt{\,\iota_*} \pi_1(G)\labelt{\,\vk_*} \pi_1(G'')\to 0
\end{equation*}
where $\pi_1(G)=\coker \![X_*(T_\ssc)\to X_*(T)]$ and similarly for $\pi_1(G')$ and  $\pi_1(G'')$.
\end{lemma}

\begin{proof}
In diagram \eqref{e:cochar} the vertical arrows are injective, and the lemma follows from the snake lemma.
\end{proof}

\begin{lemma}\label{l:ab-exact-long}
The short exact sequence \eqref{e:G'-G-G''} gives rise to
a long exact sequence of abelian cohomology for $i\ge -1$
\begin{equation*}
\dots\to H^i_\ab(F,G')\labelt{\iota_*} H^i_\ab(F,G) \labelt{\vk_*} H^i_\ab(F,G'')
   \labelt{\delta} H^{i+1}_\ab(F,G')\labelt{\iota_*}\dots
\end{equation*}
\end{lemma}

\begin{proof}
The diagram \eqref{e:tori} can be regarded as a short exact sequence of complexes of tori
\[1\to (T'_\ssc\to T') \labelto{\iota} (T_\ssc\to T)\labelto{\vk} (T''_\ssc\to T'')\to 1\]
and so it gives rise to the desired long exact sequence in hypercohomology.
\end{proof}

\section{A stable crossed module}
\label{ss:stable-crossed}

Let $G$ be a reductive group over an arbitrary field $F$.
The group $G$ acts on itself by conjugation, and by functoriality $G$ acts on $G_\ssc$.
We obtain an action
\[\theta\colon G\times G_\ssc\to G_\ssc,\qquad (g,s)\mapsto {}^g\!
    s\quad\ \text{for} \ g\in G,\ s\in G_\ssc\hs,\]
where by abuse of notation we write $g\in G$ for $g\in G(\ov F)$.
Choose a $t$-extension $G'\to G$ and a lift $g'\in G'$ of $g\in G$;
then $G_\ssc$ embeds into $G'$ and we have
\[ {}^g\!s=g's\hs g^{\prime\,-1}.\]
The three actions: \,$G_\ssc$ on $G_\ssc$\hs, \,$G$ on $G_\ssc$\hs, \,and \,$G$ on $G$, \,are compatible:
\begin{align*}
^{\rho(s)}s'=s s' s^{-1},\quad\
\rho(\hs^{g}\! s')&=g \rho(s') g^{-1}\quad\text{for}\ \,g\in G,\ s,s'\in G_\ssc\hs.
\end{align*}
In other words, $(G_\ssc\to G,\, \theta)$ is a {\em crossed module} of algebraic groups,
and therefore
\[ (G_\ssc(F^s)\to G(F^s),\,\theta)\]
is a $\G$-equivariant crossed module where $\G=\tx{Gal}(F^s\hm/F)$;
see Appendix \ref{app:crossed}.

Deligne \cite[\S2.0.2]{Deligne79} observed that the commutator map
\[ [\ ,\,]\colon G\times G\to G,\qquad
  g_1,g_2\mapsto [g_1,g_2]\coloneqq g_1 g_2 g_1^{-1} g_2^{-1}\]
lifts to a morphism of $F$-varieties
\[ \{\,,\} \colon G\times G\to G_\ssc\qquad\text{such that}\quad\ \rho\big(\{g_1,g_2\}\big)=[g_1,g_2],\]
as follows.
The commutator map
\[G_\ssc\times G_\ssc\to G_\ssc,\quad\ s_1,s_2\mapsto [s_1,s_2]\coloneqq s_1 s_2 s_1^{-1} s_2^{-1}\]
clearly factors via  a morphism of $F$-varieties
\[(G_\ssc)_\ad\times (G_\ssc)_\ad\to G_\ssc\]
where $(G_\ssc)_\ad=G_\ssc/Z_{G_\ssc}$ and $Z_{G_\ssc}$ denotes the center of $G_\ssc$.
Identifying $(G_\ssc)_\ad$ with $G_\ad\coloneqq G/Z_G$\hs, we obtain the desired morphism of $F$-varieties
\[\{\,,\}\colon G\times G\to G_\ad\times G_\ad\to G_\ssc.\]
If we choose a $t$-extension $G'\to G$ and choose lifts $g_1',g_2'\in G'$ of $g_1,g_2\in G$, respectively,
then $G_\ssc$ embeds into $G'$ and we have
\[\{g_1,g_2\}=[g_1',g_2']. \]
The map $\{\,,\}$ has the following properties first formulated by Conduch\'e \cite[(3.11)]{Conduche}:
\begin{align*}
&\rho\big(\{g_1,g_2\}\big)=[g_1,g_2];\\
& \{\rho(s_1),\rho(s_2)\}=[s_1,s_2];\\
&\{g_1,g_2\}=\{g_2,g_1\}^{-1};\\
&\{\rho(s),g\}= s\,\hs^g\!s^{-1};\\
&\{g_1g_2,\hs g_3\}=\{g_1 g_2 g_1^{-1},\,g_1 g_3 g_1^{-1}\}\hs\{g_1,g_3\};
\end{align*}
for \,$g,g_1,g_2,g_3\in G,\ s,s_1,s_2\in G_\ssc$.
In other words, $\{\,,\}$ is a {\em symmetric braiding}
of the crossed module $(G_\ssc\to G,\, \theta)$, and the tuple
$\big(G_\ssc\to G,\, \theta,\{\,,\}\big)$ is a {\em symmetrically braided}, or {\em stable} crossed module of algebraic groups
(following Conduch\'e, we  write ``stable'' instead of ``symmetrically braided'').
We regard $\big(G_\ssc\to G,\, \theta,\{\,,\}\big)$ as the {\em abelianization of $G$}, and write
\[  G_\ab=\big(G_\ssc\to G,\, \theta,\{\,,\}\big).\]
This stable crossed module of algebraic groups $G_\ab$ defines
a $\G$-equivariant stable crossed module
\[ \big(G_\ssc(F^s)\to G(F^s),\, \theta,\{\,,\}\big).\]

\begin{lemma}\label{l:braiding}
Let $\varphi\colon G\to H$ be a homomorphism of reductive groups over  $F$,
and let $\varphi_\ssc\colon G_\ssc\to H_\ssc$ denote the induced homomorphism.
Then for each $g,g_1,g_2\in G$ and $s\in G^\ssc$ we have with obvious notation:
\[ {}^{\varphi(g)}\!\big(\varphi_\ssc(s)\big)=\varphi_\ssc(\hs^g\!s),\qquad
      \big\lbrace \varphi(g_1),\varphi(g_2)\big\rbrace _H=
      \varphi_\ssc\big(\lbrace g_1,g_2\rbrace _G\big). \]
\end{lemma}

\begin{proof}
By Lemma \ref{l:toric} there exists an  $t$-extension $\varphi'\colon G'\to H'$ of $\varphi$.
We choose lifts $g', g_1', g_2'\in G'$ of $g,g_1,g_2\in G$, respectively.
Note that $s\in G_\ssc\subseteq G'$.
Then we must check the formulas
\[\vp(g')\vp(s)\vp(g')^{-1}=\vp(g' s\hs g^{\prime\,-1})\quad\  \text{and}\quad\
      \big[\vp(g_1'),\vp(g_2')\big]=\vp\big([g_1',g_2']\big), \]
which follow immediately from the fact that $\vp$ is a homomorphism.
\end{proof}

Lemma \ref{l:braiding} shows that a homomorphism of reductive $F$-groups $\varphi\colon G\to H$
induces a morphism of abelianizations
\begin{equation}\label{e:phi-ab}
\varphi_\ab\colon\,G_\ab\coloneqq\big(G_\ssc\to G,\, \theta_G,\{\,,\}_G\big)
       \,\lra\, \big(H_\ssc\to H,\, \theta_H,\{\,,\}_H\big)\eqqcolon H_\ab.\
\end{equation}

\section{The abelianization map} \label{sub:abcoh}
Let $G$ be a reductive group over an arbitrary field $F$.
In  \cite{Brv98}, the first-named author constructed a functorial {\em abelianization map}
\[\ab\colon H^1(F,G)\to H^1_\ab(F,G)\]
in the case when $F$ is a field of characteristic 0.
Later Gonz\'alez-Avil\'es \cite[Theorem 5.1]{GA12} constructed the abelianization map
over a field $F$ of arbitrary characteristic
(even over an arbitrary base scheme) using flat cohomology.
However, he used the center $Z(G)$ in his construction, and therefore his construction
is functorial only with respect to {\em normal} homomorphisms $\varphi\colon G\to H$,
that is, homomorphisms such that $\im\varphi$ is normal in $H$.
Indeed, a normal homomorphism $\varphi$ (but not an arbitrary homomorphism)
induces a homomorphism of centers $\varphi_Z\colon Z(G)\to Z(H)$.

Below we construct the abelianization map $\ab$ over a field $F$ of arbitrary characteristic
in terms of Galois cohomology rather than flat cohomology.
Our construction is functorial with respect to arbitrary homomorphisms.

Since $(G_\ssc(F^s)\to G(F^s),\, \theta)$ is  a $\G$-equivariant crossed module,
one can define the Galois (hyper)cohomology pointed set
\[ H^1(F,G_\ssc\to G,\,\theta)\coloneqq H^1(\G,\, G_\ssc(F^s)\to G(F^s),\,\theta);\]
see \cite[\S3.3.2]{Brv98}, or Noohi  \cite[\S4]{Noohi11}, or Appendix \ref{app:crossed} below.
Using the symmetric braiding $\{\,,\}$, one can define
a structure of abelian group on this pointed set;
see \cite[Corollaries 4.2 and 4.5]{Noohi11}.
Thus our stable crossed module of algebraic groups $G_\ab$
defines an abelian group
\[H^1(F,G_\ab)\coloneqq   H^1\big(\G,\, G_\ssc(F^s)\to G(F^s),\,\theta, \{\,,\}\big)\]
Using \eqref{e:phi-ab}, we see that a homomorphism of reductive $F$-groups
$\varphi\colon G\to H$ induces a morphism of stable crossed modules
\[\varphi_\ab\colon G_\ab\to H_\ab\]
and a homomorphism of abelian groups
\[\varphi_\ab^1\colon H^1(F,G_\ab) \to H^1(F,H_\ab).\]
Thus
\[G\hs\rightsquigarrow\hs H^1(F,G_\ab)\coloneqq H^1\big(F, G_\ssc\to G,\, \theta, \{\,,\}\big)\]
is a functor from the category of reductive $F$-groups to the category of abelian groups.

Let $T\subseteq G$ be a maximal torus, and set $T_\ssc=\rho^{-1}(T)\subseteq G_\ssc$.
We regard the complex $(T_\ssc\to T)$ as a stable crossed module
$\big(T_\ssc\to T,\, \theta_T, \{\,,\}_T\big)$
with trivial action $\theta_T$ and trivial braiding  $\{\,,\}_T$.
The embedding of stable crossed modules
\[ \big(T_\ssc\to T,\, \theta_T, \{\,,\}_T\big)\,\into\, \big(G_\ssc\to G,\, \theta_G, \{\,,\}_G\big)\]
induces a homomorphism of abelian groups
\[H^1(F, T_\ssc\to T)\to H^1(F, G_\ab).\]
Rosengarten showed in Appendix B to \cite{Borovoi23-CR} that for the quasi-isomorphism
\[ (T_\ssc\to T)\,\into\, (G_\ssc\to G)\]
in the category of $F$-group schemes,
the induced morphism of $\G$-equivariant crossed modules
\[ \big(T_\ssc(F^s)\into T(F^s)\big)\,\into\, \big(G_\ssc(F^s)\to G(F^s)\big)\]
is a quasi-isomorphism in the category of $\G$-groups
(even in positive characteristic when $\ker [T_\ssc\to T]$ is not smooth).
It follows that the induced homomorphism of abelian groups
\[H^1_\ab(F,G)\coloneqq H^1(F, T_\ssc\to T)\,\lra\, H^1\big(F, G_\ssc\to G,\, \theta,\{\,,\}\big)\eqqcolon H^1(F,G_\ab)\]
is an isomorphism; see \cite[Proposition 5.6]{Noohi11}.
Thus we can identify the abelian group $H^1(F,G_\ab)\coloneqq H^1\big(F, G_\ssc\to G,\, \theta,\{\,,\}\big)$
with $H^1_\ab(F,G)\coloneqq H^1(F, T_\ssc\to T)$.

For a reductive $F$-group $G$, consider the morphism of crossed modules
\[ (1\to G)\,\to\, (G_\ssc\to G).\]
It induces a morphism of pointed sets
\begin{align*}
\ab\colon\, H^1(F,G)=&H^1(F,1\to G)\,\to\,H^1(F, G_\ssc\to G)\\
= &H^1\big(F, G_\ssc\to G,\, \theta,\{\,,\}\big)\eqqcolon H^1(F,G_\ab)=H^1_\ab(F,G).
\end{align*}
This is our abelianization map;
see \cite[Appendix]{Borovoi93} for an explicit formula in terms of cocycles.
From the short exact sequence of complexes of reductive $F$-groups
\[ 1\to\, (1\to G)\,\to (G_\ssc\to G)\,\to\, (G_\ssc\to 1)\,\to 1,\]
in which  $(1\to G)$ and $(G_\ssc\to G)$ are crossed modules, but $(G_\ssc\to 1)$ is not,
we obtain an exact sequence of pointed sets
\begin{equation}\label{e:Gsc}
H^1(F,G_\ssc)\labelto{\rho_*}  H^1(F,G)\labelto\ab H^1_\tx{ab}(F,G);
\end{equation}
see \cite[(3.10.1)]{Brv98}.

\section{Computation in terms of $\pi_1(G)$}

Let $G$ be a reductive group over a local or global field $F$.
Theorem \ref{thm:tnloc} and Proposition \ref{p:R} compute $H^n_\ab(F,G)$ for $n\ge 1$
when $F$ is a local field.
Theorem \ref{thm:tnglob} computes $H^n_\ab(F,G)$ for $n\ge 1$ when $F$ is a global field.

\begin{theorem}[\hs{\cite[Corollary 5.4.1 and Theorem 5.11]{Brv98}, \cite[Theorem 5.8(i)]{GA12}}\hs]
\label{t:Brv-GA}
Let $G$ be a reductive group over a local or global field $F$. Then:
\begin{enumerate}
\item[\rm(1)] If $F$ is a non-archimedean local field  or a global  function field, then
the abelianization map $\ab\colon H^1(F,G)\to H^1_\ab(F,G)$ bijective.

\item[\rm(2)] If $F$ is a number field, then the following commutative diagram
\[
\xymatrix@C=15mm{
H^1(F,G)\ar[r]^-\ab\ar[d]_-{\tx{loc}_\infty}             & H^1_\tx{ab}(F,G)\ar[d]^-{\tx{loc}_\infty} \\
\prod\limits_{v|\infty} H^1(F_v,G) \ar[r]^-{\prod_{v|\infty} \ab_v}   &\prod\limits_{v|\infty} H^1_\tx{ab}(F_v,G)
}
\]
is Cartesian and all arrows in it are surjective.
Here $\prod_{v|\infty}$ denotes $\prod_{v\in \V_\infty(F)}$.
\end{enumerate}
\end{theorem}

Theorem \ref{t:Brv-GA} was obtained by the first-named author
in characteristic 0
and by Gonz\'alez-Avil\'es in positive characteristic.
The following theorem computes $H^1(F,G)$ for a reductive group $G$
over a non-archimedean local field.

\begin{theorem} \label{thm:nal}
Let $G$ be a reductive group over a non-archimedean local field $F$ with fundamental group $M=\pi_1(G)$.
Then:
\begin{enumerate}
\item[\rm(1)] There is a functorial  bijection $H^1(F,G)\isoto H^1_\tx{ab}(F,G)$
and a functorial isomorphism $ H^1_\tx{ab}(F,G)\isoto M_\Gt$\hs, whence we obtain a functorial bijection
 \[H^1(F,G)\isoto M_\Gt.\]
\item[\rm(2)] There is a functorial isomorphism
\[ H^2_\ab(F,G)\isoto (M\otimes \Q/\Z)_\G=M_\G\otimes \Q/\Z.\]
\item[\rm(3)] $H^i_\ab(F,G)=0$ for $i\ge 3$.
\end{enumerate}
\end{theorem}

This result is  known in characteristic 0; see \cite[Corollary 5.5(i) and Proposition 4.1]{Brv98}.
The bijection  \,$H^1(F,G)\isoto (M_\G)_\tor$ \,goes back to Kottwitz \cite[Proposition 6.4]{Kot84} and  \cite[Theorem 1.2]{Kot86}.

\begin{proof}
Observe that $H^i_\ab(F,G)=H^i(F,T_\ssc\to T)$ and that $H^i(F,T_\ssc\to T)$
for $i\ge 1$ was computed in terms of $M$ in Theorem \ref{thm:tnloc}.
Now assertions (2) and (3) follow from assertions (2) and (3) of Theorem \ref{thm:tnloc},
and assertion (1)  follows from Theorem \ref{t:Brv-GA}(1) and  Theorem \ref{thm:tnloc}(1).
\end{proof}

\begin{remark}
Let $G$ be a  simply connected semisimple group over a non-archime\-dean local field $F$.
Then $M=0$, and Theorem \ref{thm:nal}(1) says that $H^1(F,G)=1$.
This is a classical result of Kneser \cite{Kneser65-I}, \cite{Kneser65-II},
and of Bruhat and Tits \cite{BT3};
see also Platonov and Rapinchuk \cite[Chapter 6]{PR94}.
We do not give a new proof of this result; we use it in our proof of
Theorem \ref{thm:nal} via Corollary 5.4.1 of \cite{Brv98}.
\end{remark}

\begin{theorem}\label{t:H-abelian}
Let $G$ be a reductive group over a global field $F$. Write $M=\pi_1(G)$.
With the notation of Definition \ref{d:F1}  there are functorial isomorphisms of abelian groups:
\begin{align}
&H^1_\ab(F,G)\isoto \cF^1(M/F);\tag1\\
&H^2_\ab(F,G)\isoto\cF^2(M/F);\tag2\\
&H^i_\ab(F,G)\isoto\!\!\!\!\!\prod_{v\in V_\R(F)}\!\!\!H^{i-2}(F_v, M)\quad\text{for}\ \, i\ge 3.\tag3
\end{align}
\end{theorem}

\begin{proof}
Let $T\subset G$ be a maximal torus;
then  $H^i_\ab(F,G)= H^i(F, T_\ssc\to T)$ and $M=\coker[X_*(T_\ssc)\to X_*(T)]$.
Now we obtain the theorem from Theorem \ref{thm:tnglob}.
\end{proof}

\begin{lemma}\label{l:V_f}
Let $G$ be a reductive group over a global field $F$ with fundamental group $M=\pi_1(G)$,
and let $E/F$ be as in \S\ref{ss:E/F}.
Then there are functorial isomorphisms
\begin{align}\tag{1}
\big(M[V_f(E)]_0\big)_\Gt\,\isoto\,
   &\ker\Big[H^1_\ab(F,G)\to\prod_{v|\infty} H^1_\ab(F_v,G)\Big],\\
\tag{2}
\big(M[V_f(E)]_0\big)_\G\otimes\Q/\Z\,\isoto\,
    &\ker\Big[H^2_\ab(F,G)\to\prod_{v|\infty} H^2_\ab(F_v,G)\Big].
\end{align}
\end{lemma}

\begin{proof}
We construct the isomorphism (1).
By Theorem \ref{t:H-abelian} we may and will identify $H^1_\ab(F,G)=\cF^1(M)$.
From the definition of $\cF^1(M)$ we see that
\begin{align*}
\ker\big[\cF^1(M)\to \prod_{v|\infty}H^1_\ab(F_v,G)\big]\,\cong\, &\ker\Big[ (M[V_E]_0)_\Gt\to \prod_{v|\infty} M_\Gvt\Big]\\
      =\,&\ker\Big[(M[V_E]_0)_\Gt\to M[V_\infty(E)]_\Gt\Big].
\end{align*}
It remains to show that the natural homomorphism
\[\big(M[V_f(E)]_0\big)_\Gt\, \lra\ \ker\hm\Big[ (M[V_E]_0)_\Gt\!\!\to\! M[V_\infty(E)]_\Gt \Big]\]
is an isomorphism.
This follows from Lemma  \ref{l:spl} below.
The isomorphism (2) can be constructed similarly, again using Theorem \ref{t:H-abelian} and Lemma \ref{l:spl}.
\end{proof}

\begin{lemma}\label{l:spl}
The exact sequence
\[0\to M[V_f(E)]_0\to M[V_E]_0\to M[V_\infty(E)] \]
admits a $\G$-equivariant splitting
\[\phi_*\colon M[V_\infty(E)]\to M[V_E]_0\hs,\]
and therefore
\[ M[V_E]_0\,=\, M[V_f(E)]_0\oplus \phi_*\big(M[V_\infty(E)]\big).\]
\end{lemma}

\begin{proof}
Since  for each $w\in V_\infty(E)$ the decomposition group $\G(w)$ is cyclic,
by the Chebotarev density theorem there exists a $\G$-equivariant map
\[\phi\colon V_\infty(E)\to V_f(E).\]
We define a splitting  by
\[\phi_*\colon\! \sum_{w\in V_\infty(E)}\hskip-3mm  m_w\cdot w \,
\ \longmapsto\, \hskip-2mm
\sum_{w\in V_\infty(E)}\hskip-3mm m_w\cdot \big(w-\phi(w)\big).\qedhere
\]
\end{proof}

\begin{corollary}\label{c:tot-imaginary}
If $F$ is a global field without real places (that is, a function field or a  totally imaginary number field),
then there are functorial isomorphisms
\begin{align}\tag{1}
&\big(M[V_f(E)]_0\big)_\Gt\,\isoto\, H^1_\ab(F,G),\\
\tag{2}
&\big(M[V_f(E)]_0\big)_\G\otimes\Q/\Z\,\isoto\, H^2_\ab(F,G),
\end{align}
and a functorial bijection
\begin{equation}\tag{3}
H^1(F,G)\isoto \big(M[V_f(E)]_0\big)_\Gt\hs.
\end{equation}
\end{corollary}

\begin{proof}
Indeed, then for all $v\in V_\infty(F)$ we have
$H^1(F_v,G)=1$,  \,$H^1_\ab(F_v,G)=0$, \, and \,$H^2_\ab(F_v,G)=0$.
We obtain (1) and (2) from Lemma \ref{l:V_f}, and we obtain (3) from (1) and Theorem \ref{t:Brv-GA}.
\end{proof}

Recall that an abelian group $A$ is called {\em divisible}
if for each element $a\in A$ and for each positive integer $n$,
there exists an element $a'\in A$ such that $na'=a$.

\begin{corollary}\label{l:divisible}
Let $G$ be a reductive group over a global field $F$ without real places.
Then the group $H^2_\ab(F,G)$ is divisible.
\end{corollary}

\begin{proof}
Write $M=\pi_1(G)$.
By Corollary \ref{c:tot-imaginary}(2) we have
\[ H^2_\ab(F,G)\cong \big(M[\Vbar\!\!_f]_0\big)_\G\otimes \Q/\Z,\]
which is clearly divisible.
\end{proof}

\begin{definition}\label{d:H1}
Let $G$ be a reductive group over a global field $F$, and write $M=\pi_1(G)$.
For each infinite place $v$ of $F$ consider the following map:
\[\psi_v\colon H^1(F_v,G)\labelto\ab H^1_\tx{ab}(F_v,G)\cong H^{-1}(\G\sub \vbar,M)\into M_{\G\sub \vbar,\tor}.\]
Let $\cH^1(G/F)$, or just $\cH^1(G)$ when $F$ is understood, be the fiber product of
\[\xymatrix@1@C=10mm{
(M[\Vbar]_0)_{\G,{\Tors}}\,\ar[r]^-{l_\infty}
&\,\prod_{v|\infty} M_{\G\sub \vbar,\Tors}\,   &\,\prod_{v|\infty} H^1(F_v, G)\ar[l]_{\psi_\infty}
}\]
where the  arrow at left is the product of the localization homomorphisms $l_v$
and the  arrow at right is the product of the maps $\psi_v$\hs.
\end{definition}

Recall that by Definition \ref{d:F1}, the abelian group $\cF^1(M)$ is defined to be the fiber product of
$(M[\Vbar]_0)_\Gt$ and $\prod_{v|\infty}H^{-1}(\G(\vbar), M)$ over  $\prod_{v|\infty} M_{\G\sub \vbar,\Tors}$.
For each $v\in V_\infty(F)$, consider the composite map
\[ H^1(F_v,G)\labelto{\ab} H^1_\ab(F_v,G)=H^1(F_v, T^\ssc\to T)\isoto H^{-1}(F_v,M)\]
where the right-hand arrow is the Tate--Nakayama isomorphism.
We obtain a commutative diagram
\[
\xymatrix{
(M[\Vbar]_0)_{\G,{\Tors}}\,\ar[r]^-{l_\infty}\ar[d]_-\id
&\,\prod_{v|\infty} M_{\G\sub \vbar,\Tors}\,\ar[d]^-\id   &\,\prod_{v|\infty} H^1(F_v, G)\ar[l]_{\psi_\infty}\ar[d]\\
(M[\Vbar]_0)_{\G,{\Tors}}\,\ar[r]^-{l_\infty}
&\,\prod_{v|\infty} M_{\G\sub \vbar,\Tors}\,   &\,\prod_{v|\infty} H^{-1}(\G(\vbar),M)\ar[l]
}
\]
which induces a functorial map
\begin{equation}\label{e:H-F}
\abcal\colon \cH^1(G)\to \cF^1(M)\quad \text{where}\ \, M=\pi_1(G).
\end{equation}

We now construct a map
\begin{equation}\label{e:H1-H1}
H^1(F,G)\,\to\,\cH^1(G)
\end{equation}
by describing two maps from $H^1(F,G)$:
a map into $(M[\Vbar]_0)_{\G,{\Tors}}$
and a map onto $\prod_{v|\infty} H^1(F_v, G)$.
The first map is
\[H^1(F,G)\labelto\ab H^1_\tx{ab}(F,G)\isoto \mc{F}^1(M)\into (M[\Vbar]_0)_{\G,{\Tors}}\hs;\]
compare Remark \ref{r:injective}.
The second map is the product of the localization maps
 $$\tx{loc}_v\colon H^1(F,G)\to H^1(F_v,G)$$
 over the infinite places $v\in \V_\infty(F)$.

\begin{theorem} \label{thm:main}
Let $G$ be a reductive group over a global field $F$, and write $M=\pi_1(G)$.
Then the map \eqref{e:H1-H1} is a functorial bijection
between the pointed  sets $H^1(F,G)$ and $\cH^1(G)$.
\end{theorem}

\begin{proof}
In the diagram
\[
\xymatrix@C=8mm{
H^1(F,G)\ar[r]\ar[d]_-{\tx{loc}_\infty}             & H^1_\tx{ab}(F,G)\ar[d]^-{\tx{loc}_\infty}
       \ar[r] &\big(M[\Vbar]_0\big)_\Gt\hskip-5mm \ar[d]^-{l_\infty} \\
\prod\limits_{v|\infty} H^1(F_v,G) \ar[r]   &\prod\limits_{v|\infty} H^1_\tx{ab}(F_v,G)
        \ar[r] &\prod\limits_{v|\infty}M_{\G\sub \vbar,\Tors}\hskip -2mm
}
\]
the inner rectangle at left is Cartesian by Theorem \ref{t:Brv-GA}(2),
and the inner rectangle at right is Cartesian by Theorems
\ref{thm:tnglob}(1) and \ref{t:H-abelian}(1) and Proposition \ref{p:R};
compare Definition \ref{d:F1}.
By Lemma \ref{l:MacLane} the outer rectangle is Cartesian as well,
which completes the proof of the theorem.
\end{proof}

\begin{remark}
If $F$ is a global field without real places,
then by Theorem \ref{t:Brv-GA}(2) we have a functorial bijection $H^1(F,G)\isoto H^1_\ab(F,G)$,
and by Corollary \ref{c:tot-imaginary}(3) we obtain a functorial bijection
$H^1(F,G)\isoto \big(M[\,\ov V\!_f]_0\big)_\Gt$
where $\ov V\!_f$ is the set of finite places in $\Vbar$.
In particular, when $F$ is a function field, we have a bijection $H^1(F,G)\isoto \big(M[\Vbar]_0\big)_\Gt$\hs.
\end{remark}

\begin{remark}\label{r:HPs5}
When $F$ is a global field and $G$ is a simply connected semisimple $F$-group,
we have $M=0$ and $\cH^1(G)=\prod_{v|\infty} H^1(F_v,G)$.
Then  Theorem \ref{thm:main} gives a bijection
\[ H^1(F,G)\isoto \prod\limits_{v|\infty} H^1(F_v,G),\]
which is the Hasse principle of Kneser \cite{Kneser69},
Harder \cite{Harder65}, \cite{Harder66}, \cite{Harder75}, and Chernousov \cite{Chernousov89};
see also Platonov and Rapinchuk \cite[Chapter 6]{PR94}.
We do not give a new proof of this result; we use it via Theorem \ref{t:Brv-GA}(2)
in our proof of Theorem \ref{thm:main}.
\end{remark}

\begin{proposition}\label{p:H-F}
The following diagram commutes:
 \[
\xymatrix{
H^1(F,G)\ar[r]\ar[d]_-{\ab}\ar[r]^-\sim 	&\cH^1(G)\ar[d]^-\abcal\\
H^1_\ab(F,G)\ar[r]\ar[r]^-\sim	            &\cF^1(M)
}
\]
where $\abcal$ is the map \eqref{e:H-F}.
\end{proposition}

\begin{proof}
We have commutative diagrams
\begin{equation}\label{e:slanted-1}
\begin{aligned}
\xymatrix{
H^1(F,G)\ar[dr]\ar[ddr] \ar[rrd] \\
&\cH^1(G)\ar[r] \ar[d]     &\prod_{v|\infty} H^1(F_v,G)\ar[d]\\
&\big(M[\Vbar]_0\big)_\Gt\ar[r]   &\prod_{v|\infty} M_{\G(\vbar),\Tors}
}
\end{aligned}
\end{equation}

\begin{equation}\label{e:slanted-2}
\begin{aligned}
\xymatrix{
H^1_\ab(F,G)\ar[dr]\ar[ddr] \ar[rrd] \\
&\cF^1(M)\ar[r] \ar[d]     &\prod_{v|\infty} H^1_\ab(F_v,G)\ar[d]\\
&\big(M[\Vbar]_0\big)_\Gt\ar[r]   &\prod_{v|\infty} M_{\G(\vbar),\Tors}
}
\end{aligned}
\end{equation}
where in diagram \eqref{e:slanted-2} we identify $H^{-1}(\G(\vbar),M)=H^1_\ab(F_v,G)$
using the Tate--Nakayama isomorphism of Proposition \ref{p:R}.
In these diagrams, the rectangles are Cartesian.
By construction, the slanted arrow $H^1(F,G) \to \big(M[\Vbar]_0\big)_\Gt$ in  diagram \eqref{e:slanted-1} factors via $\cF^1(M)$,
and the following diagram commutes:
\[
\xymatrix{
H^1(F,G) \ar[r]\ar[d]_-\ab  & \big(M[\Vbar]_0\big)_\Gt\hskip-5mm\ar[d]^-\id\\
H^1_\ab(F,G) \ar[r]  & \big(M[\Vbar]_0\big)_\Gt\hskip-5mm
}
\]
Moreover, the slanted arrow $H^1(F,G) \to \prod_{v|\infty} H^1(F_v,G)$ in  diagram \eqref{e:slanted-1}
fits into the following commutative diagram:
\[
\xymatrix{
H^1(F,G) \ar[r]\ar[d]_-\ab  &  \prod_{v|\infty} H^1(F_v,G)\hskip-5mm\ar[d]^\ab\\
H^1_\ab(F,G) \ar[r]  &  \prod_{v|\infty} H^1_\ab(F_v,G)\hskip-5mm
}
\]
It follows that  the diagram of the proposition  commutes.
\end{proof}

\section{Compatibility with restriction}
The abelian cohomology group $H^1_\tx{ab}(F,G)$ and the  abelianization map
\[ \ab\colon H^1(F,G)\to H^1_\tx{ab}(F,G)\]
are clearly functorial in $F$.
When $F$ is a non-archimedean local field or a global function field, the abelianization map $\ab$ is bijective,
and the compatibility diagrams for $H^1(F,G)$ are similar
to the compatibility diagrams for $H^1_\tx{ab}(F,G)=H^1(F, T_\ssc\to T)$ of Chapter \ref{s:ab-coho}.
We assume that $F$ is a number field.
We describe the compatibility with restriction.

Let $G$ be a reductive group over a number field $F$, and
let $F'/F$ be a finite extension, not necessarily Galois.
We write $V_\infty(F)$ for the set of infinite places of $F$.
We define a map
\[r_{F'/F,\infty} \colon\,\prod_{v\in \V_\infty(F)}\!\!\!\! H^1(F_v,G)\,\lra\!\!\! \prod_{v'\in \V_\infty(F')}\!\!\!\! H^1(F'_{v'},G)\]
as follows.
Let $(x_v)\in\prod_{v|\infty} H^1(F_v,G)$.
Then $r_{F'/F,\infty}$ sends $(x_v)$ to some $(x'_{v'})\in\prod_{v'|\infty}  H^1(F'_{v'},G)$.
We describe $x'_{v'}$ for $v'\in \V_\infty(F')$.
If $F'_{v'}\simeq \C$, then  $x'_{v'}=1$. If $F'_{v'}\simeq \R$, then we write $v=v'|_F\in \V_\infty(F)$.
The embedding $\iota\colon F\into F'$ induces the unique isomorphism
$\iota_{v'}\colon F_v\isoto F'_{v'}$ and a bijection $\iota_{v'*}\colon H^1(F_v,G)\isoto  H^1(F'_{v'},G)$.
We set $x'_{v'}=\iota_{v'*}(x_v)$.

We also define a map
\begin{equation}\label{e:r-F'-F-infty-M}
N_\infty^M \colon\,\prod_{v\in \V_\infty(F)}\!\!\!\! M_{\G(\vbar),\Tors}\,\lra\!\!\! \prod_{v'\in \V_\infty(F')}\!\!\!\!M_{\G(\vbar'),\Tors}
\end{equation}
to be the product of the  transfer  maps
\[N_{\G(\vbar')\lmod\G(\vbar)}\colon  M_{\G(\vbar),\Tors}\to M_{\G(\vbar'),\Tors}\hs.\]
Such a transfer map $N_{\G(\vbar')\lmod\G(\vbar)}$ is an isomorphism when the place $v'$ is real, and is trivial when $v'$ is complex.

\begin{lemma}\label{l:F'/F}
Let $F'$ be a finite extension of a number field $F$, not necessarily Galois.
With the notation of \S\ref{ss:com-res}, the following diagrams commute:
\[
\xymatrix@R=9.8mm{
H^1(F,G) \ar[r] \ar[d]_-{\tx{Res}}   &(M[\Vbar]_0)_\G\ar[d]^{N_{\G'\backslash\G}}\\
H^1(F',G)\ar[r]                      &(M[\Vbar]_0)_{\G'}
}
\qquad\quad
\xymatrix@C=12mm@R=7mm{
H^1(F,G) \ar[r]^-{\tx{loc}_\infty}\ar[d]_-{\tx{Res}}
     &\!\!\!\!\prod\limits_{v\in \V_\infty(F)}\hskip -4mm H^1(F_v,G)   \hskip-15mm\ar[d]^-{r_{F'/F,\infty}}  \\
H^1(F',G) \ar[r]^-{\tx{loc}_\infty}
     &\!\!\!\!\prod\limits_{v'\in \V_\infty(F')}\hskip-5mm H^1(F_{v'},G)\hskip-15mm
}
\]
in which the horizontal arrows in the diagram at left are defined as follows:
\[ H^1(F,G)\labelto\ab H^1_\tx{ab}(F,G)\isoto \mc{F}^1(M/F)\into (M[\Vbar]_0)_\G, \]
and similarly for $F'$ instead of $F$, where we write $\G'=\G\sub{F'}$.
\end{lemma}

\begin{proof}
We have a diagram
\[
\xymatrix{
H^1(F,G) \ar[r] \ar[d]^-{\tx{Res}}
& H^1_\tx{ab}(F,G)\ar[r]^-\sim \ar[d]^-{\tx{Res}}   & \mc{F}^1(M/F)\ar@{^(->}[r] \ar[d]^{N_{\G'\backslash\G}}
&(M[\Vbar]_0)_\G\ar[d]^{N_{\G'\backslash\G}}\\
H^1(F',G)\ar[r]
& H^1_\tx{ab}(F',G)\ar[r]^-\sim   & \mc{F}^1(M/F')\ar@{^(->}[r]
&(M[\Vbar]_0)_{\G'}
}
\]
in which the left-hand rectangle commutes because the abelianization map is functorial in $F$,
the middle rectangle commutes by Proposition \ref{p:Res-gl}(3), and the right-hand one clearly commutes.
Thus the diagram at left in the lemma commutes.

When $F'_{v'}\simeq \R$, we have commutative diagrams
\[
\xymatrix{
F\ar[r]\ar@{^(->}[d] &F_v\ar[d]_-\cong^-{\iota_{v'}}\\
F'\ar[r]             &F'_{v'}
}
\qquad\quad
\xymatrix{
H^1(F,G)\ar[r]^-{\tx{loc}_v}\ar[d]_-{\tx{Res}}    & H^1(F_v,G)\ar[d]_-\cong^-{\iota_{v'*}}\\
H^1(F',G)\ar[r]^-{\tx{loc}_{v'}}                  & H^1(F'_{v'},G)
}
\]
which show that the diagram at right in the lemma commutes.
\end{proof}

\begin{corollary}\label{c:H1-res}
Let  $r_{F'/F}\colon \cH^1(G/F)\to \cH^1(G/F')$ denote the fiber product of the maps
\[ N_{\G'\backslash\G}\colon(M[\Vbar]_0)_\G\to (M[\Vbar]_0)_{\G'}\ \, \text{and}\ \,
   r_{F'/F,\infty}\colon\!\!\!\! \prod_{v\in \V_\infty(F)}\!\!\!\! H^1(F_v,G)\to\!\!\!\!\!
   \prod_{v'\in \V_\infty(F')}\!\!\!\!\! H^1(F_{v'},G)\]
over the map \eqref{e:r-F'-F-infty-M}.
Then the following diagram commutes:
\[
\xymatrix{
H^1(F,G)\ar[r]^\sim\ar[d]_-{\tx{Res}} &\cH^1(G/F)\ar[d]^-{r_{F'/F}}\\
H^1(F',G)\ar[r]^\sim &\cH^1(G/F')
}
\]
\end{corollary}

\section{Compatibility with localization}

We assume that $F$ is a number field.

\begin{proposition}\label{p:loc}
For any place $v$ of $F$, the following diagram commutes:
\[\xymatrix{
H^1(F,G)\ar[r]\ar[d]_-{\loc_v}\ar[r]^-\sim	&\cH^1(G)\ar[d]   \\
H^1(F_v,G)\ar[r]  	&M_{\G\sss{\vbar},{\Tors}}
}\]
in which the right-hand vertical arrow is the composition
\[ \cH^1(G)\to \big(M[\Vbar]_0\big)_\Gt\labelto{l_\vbar} M_{\G\sub\vbar,\Tors}\hs.\]
\end{proposition}

\begin{proof}
Consider the diagram
\[
\xymatrix@R=7mm{
&H^1(F,G)\ar[r]\ar[d]^-{\ab}\ar[r]^-\sim\ar[ldd]_-{\loc_v}	&\cH^1(G)\ar[d]_-\abcal\ar[rd]   \\
&H^1_\ab(F,G)\ar[r]\ar[d]^-{\loc_v}\ar[r]^-\sim	&\cF^1(M)\ar[d]\ar[r] &\big(M[\Vbar]_0\big)_\Gt\ar[ld]^-{l_\vbar}  \\
 H^1(F_v,G)\ar[r]^-{\ab_v} &H^1_\ab(F_v,G)\ar[r]  	&M_{\G\sss{\vbar},{\Tors}}
}
\]
In this diagram,
the top rectangle commutes by Proposition \ref{p:H-F},
the bottom rectangle commutes by Proposition \ref{p:loc2},
and the triangles clearly commute.
It follows that the diagram of the proposition commutes.
\end{proof}

\section{Compatibility with connecting maps}

Let $F$ be a number field.
We can use $H^i_\ab$ in order to extend a cohomology exact sequence.
For the short exact sequence of reductive $F$-groups \eqref{e:G'-G-G''} in \S\ref{ss:fund-gp},
consider the following commutative diagram with exact rows
\[
\xymatrix@R=6mm@C=7mm{
H^1\hs G'\ar[r]^{\iota_*}\ar[d]^-{\ab'}             &H^1\hs G \ar[r]^{\vk_*}\ar[d]^-{\ab}            &H^1\hs G''\ar[d]^-{\ab''}\\
H^1_\ab\hs G'\ar[r]^-{\iota^1_\ab}                  &H^1_\ab \hs G\ar[r]^-{\vk^1_\ab}                &H^1_\ab \hs G''\ar[r]^{\delta^1_\ab}
&H^2_\ab\hs G'\ar[r]^-{\iota^2_\ab}                 &H^2_\ab \hs G\ar[r]^-{\vk^2_\ab}             &H^2_\ab \hs G'' \ar[r]^-{\delta^2_\ab} &\dots
}
\]
where we write $H^1\hs G$ for $H^1(F,G)$ and so on.
In this diagram, the bottom row is the long exact sequence of Lemma \ref{l:ab-exact-long}.
We define a composite map
\[\delta^1=\delta^1_\ab\circ \ab''\colon\  H^1\hs G''\to H^1_\ab\hs G''\to H^2_\ab\hs G'.\]
Then the following infinite sequence of pointed sets and abelian groups is exact:
\begin{multline}\label{e:delta-1-ab}
H^1\hs G'\labelto{\iota_*^1} H^1\hs G \labelto{\vk_*^1}  H^1\hs G''\\
\labelto{\delta^1} H^2_\ab\hs G'\labelto{\iota^2_\ab}   H^2_\ab \hs G \labelto{\vk^2_\ab}
H^2_\ab \hs G''\\
\labelto{\delta^2_\ab} H^3_\ab\hs G'\labelto{\iota^3_\ab}   H^3_\ab \hs G \labelto{\vk^3_\ab}
H^3_\ab \hs G''\labelto{\delta^3_\ab}\dots
\end{multline}
Indeed, it is clearly exact at $H^1\hs G$ and after  $H^2_\ab\hs G'$.
Since $F$ is a number field, this sequence is exact at $H^1\hs G''$ and at $H^2_\ab\hs G'$ by \cite[Proposition 5.8]{Brv98}.

We compute the sequence \eqref{e:delta-1-ab}.
Consider the following infinite sequence:
\begin{multline}\label{e:explicit-exact}
\cH^1(G')\labelto{\iota_*^1} \cH^1(G)\labelto{\vk_*^1} \cH^1(G'')
\labelto{\delta_\cH^1}\cF^2(M')\labelto{\iota_*^2}\cF^2(M)\labelto{\vk_*^2}\cF^2(M'')\\
\labelto{\delta_\cF^2} \prod_{v|\infty} H^1(\G\sss{\vbar },M')
    \labelto{\iota_*^3} \prod_{v|\infty} H^1(\G\sss{\vbar },M)\labelto{\vk_*^3}
    \prod_{v|\infty} H^1(\G\sss{\vbar },M'')
    \labelto{\delta_\infty^3}\dots
\end{multline}
In this sequence, the connecting maps $\delta_\cH^1$, $\delta_\cF^2$,
and $\delta_\infty^n$ for $n>2$ are defined as follows:
\begin{itemize}
\item The connecting map $\delta^1_\cH$ is the composite map
\[\delta^1_\cH\colon \cH^1(G'')\labelto \abcal  \cF^1(M'')\labelto{\delta^1_\cF} \cF^2(M')\]
where  $\abcal$ is the map \eqref{e:H-F},
and the connecting homomorphism $\delta^1_\cF$
is the top horizontal arrow in \eqref{e:delta-1-F}.

\item The connecting homomorphism $\delta^2_\cF$ is the composite homomorphism
\[ \delta^2_\cF\colon\, \cF^2(M'')\to \prod_\infty H^0(\G\sss{\vbar },M'')
     \to \prod_\infty H^1(\G\sss{\vbar },M')\]
where the first homomorphism comes from the definition of $\cF^2(M'')$,
and the second one is the product  of the connecting homomorphisms
\[  H^0(\G\sss{\vbar }, M'')\to H^1(\G\sss{\vbar }, M')\]
coming from the short exact sequence $0\to M'\to M\to M''\to 0$
of Lemma \ref{l:G'-G-G''-pi1}.

\item The connecting homomorphism $\delta_\infty^n$ for $n\ge3$ is the direct product of
the connecting homomorphisms
\[ H^{n-2}(\G\sss{\vbar }, M'')\to  H^{n-1}(\G\sss{\vbar }, M').\]
\end{itemize}

\begin{proposition}\label{p:connecting}
For the short exact sequence \eqref{e:G'-G-G''} in  \S\ref{ss:fund-gp}, the following diagrams with bijective vertical arrows
commute:
\[ \xymatrix{
\cH^1(G'')\ar[r]^{\delta^1_\cH}\ar@{<-}[d] &\cF^2(M')\ar@{<-}[d]  &&\cF^2(M'') \ar[r]^-{\delta^2_\cF}\ar[d]
   &\prod\limits_{v|\infty}\! H^1(\G\sss{\vbar }, M')\hskip-20mm \ar[d]^-{\prod\cup\alpha_v} \\
H^1(F,G'')\ar[r]^{\delta^1}         &H^2_\ab(F, G')  &&H^2_\ab(F_\infty, G'')\ar[r]^-{\delta^2}   &H^3_\ab(F_\infty, G')\hskip-14mm
}\hskip20mm
\]
\[
\begin{aligned}
\xymatrix@R=10mm@C=15mm{
\hskip24mm\prod\limits_{v|\infty}\! H^{n-2}(\G\sss{\vbar }, M'')\ar[r]^-{\delta^n_\infty}\ar[d]^-{\prod \cup\alpha_v}
    &\prod\limits_{v|\infty}\! H^{n-1}(\G\sss{\vbar }, M')\hskip-24mm\ar[d]^-{\prod\cup\alpha_v}\\
\hskip13mm H^n_\ab(F_\infty, G'')\ar[r]^-{\delta^n}                                 &H^{n+1}_\ab(F_\infty, G')\hskip-15mm
}
\end{aligned}
 \qquad\text{\hskip 15mm for}\ \, n>2.
 \]
In the  first diagram, the vertical arrows are the  bijection of Theorem \ref{thm:main} and the isomorphism
of Theorem \ref{t:H-abelian}(2).
\end{proposition}

This proposition together with Theorem \ref{thm:main} and  Theorem \ref{thm:tnglob}
says that the exact sequence \eqref{e:delta-1-ab} is isomorphic to
the explicitly constructed sequence \eqref{e:explicit-exact}.

\begin{proof}
Consider the following diagrams:
\[
\xymatrix@R=11mm@C6mm{
\cH^1(G'')\ar[r]^-{\delta_\cH^1}\ar[d]  &\cF^2(M')\ar@{=}[d]
\\
\cF^1(M'')\ar[r]^-{\delta^1_\cF}                       &\cF^2(M')
\\
H^1_\ab(F,G'')\ar[r]^-{\delta^1}\ar[u] &H^2_\ab(F,G')\ar[u]
}\hskip -18pt
\xymatrix@R=8.2mm@C=6mm{
 \cF^2(M'')\ar[r]^-{\delta^2_\cF}\ar[d]    &\prod\limits_{v|\infty} H^1(\G\sss{\vbar }, M')\hskip-62pt\ar@{=}[d]
\\
 \hskip65pt\prod\limits_{v|\infty}  H^0(\G\sss{\vbar },M'')\ar[r]^-{\delta^0}\ar[d]^-{\prod\cup\alpha_v}
& \prod\limits_{v|\infty}  H^1(\G\sss{\vbar },M')\hskip-64pt \ar[d]^-{\prod\cup\alpha_v}
\\
 H^2_\ab(F_\infty, G'')\ar[r]^-{\delta^2_\ab}  &H^3(F_\infty, G')\hskip-10mm
}\hskip 20mm
\]
in which the unlabeled vertical arrows come
from the definitions of $\cH^1(G'')$, $\cF^1(M'')$, and $\cF^2(M'')$.
In the diagram at left, the top rectangle commutes by the definition of the map $\delta^1_\cH$,
and the bottom one commutes by Proposition \ref{p:F-connecting-global}.
In the diagram at right, the top rectangle commutes by the definition of the homomorphism $\delta^2_\cF$,
and the bottom one commutes because cup product with $\alpha_v$ commutes with the connecting homomorphisms.
This shows that the first two diagrams of the proposition commute.
The third one commutes because cup product with $\alpha_v$ commutes with the connecting homomorphisms.
\end{proof}

\chapter{Explicit computations of Galois cohomology}
\label{sec:explicit}

In this chapter, for a reductive group $G$ over a global field $F$,
we decompose the abelian groups $H^1_\ab(F,G)$ and $H^2_\ab(F,G)$
into direct sums of finite abelian groups, each of which can be explicitly computed.
Moreover, we decompose the pointed set $H^1(F,G)$ into a direct product of a finite pointed set and an abelian group,
where the abelian group is a direct sum of finite abelian groups.
Here again, the finite pointed set and the finite abelian groups can be explicitly computed.

\section{Decomposing $M[V_E]_0$ and  $M[V_E]$ }
\label{ss:M[VE]0}

Let $M$ be a finitely generated (over $\Z$) $\GF$-module where $F$ is a global field.
Let $\Tbul=(T^{-1}\to T^0)$ be as in Chapter \ref{s:ab-coho},
in particular, $M=\coker\big[X_*(T^{-1})\to X_*(T^0)\big]$.

Let $E/F$ be a finite Galois extension in $F^s$
such that $\GF$ acts on $M$ via $\G\sub{E/F}$.
We write $\G=\G\sub{E/F}$. For a place $v\in V_F$ we denote by $\V_E(v)$ the preimage
of $v$ in $V_E$; it is an orbit of $\G$ in $V_E$.
For each $v$, we choose a lift $\bv\in \V_E(v)\subset V_E$.

Let $S_F\subset V_F$ be a finite subset.
We denote by $S_E\subset V_E$ the preimage of $S_F$ in $V_E$, and we write
$S_F^\cmp=V_F\smallsetminus S_F$, $S_E^\cmp=V_E\smallsetminus S_E$.
We may and will choose $S_F$ in such a way that
\begin{equation}\label{e:wS}
\text{for each  place $v\in S_F^\cmp$ there exists a place  $w_S\in S_E$ with}\ \G(w_S)\supseteq \G(\bv)
\end{equation}
where $\G(w_S)$ and  $\G(\bv)$ are the corresponding decomposition groups.
We fix such $w_S$ for each $v\in S_F^\cmp$\hs,  and define a $\G$-equivariant map
\[\phi\colon S_E^\cmp\to S_E\]
by sending $\gamma\cdot \bv$ to $\gamma\cdot w_S$
for $\gamma\in\G$, $v\in S_F^\cmp$.
By \eqref{e:wS} the map $\phi$ is well defined.
We fix such $S_F$ and $\phi$.

We denote
\[M[S_E]=\Big\{\sum_{w\in S_E} m_w\cdot w\ \ \big|\  \   m_w\in M\Big\},\]
and we define $M\big[S_E^\cmp\big]$ similarly.
We denote
\[M[S_E]_0=M[S_E]\cap M[V_E]_0= \Big\{\sum m_w\cdot w\in M[S_E]\ \ \Big|\ \ \sum m_w=0\Big\}.\]
We have an exact sequence
\[0\to M[S_E]_0\to M[V_E]_0\to M\big[S_E^\cmp\big],\]
and we construct a splitting
\begin{equation*}
\phitil\colon M\big[S_E^\cmp\big]\to M[V_E]_0,\quad
\sum_{w\in S_E^\cmp}\!\! m_w\cdot w\ \,\longmapsto\! \sum_{w\in S_E^\cmp}\!\! \big(m_w\cdot w-m_w\cdot \phi(w)\big).
\end{equation*}
We see that the sequence
\[0\to M[S_E]_0\to M[V_E]_0\to M\big[S_E^\cmp\big]\to 0\]
is exact.
We obtain a decomposition into a direct sum
\begin{equation}\label{e:direct-phi*}
M[V_E]_0=M[S_E]_0\oplus\phitil\big(M\big[S_E^\cmp\big]\big)
\end{equation}
and mutually inverse isomorphisms
\begin{equation*}
 M[S_E]_0\oplus M\big[S_E^\cmp\big]\hs\overset\sim\longleftrightarrow\hs M[V_E]_0
\end{equation*}
fitting into  commutative diagram
\begin{equation}\label{e:VSC1}
\begin{aligned}
\xymatrix@C=12mm{
M[S_E]_0\oplus M\big[S_E^\cmp\big]\ar[r]_-\sim^-{i_0\oplus\phitil}\ar@{^(->}[d]    & M[V_E]_0 \ar@{^(->}[d]\\
M[S_E]\oplus M\big[S_E^\cmp\big]\ar[r]_-\sim^-{i\oplus\phitil}  & M[V_E]
}
\end{aligned}
\end{equation}
In this diagram,
the maps $i_0\colon M[S_E]_0\into M[V_E]_0$ and $i\colon M[S_E]\into M[V_E]$
are the inclusion maps, and both horizontal arrows are the homomorphisms given by the formula
\begin{equation}\label{e:i-phi}
\Big(\sum_{w\in S_E}m_w\cdot w,  \sum_{w\in S_E^\cmp}m_w\cdot w  \Big)\,\longmapsto
\sum_{w\in S_E}m_w\cdot w- \sum_{w\in S_E^\cmp}m_w\cdot \phi(w) + \sum_{w\in S_E^\cmp}m_w\cdot w.
\end{equation}
Both corresponding inverse maps
are given by the formula
\begin{equation}\label{e:i+phi}
\sum_{w\in V_E}\hskip-1mm m_w\cdot w\,\longmapsto\,
    \Big(\sum_{w\in S_E}\hskip-1mm m_w\cdot w+\!\!\sum_{w\in S_E^\cmp}\hskip-1mm m_w\cdot\phi(w)\ ,\
  \, \sum_{w\in S_E^\cmp}\hskip-1.5mm m_w\cdot w\,\Big).
\end{equation}
It is easy to see that the maps \eqref{e:i-phi} and \eqref{e:i+phi} are indeed mutually inverse,
and hence they are isomorphisms.

\begin{lemma}\label{l:AV}
With the above notation, there are commutative diagrams
in which the horizontal arrows are isomorphisms:
\begin{gather}
\tag1
\hskip16mm
\begin{aligned}\xymatrix@C=18mm{
\hskip -13mm\big(M[S_E]_0\big)_\Gt
      \hs\oplus\hs\bigoplus\limits_{v\in S_F^\cmp} M_\Gvt\ar[r]^-\sim\ar[d]_-{\loc^1_S\oplus \id}
&\big(M[V_E]_0\big)_\Gt\hskip -5mm\ar[d]^-{\loc^1}\\
\hskip -8mm M[S_E]_\Gt    \hs \oplus\hs\bigoplus\limits_{v\in S_F^\cmp} M_\Gvt\ar[r]^-\sim
&M[V_E]_\Gt\hskip-4mm
}\end{aligned}
\\
\notag
\\ \tag2
\hskip 14mm
\begin{aligned}\xymatrix@C=11mm{
 \hskip -14mm\Big(\big(M[S_E]_0\big)_\G\otimes\Q/\Z\Big)\,\, \oplus\,
     \!\!\bigoplus\limits_{v\in S_F^\cmp}\Big( M_\Gv\otimes \Q/\Z\Big)\ar[d]_-{\loc^2_S\oplus \id}
\ar[r]^-\sim &\big(M[V_E]_0\big)_\G\otimes \Q/\Z\ar[d]^-{\loc^2}\\
\hskip -8mm \Big(M[S_E]_\G\otimes\Q/\Z\Big)\,\, \oplus\, \!\!\bigoplus\limits_{v\in S_F^\cmp}\Big( M_\Gv\otimes \Q/\Z\Big)
 \ar[r]^-\sim
&M[V_E]_\G\otimes \Q/\Z
}\end{aligned}
\end{gather}
where
\begin{align*}
&\loc^1_S\colon \big(M[S_E]_0\big)_\Gt \to M[S_E]_\Gt\hs,\\
&\loc^2_S\colon\big(M[S_E]_0\big)_\G\otimes\Q/\Z \to M[S_E]_\G\otimes\Q/\Z
\end{align*}
are the homomorphisms induced by the inclusion homomorphism $M[S_E]_0\into M[S_E]$.
\end{lemma}

\begin{proof}
We obtain diagrams (1) and (2)  by applying to  diagram \eqref{e:VSC1}
the functors \ $(\,\cdot\,)_\Gt$ \  and \ $(\,\cdot\,)_\G\otimes\Q/\Z$, \  respectively,
and after that using the Shapiro isomorphisms
\begin{align*}
&M\big[S_E^\cmp\big]_\Gt=\bigoplus_{v\in S_F^\cmp}M[V_E(v)]_\Gt\isoto \bigoplus_{v\in S_F^\cmp}M_\Gvt\hs,\\
&M\big[S_E^\cmp\big]_\G\otimes\Q/\Z=\bigoplus_{v\in S_F^\cmp}M[V_E(v)]_\G\otimes\Q/\Z
    \isoto \bigoplus_{v\in S_F^\cmp}M_\Gv\otimes\Q/\Z.\qedhere
\end{align*}
\end{proof}

\begin{remark}\label{r:indep-phi}
The restrictions
\begin{align*}
&\big(M[S_E]_0\big)_\Gt\to \big(M[V_E]_0\big)_\Gt\hs,\\
&\big(M[S_E]_0\big)_\G\otimes \Q/\Z\to \big(M[V_E]_0\big)_\G\otimes \Q/\Z
\end{align*}
of the top horizontal arrows in the diagrams (1) and (2) of Lemma \ref{l:AV}
are induced by the inclusion map
$M[S_E]_0\into M[V_E]_0$
and hence do not depend on the choice of the map $\phi$.
Similarly, the restrictions of the bottom horizontal arrows in diagrams (1) and (2)
\begin{align*}
&M[S_E]_\Gt\to M[V_E]_\Gt\hs,\\
&M[S_E]_\G\otimes \Q/\Z\to M[V_E]_\G\otimes \Q/\Z
\end{align*}
are induced by the inclusion map
$M[S_E]\into M[V_E]$
and hence do not depend on the choice of the map $\phi$.
\end{remark}

\section{Explicit computation of $H^1(F,\Tbul)$ and $H^2(F,\Tbul)$}
\label{ss:explicit-F}
In this section, using the decompositions of Lemma \ref{l:AV},
we decompose $H^1(F,\Tbul)$ and $H^2(F,\Tbul)$.
To that end, with the assumptions and notation of \S\ref{ss:M[VE]0},
we assume also that $E$ has no real places, and
we choose $S_F\subset V_F$ in such a way that $S_F$ contains all archimedean places of $F$ and that
for each (finite) place $v\in S_F^\cmp$ there exists a {\em finite} place  $w_S\in S_E$
satisfying \eqref{e:wS}.
We use these places $w_S$ when constructing the map $\phi\colon S_E^\cmp\to  S_E$.

\begin{remark}\label{r:factors}
Since by construction for any $w\in S_E^\cmp$ the place  $\phi(w)$ is finite,
we see that the composite map
\[M[S_E^\cmp]\labelto{\phitil} M[V_E]_0\lra M\big[V_\infty(E)\big]\]
is the zero map, and therefore the canonical projection $M[V_E]_0\to  M\big[V_\infty(E)\big]$
factors via  the direct summand $M[S_E]_0$ in the decomposition \eqref{e:direct-phi*}.
We also have a decomposition into a direct sum
\begin{equation}\label{e:direct-phi*-bis}
M[V_E]=M[S_E]\oplus\phitil\big(M\big[S_E^\cmp\big]\big)
\end{equation}
compatible with \eqref{e:direct-phi*},
and the canonical projection $M[V_E]\to  M\big[V_\infty(E)\big]$
factors via  the direct summand $M[S_E]$ in the decomposition \eqref{e:direct-phi*-bis}.
\end{remark}

According to Definition \ref{d:F1}, the abelian group $\cF^1(M)$
is the fiber product of the homomorphisms (described in the definition)
\[\xymatrix@1@C=10mm{
\big(M[V_E]_0\big)_\Gt\,\ar[r]^-{l_\infty^1} &\,\bigoplus\limits_{v|\infty}\! M_{\Gv,\hs\Tors}\, &\,\bigoplus\limits_{v|\infty}\! H^{-1}(\G\sub{\bv},M)\ar[l]
}\]
Let $l_{\infty,S}^1$ denote the restriction of the homomorphism $l_\infty^1$
to the direct factor $\big(M[S_E]_0\big)_\Gt$ of  $(M[V_E]_0)_\Gt$\hs.
By Remark \ref{r:factors} the homomorphism $l_\infty^1$ factors via $l_{\infty,S}^1$.
 We define $\cF_{\! S}^1(M)$ to be the fiber product of the homomorphisms
\[\xymatrix@1@C=10mm{
\big(M[S_E]_0)_\Gt\,\ar[r]^-{l_{\infty,S}^1} &\bigoplus\limits_{v|\infty}\! M_{\Gv,\hs\Tors}\, &\bigoplus\limits_{v|\infty} \! H^{-1}(\G\sub{\bv},M).\ar[l]
}\]
We write $\MM^1_\infty=\bigoplus_{v|\infty} M_\Gvt$\hs, and we denote by $\FP^1$ the fiber product functor
\[(\,\cdot\,)\,\underset{\,\hs\MM^1_\infty}{\boldsymbol{\times}}\, \bigoplus_{v|\infty} H^{-1}(\G(\bv),M)\,\cong\,
(\,\cdot\,)\,\underset{\,\hs\MM^1_\infty}{\boldsymbol{\times}}\, \bigoplus_{v|\infty} H^1(F_v,\Tbul)\]
where the isomorphism comes from Proposition \ref{p:R}.
Then
\[ \cF^1(M)= \FP^1\big((M[V_E]_0)_\Gt\big)
   \quad\text{and}\quad\cF^1_S(M)=\FP^1\big((M[S_E]_0)_\Gt\big).\]

Similarly, according to Definition \ref{d:F2}, the abelian group $\cF^2(M)$
is the fiber product of the homomorphisms
\[\xymatrix@1@C=10mm{
\big(M[V_E]_0\big)_\G\otimes\Q/\Z\,\ar[r]^-{l_\infty^2} &\bigoplus\limits_{v|\infty}\! M_\Gv\otimes \Q/\Z\,
       &\bigoplus\limits_{v\in\infty}\! H^{0}(\Gv,M)\ar[l]
}\]
Let $l_{\infty,S}^2$ denote the restriction of the homomorphism $l_\infty^2$
to the direct factor $\big(M[S_E]_0\big)_\G\otimes \Q/\Z$ of  $(M[V_E]_0)_\G\otimes\Q/\Z$.
By Remark \ref{r:factors} the homomorphism $l_\infty^2$ factors via $l_{\infty,S}^2$.
We define $\cF_{\! S}^2(M)$ to be the fiber product of the homomorphisms
\[\xymatrix@1@C=10mm{
\big(M[S_E]_0\big)_\G\otimes\Q/\Z\,\ar[r]^-{l_{\infty,S}^2} &\bigoplus\limits_{v|\infty}\! M_\Gv\otimes\Q/\Z\,
      &\bigoplus\limits_{v|\infty}\! H^{0}(\Gv,M).\ar[l]
}\]
We write $\MM^2_\infty=\bigoplus_{v|\infty} M_\Gv\otimes\Q/\Z$, and we denote by $\FP^2$ the fiber product functor
\[(\,\cdot\,)\,\underset{\,\hs\MM^2_\infty}{\boldsymbol{\times}}\, \bigoplus_{v|\infty} H^0(\G(\bv),M)\,\cong\,
(\,\cdot\,)\,\underset{\,\hs\MM^2_\infty}{\boldsymbol{\times}}\, \bigoplus_{v|\infty} H^2(F_v,\Tbul)\]
where the isomorphism comes from Proposition \ref{p:R}.
Then
\[ \cF^2(M)= \FP^2\big((M[V_E]_0)_\G\otimes\Q/\Z\big) \quad\text{and}
    \quad \cF^2_S(M)=\FP^2\big(\hs(M[S_E]_0)_\G\otimes \Q/\Z\hs\big).\]

\begin{proposition}\label{p:FS}
With the above notation, the groups $H^1(F,\Tbul)$ and $H^2(F,\Tbul)$ fit into the following commutative diagrams in which the horizontal arrows are isomorphisms:
\begin{gather}
\tag{1}
\hskip4mm
\begin{aligned}
\xymatrix@C=14mm{
\hskip -0.1mm\cF^1_S(M)\hs\oplus\!\bigoplus\limits_{v\in S_F^\cmp}\! M_\Gvt\ar[d]_-{l^1_S\oplus \id}\ar[r]^-\sim
&H^1(F,\Tbul)\ar[d]^-{\loc^1}
\\
\hskip -12.7mm\bigoplus\limits_{v\in S_F}\!\! H^1(F_v,\Tbul) \hs \oplus\!\!\bigoplus\limits_{v\in S_F^\cmp}\! M_\Gvt\ar[r]^-\sim
&\bigoplus\limits_{v\in V_F}\! H^1(F_v, \Tbul)
}
\end{aligned}
\\ \notag
\\ \tag{2}
\begin{aligned}
\xymatrix@C=10mm{
\hskip 6mm\cF^2_S(M)\hs\oplus\!\!\bigoplus\limits_{v\in S_F^\cmp}\! M_\Gv\otimes\Q/\Z\ar[d]_-{l^2_S\oplus \id}\ar[r]^-\sim
&H^2(F,\Tbul)\ar[d]^-{\loc^2}
\\
\hskip-7.5mm\bigoplus\limits_{v\in S_F}\!\! H^2(F_v,\Tbul) \hs \oplus\!\!\bigoplus\limits_{v\in S_F^\cmp}\! M_\Gv\otimes\Q/\Z \ar[r]^-\sim
&\bigoplus\limits_{v\in V_F}\!\! H^2(F_v, \Tbul)
}
\end{aligned}
\end{gather}
\end{proposition}

We specify the left-hand vertical arrows in diagrams (1) and (2).
In diagram (1) in the arrow $l^1_S\oplus \id$, the map $l^1_S$ is the natural homomorphism
obtained by applying the fiber product functor $\FP^1$
to the homomorphism  $\big(M[S_E]_0\big)_\Gt\to M[S_E]_\Gt$
induced by the inclusion  homomorphism $M[S_E]_0\into M[S_E]$.
For the finite places $v\in S_F$, we use the isomorphisms $M_\Gvt\cong H^1(F_v,\Tbul)$
of Theorem \ref{thm:tnloc}(1).
Similarly, in diagram (2) the map $l^2_S$ is the natural homomorphism
obtained by applying the fiber product functor $\FP^2$
to the homomorphism  $\big(M[S_E]_0\big)_\G\otimes\Q/\Z\to M[S_E]_\G\otimes \Q/\Z$
induced by the inclusion  homomorphism $M[S_E]_0\into M[S_E]$.
For the finite places $v\in S_F$, we use the isomorphisms $(M\otimes\Q/\Z)_\Gv\cong H^1(F_v,\Tbul)$
of Theorem \ref{thm:tnloc}(2).

\begin{proof}
We construct diagram (1) from the diagram (1) of Lemma \ref{l:AV}.
From each group in that diagram, there is a canonical homomorphism to
$$\MM^1_\infty\coloneqq \bigoplus_{v|\infty} M_\Gvt\hs.$$
We apply the functor $\FP^1$
to all groups in the  diagram  (1) of Lemma \ref{l:AV}
and compute the four groups in the obtained diagram; see below.
We use Remark  \ref{r:factors} and Lemma  \ref{l:fiber-factor}
from Appendix \ref{app:fiber}.

{\em The top-left group:}
By  Remark  \ref{r:factors}  the homomorphism
\[l^1_\infty\colon \big(M[S_E]_0\big)_\Gt\hs\oplus\hs\bigoplus\limits_{v\in S_F^\cmp}\! M_\Gvt \ \lra\  \MM^1_\infty\coloneqq \bigoplus_{v|\infty} M_\Gvt\]
factors via the direct summand  $\big(M[S_E]_0\big)_\Gt$\hs, and  by Lemma \ref{l:fiber-factor}
we have a functorial isomorphism
\begin{multline*}
\FP^1\bigg(\hs(M[S_E]_0)_\Gt\oplus\hs\bigoplus\limits_{v\in S_F^\cmp}\! M_\Gvt\bigg)\\ \cong\FP^1\Big(\hs(M[S_E]_0)_\Gt\Big)
\oplus \bigoplus\limits_{v\in S_F^\cmp}\! M_\Gvt = \cF^1_S(M) \oplus \bigoplus\limits_{v\in S_F^\cmp}\! M_\Gvt\hs
\end{multline*}
where the equality follows from the definition of $\cF^1_S$.

{\em The top-right group:}
We have a functorial isomorphism
\[\FP^1\Big(\hs(M[V_E]_0)_\Gt\Big)=\cF^1(M)\isoto H^1(F,\Tbul)\]
where the equality is the definition of $\cF^1(M)$
and the isomorphism at right is inverse to that of Theorem \ref{thm:tnglob}(1).

{\em The bottom-left group:}
By  Remark  \ref{r:factors} and  Lemma \ref{l:fiber-factor}
we have  functorial isomorphisms
\begin{multline*}
\FP^1 \bigg(M[S_E]_\Gt  \hs \oplus\hs\bigoplus\limits_{v\in S_F^\cmp} M_\Gvt\bigg)
\cong \FP^1\Big( M[S_E]_\Gt\Big) \oplus \bigoplus\limits_{v\in S_F^\cmp} M_\Gvt\\
\cong\bigoplus\limits_{v\in S_F}\!\! H^1(F_v,\Tbul) \hs \oplus\hs\bigoplus\limits_{v\in S_F^\cmp} M_\Gvt\hs.
\end{multline*}

{\em The bottom-right group:}
By Lemma \ref{l:fiber-factor} we have
\begin{multline*}
 \FP^1\big([M[V_E]_\Gt\big)\cong
\FP^1\bigg(\, \bigoplus_{v\in V_\infty(F)} M_\Gvt\bigg)
\oplus \bigoplus_{v\in V_f(F)} M_\Gvt\\
\cong\bigoplus_{v\in V_\infty(F)} H^1(F_v,\Tbul)\ \oplus\bigoplus_{v\in V_f(F)} H^1(F_v,\Tbul)
=\bigoplus_{v\in V_F} H^1(F_v,\Tbul)
\end{multline*}
(we use Theorem \ref{thm:tnloc}(1)\hs).

We have constructed diagram (1) by applying the fiber product  functor $\FP^1$
to the diagram (1) of Lemma \ref{l:AV} and  using Theorem \ref{thm:tnglob}(1),
the definitions of the groups  $\cF^1_S(M)$ and $\cF^1(M)$,
and Theorem \ref{thm:tnloc}(1).
Similarly, one can construct diagram (2) by applying the fiber product  functor $\FP^2$
to the diagram (2) of Lemma \ref{l:AV} and  using Theorem \ref{thm:tnglob}(2),
the definitions of the groups $\cF^2_S(M)$ and $\cF^2(M)$, and Theorem \ref{thm:tnloc}(2).
\end{proof}

\begin{remark}\label{r:indep-phi-Tbul}
The restrictions
\[ \cF^n_S(M)\to H^n(F,\Tbul),\quad\ n=1,2\]
of the top horizontal arrows in the diagrams (1) and (2) of Proposition \ref{p:FS}
are induced by the inclusion homomorphism $M[S_E]_0\into M[V_E]_0$
and hence do not depend on the choice of the map $\phi$.
Similarly, the  restrictions
\[ \bigoplus\limits_{v\in S_F}\! H^n(F_v, \Tbul)\to \bigoplus\limits_{v\in V_F}\! H^n(F_v, \Tbul),\quad\ n=1,2\]
of the bottom horizontal arrows in the diagrams (1) and (2) of Proposition \ref{p:FS}
are induced by the inclusion $S_F\into V_F$
and hence do not depend on the choice of the map $\phi$.
\end{remark}

\section{Decomposing $H^1(F,G)$}
\label{ss:explicit-H}

Let $F$ be a global field and $G$ be a reductive $F$-group with fundamental group
$M=\pi_1(G)$. We use the notation of \S\ref{ss:M[VE]0};
in particular, we write $\G$ for $\G\sub{E/F}$.

\begin{theorem}\label{t:FS}
Let $G$ be a reductive group over a global field $F$.
Write $M=\pi_1(G)$.
Then with the assumptions and  notation of \S\ref{ss:M[VE]0} and \S\ref{ss:explicit-F},
there are isomorphisms of abelian groups:
\begin{align}
\tag1  &\cF_{\!S}^1(M)\oplus \bigoplus_{v\in S_F^\cmp} M_{\Gvt}\hs \isoto\hs H^1_\ab(F,G),\\
\tag2 &\cF_{\!S}^2(M)\oplus \bigoplus_{v\in S_F^\cmp} M_{\Gv}\otimes\Q/\Z\hs\isoto \,   H^2_\ab(F,G).
\end{align}
\end{theorem}

\begin{proof}
The theorem follows immediately from Proposition \ref{p:FS}
applied to the complex of tori $(T_\ssc\to T)$.
\end{proof}

According to Definition \ref{d:H1}, the pointed set $\cH^1(G)$ is the fiber product of the maps
\begin{equation*}
\xymatrix@1@C=10mm{
\big(M[V_E]_0)_\Gt\,\ar[r]^-{l_\infty^1} &\,\MM^1_\infty
    &\,\bigoplus\limits_{v|\infty} H^1(F_v,G).\ar[l]
}
\end{equation*}
By Theorem \ref{thm:main} we may identify  $\cH^1(G)$ with $H^1(F,G)$.
For $S_F\subset V_F$ as in \S\ref{ss:M[VE]0} and \S\ref{ss:explicit-F}, let $l_{\infty,S}^1$
denote the restriction of the homomorphism $l_\infty^1$
to the direct factor $\big(M[S_E]_0\big)_\Gt$ of  $(M[V_E]_0)_\Gt$
as in the diagram (1) of Lemma \ref{l:AV}.
By Remark \ref{r:factors} the homomorphism $l_\infty^1$ factors via $l_{\infty,S}^1$\hs.
We define $\cH_{S}^1(G)$ to be the fiber product of the maps
\begin{equation*}
\xymatrix@1@C=10mm{
\big(M[S_E]_0)_\Gt\,\ar[r]^-{l_{\infty,S}^1} &\,\MM^1_\infty
    &\,\bigoplus\limits_{v|\infty} H^{1}(F_v,G).\ar[l]
}
\end{equation*}

\begin{theorem}\label{t:explicit}
Let $G$ be a reductive group over a global field $F$
with algebraic fundamental group $M=\pi_1(G)$.
Let $E/F$ be a finite Galois extension in $F^s$
such that $\GF$ acts on $M$ via $\G\sub{E/F}$
and that $E$ has no real places. Write $\G=\G\sub{E/F}$.
 Let  a map $v\mapsto \bv\colon V_F\to V_E$, a subset $S_F\subset V_F$,
and a $\G$-equivariant map $\phi\colon S_E^\cmp\to S_E$
be as in \S\ref{ss:M[VE]0} and \S\ref{ss:explicit-F}.
Then the pointed set $H^1(F,G)\cong\cH^1(G)$  fits
into a commutative diagram of pointed sets with bijective horizontal arrows
\begin{equation*}
\begin{aligned}
\xymatrix@C=13mm{
\cH_{\!S}^1(G)\times\! \bigoplus\limits_{v\in S_F^\cmp}\!\! M_{\Gv,\Tors} \ar[r]^-\sim
     \ar[d]_-{l^1_S\hs\times\hs\id}
&H^1(F,G)\ar[d]^-{\loc^1}
\\
\hskip-14mm\bigoplus\limits_{v\in S_F} H^1(F_v,G)\times\!
      \bigoplus\limits_{v\in S_F^\cmp}\!\! M_{\Gv,\Tors} \ar[r]^-\sim
&\bigoplus\limits_{v\in V_F}\!\! H^1(F_v,G)
}\end{aligned}
\end{equation*}
in which
\[l^1_S\colon \cH^1_S(G)\to \bigoplus_{v\in S_F} H^1(F_v,G)\quad\
    \text{and}\quad\ \loc^1\colon H^1(F,G)\to \bigoplus_{v\in V_F} H^1(F_v,G) \]
are the localization maps, and  where for the non-archimedean places $v\in S_F$ we identify
$H^1(F_v,G)=M_{\Gvt}$ as in  Theorem \ref{thm:nal}(1).
\end{theorem}

\begin{proof}
Similarly to the proof of Proposition \ref{p:FS},
we obtain the diagram of the theorem  from the diagram (1) of Lemma \ref{l:AV}
by taking fiber product of each of the four groups in that diagram
with the pointed set $\prod_{v|\infty} H^1(F_v,G)$ over
$\MM^1_\infty\coloneqq\bigoplus_{v|\infty} M_\Gvt$
and  using Theorem \ref{thm:tnglob}(1), the definitions of the pointed sets $\cH^1_S(G)$ and $\cH^1(G)$,
and Theorem  \ref{thm:nal}(1).
Remark  \ref{r:factors} permits us to apply  Lemma \ref{l:fiber-factor}.
\end{proof}

We observe that the restrictions
\[ \cH_{\!S}^1(G)\to H^1(F,G)\quad\ \text{and}\quad\
   \bigoplus\limits_{v\in S_F} H^1(F_v,G)\to \bigoplus\limits_{v\in V_F} H^1(F_v,G)\]
of the horizontal arrows in the diagram of Theorem \ref{t:explicit}
are induced by the inclusion $S_F\into V_F$ and
hence do not depend on the choice of the map $\phi$.

\chapter {Computations with Tate--Shafarevich kernels}
\label{sec:Shafarevich}

In this chapter, using results of Chapter \ref{sec:explicit},
for a reductive group $G$ over a global field $F$
we compute the Tate--Shafarevich groups $\Sha^n_\ab(F,G)$ for $n=1,2$
as subgroups of $H^n_\ab(F,G)$,
and we compute the Tate--Shafarevich group $\Sha^1(F,G)$ as a subset of the pointed set $H^1(F,G)$.

\section{The Tate--Shafarevich groups $\Sha^n(M)$ and $\Sha^n_S(M)$}

Let $E/F$ be a finite Galois extension of global fields with Galois group $\G=\G(E/F)$,
and let $M$ be a finitely generated $\G$-module.
Let $S_E\subseteq V_E$ be a non-empty $\G$-invariant set of places of $E$, not necessarily finite,
and let $S_F\subseteq V_F$ denote its image in $V_F$.
 We define:

\begin{multline*}
\Sha^1_S(M)=\ker\Big[\big(M[S_E]_0\big)_\Gt\!\to M[S_E]_\Gt\Big]\\=
   \ker\Big[\big(M[S_E]_0\big)_\Gt\!\to\! \bigoplus_{v\in S_F}\!M_\Gvt\Big],
\end{multline*}\vskip-5mm

\begin{multline*}
\Sha^1(M)=\Sha^1_V(M)=\ker\Big[\big(M[V_E]_0\big)_\Gt\!\to M[V_E]_\Gt\Big]\\=
   \ker\Big[\big(M[V_E]_0\big)_\Gt\!\to\! \bigoplus_{v\in\V_F}\!M_\Gvt\Big],
\end{multline*}\vskip-5mm

\begin{multline*}
\Sha^2_S(M)=\ker\Big[\hs \big(M[S_E]_0\big)_\G\otimes \Q/\Z\to M[S_E]_\G\otimes\Q/\Z\Big]\\
         =\ker\Big[\hs \big(M[S_E]_0\big)_\G\otimes\Q/\Z\to\bigoplus_{v\in S_F}M_\Gv\otimes\Q/\Z\Big],
\end{multline*}\vskip -5mm

\begin{multline*}
\Sha^2(M)=\Sha^2_V(M)=\ker\Big[\hs \big(M[V_E]_0\big)_\G\otimes \Q/\Z\to M[V_E]_\G\otimes\Q/\Z\Big]\\
         =\ker\Big[\hs \big(M[V_E]_0\big)_\G\otimes\Q/\Z\to\bigoplus_{v\in V_F}M_\Gv\otimes\Q/\Z\Big].
\end{multline*}

Now let $S_E,S_E'\subseteq V_E$ be two $\G$-invariant subsets, not necessarily finite,
and assume that there exists a $\G$-equivariant map $\phi\colon S_E\to S'_E$.
Then $\phi$ induces a $\G$-equivariant  homomorphism
\[\phi_*\colon M[S_E]_0\to M[S'_E]_0, \quad\ \sum_{w\in S_E}m_w\cdot w\mapsto \sum_{w\in S_E} m_w\cdot\phi(w),\quad m_w\in M\]
and homomorphisms of $\Sha$-groups
\[\phi_*^1\colon \Sha^1_S(M)\to \Sha^1_{S'}(M),\qquad \phi_*^2\colon \Sha^2_S(M)\to \Sha^2_{S'}(M).\]
We have the following surprising lemma:
\begin{lemma}\label{l:phi-psi 1-2}
The induced homomorphisms $\phi_*^n\colon \Sha^n_S(M)\to \Sha^n_{S'}(M)$ for $n=1,2$
do not depend on the choice of $\phi$.
\end{lemma}

\begin{proof}
Let $n=1$ and let
\[[\mtil]\in\Sha_S^1(M)\subset \big( M[S_E]_0\big)_\Gt\hs,
    \quad\ \mtil\in M[S_E]_0, \quad\ \mtil=\sum_{w\in S_E} m_w\cdot w,\quad\ m_w\in M.\]
Since $[\mtil]\in \Sha_S^1(M)$, the image of $[\mtil]$ in $M[S_E]_\Gt$ is zero, that is,
\[\mtil=\sum_{\gamma\in\G}(\gamma\cdot a_\gamma-a_\gamma)
   \quad\ \text{for some}\ \,a_\gamma\in M[S_E],\quad a_\gamma
   =\sum_{w\in S_E}a_{\gamma,w}\cdot w, \quad a_{\gamma,w}\in M.\]
Then
\[\phi_*^1(\mtil)=\sum_{\gamma\in\G} \big(\gamma\cdot\phi(a_\gamma)-\phi(a_\gamma)\big).\]

Let $\psi\colon S_E\to S'_E$ be another $\G$-equivariant homomorphism.
Similarly we obtain that
\[\psi_*^1(\mtil)=\sum_{\gamma\in\G} \big(\gamma\cdot\psi(a_\gamma)-\psi(a_\gamma)\big).\]
Then
\[ \psi_*^1(\mtil)-\phi_*^1(\mtil)=\sum_{\gamma\in\G}(\gamma\cdot b_\gamma-b_\gamma)
      \quad\ \text{where}\ \,b_\gamma=\psi(a_\gamma)-\phi(a_\gamma).\]
We have
\[b_\gamma=\sum_{w\in S_E}\big(a_{\gamma,w}\cdot\psi(w)-a_{\gamma,w}\cdot\phi(w)\big).\]
We write
\[b_\gamma=\sum_{w'\in S_E'}b_{\gamma,w'}\cdot w'\ \ \text{with}\ \ b_{\gamma,w'}\in M;
\quad\ \text{then}\ \ \sum_{w'\in S_E'} b_{\gamma,w'}=0\]
and $b_\gamma\in M[S_E']_0$\hs.  Therefore,
\[ [\psi_*^1(\mtil)-\phi_*^1(\mtil)]=0\in \big(M[S_E']_0\big)_\Gt\hs.\]
We conclude that
$$\psi_*^1[\mtil]-\phi_*^1[\mtil]=0
   \in \Sha^1_{S'_E}(M)\subset\big(M[S_E']_0\big)_\Gt\hs,$$
as required.

Let $n=2$. Let
\[[\mtil]\in\Sha_S^2(M)\subset \big( M[S_E]_0\otimes\Q/\Z\big)_\G\hs,
\quad \mtil=\sum_{w\in S_E} m_w\cdot w,\quad m_w\in M\otimes\Q/\Z.\]
Since $[\mtil]\in \Sha_S^2(M)$, the image of $[\mtil]$ in $\big(M[S_E]\otimes\Q/\Z\big)_\G$ is zero, that is,
\begin{align*}
&\mtil=\sum_{\gamma\in\G}(\gamma\cdot a_\gamma-a_\gamma)
 \quad  \ \text{for some}\ \,a_\gamma\in M[S_E]\otimes\Q/\Z,\\
&a_\gamma=\sum_{w\in S_E}a_{\gamma,w}\cdot w, \quad a_{\gamma,w}\in M\otimes\Q/\Z.
   \end{align*}
Then
\[\phi_*^2(\mtil)=\sum_{\gamma\in\G} \big(\gamma\cdot\phi(a_\gamma)-\phi(a_\gamma)\big)
\quad\ \text{and}\quad\
\psi_*^2(\mtil)=\sum_{\gamma\in\G} \big(\gamma\cdot\psi(a_\gamma)-\psi(a_\gamma)\big).\]
We obtain that
\[ \psi_*^2(\mtil)-\phi_*^2(\mtil)=\sum_{\gamma\in\G}(\gamma\cdot b_\gamma-b_\gamma)
      \quad\ \text{where}\ \,b_\gamma=\psi(a_\gamma)-\phi(a_\gamma).\]
We have
\[b_\gamma=\sum_{w\in S_E}\big(a_{\gamma,w}\cdot\psi(w)-a_{\gamma,w}\cdot\phi(w)\big).\]
We write
\[b_\gamma=\sum_{w'\in S_E'}b_{\gamma,w'}\cdot w'\ \ \text{with}\ \ b_{\gamma,w'}\in M\otimes\Q/\Z;
\quad\ \text{then}\ \ \sum_{w'\in S_E'} b_{\gamma,w'}=0.\]
Since the short exact sequence
\[0\to M[S_E']_0 \to M[S_E']\to M\to 0\]
splits, the induced sequence
\[0\to M[S_E']_0\otimes\Q/\Z \to M[S_E']\otimes\Q/\Z\to M\otimes\Q/\Z\to 0\]
also splits and hence is exact.
We know that $b_\gamma \in  M[S_E']\otimes\Q/\Z$ and that
the image of $b_\gamma$ in $M\otimes\Q/\Z$ is zero.
It follows that $b_\gamma\in M[S_E']_0\otimes \Q/\Z$\hs. Therefore,
\[ [\psi_*(\mtil)-\phi_*(\mtil)]=0\in \big(M[S_E']_0\otimes\Q/\Z\big)_\G\hs.\]
We conclude that $$\psi_*^2[\mtil]-\phi_*^2[\mtil]=0\in \Sha^2_{S'_E}(M)
   \subset\big(M[S_E']_0\otimes \Q/\Z\big)_\G\hs,$$
   as required.
\end{proof}

\begin{corollary}\label{c:phi-circ-psi}
Assume that there exist $\G$-equivariant maps (not necessarily mutually inverse)
\[\phi\colon S_E\to S'_E\quad\ \text{and}\quad\ \psi\colon S'_E\to S_E\hs.\]
Then
\[\psi_*^n\circ \phi_*^n=\id \quad\ \text{and}\quad\ \phi_*^n\circ \psi_*^n=\id \]
for $n=1,2$.
In particular, $\phi^n_*$ and $\psi^n_*$ are isomorphisms.
\end{corollary}

\begin{proof}
We apply Lemma \ref{l:phi-psi 1-2}
to the pairs of $\G$-equivariant maps
\[(\psi\circ\phi, \id)\colon\hs S_E\to S_E\quad\ \text{and}\quad
      \ (\phi\circ\psi, \id)\colon\hs S'_E\to S'_E\hs.\qedhere\]
\end{proof}

\begin{corollary}\label{c:SE-VE}
Let $S_E\subset V_E$ be a $\G$-invariant subset such that there exists
a $\G$-equivariant map $\phi\colon V_E\to S_E$\hs.
Let $\iota\colon S_E\into V_E$ denote the inclusion map.
Then the homomorphism
\[\iota_*^n\colon\Sha^n_S(M)\to\Sha^n(M)\coloneqq\Sha^n_V(M)\]
for $n=1,2$ is an isomorphism.
\end{corollary}

\begin{proof}
We apply Corollary \ref{c:phi-circ-psi} to the pair of $\G$-equivariant maps $(\iota,\phi)$.
\end{proof}

\section{The Tate--Shafarevich groups for $(T^{-1}\to T^0)$}

For a homomorphism with finite kernel $T^{-1}\to T^0$ of tori over a global field $F$,
we write $M=\coker\big[X_*(T^{-1})\to X_*(T^0)\big]$.

With the assumptions and notation of \S\ref{ss:M[VE]0}  and \S\ref{ss:explicit-F},
for $n\ge 1$ we define the Tate--Shafarevich groups
\[\Sha^n(F,\Tbul)=\ker\Big[H^n(F,\Tbul)\to \prod_{v\in\V_F} H^n(F_v,\Tbul)\Big].\]
where $\Tbul=(T^{-1}\to T^0)$.
By Theorem \ref{thm:tnglob}(3) we have $\Sha^n(F,T^{-1}\to T^0)=0$ for $n\ge 3$.

The commutative diagrams of Proposition \ref{p:loc2}
show that the homomorphisms
\begin{align*}
&H^1(F, \Tbul)\isoto \cF^1(M)\to \big(M[V_E]_0\big)_\Gt\hs,\\
&H^2(F,  \Tbul)\isoto \cF^2(M)\to \big(M[V_E]_0\big)_\G\otimes\Q/\Z
\end{align*}
induce homomorphisms
\begin{align}
&\Sha^1(F,  \Tbul)\to \Sha^1(M),\label{e:Sha1}\\
&\Sha^2(F,  \Tbul)\to \Sha^2(M).\label{e:Sha2}
\end{align}

\begin{proposition}\label{l:Sha-iso}
The homomorphisms \eqref{e:Sha1} and \eqref{e:Sha2} are isomorphisms.
\end{proposition}

\begin{proof}
Let $\mtil\in \Sha^1(M)\subset \big(M[V_E]_0\big)_\Gt$\hs.
Then the image of $\mtil$ in the group $\prod_{v|\infty} M_\Gvt$ is zero.
Consider the element $1_\infty=\prod_{v|\infty} 1\in \prod_{v|\infty} H^1(F_v,\Tbul)$.
Then $\mtil$ and $1_\infty$ have the same image $0$ in $\prod_{v|\infty} M_\Gvt$\hs, whence
$(\mtil,1_\infty)\in \cF^1(M)$.
By Theorem \ref{thm:tnglob}(1) the pair $(\mtil,1_\infty)$ comes from some $s\in H^1(F,\Tbul)$.
By construction we have $\loc_v(s)=1\in H^1(F_v,\Tbul)$ for all infinite places $v$.
For all places $v$, by Proposition \ref{p:loc2} the image of $\loc_v(s)$ in $M_\Gvt$ is zero.
Since the bottom horizontal arrows in Proposition \ref{p:loc2} are bijective for all finite places $v$,
we have $\loc_v(s)=1\in H^1(F_v,\Tbul)$ for such $v$.
Thus $s\in \Sha^1(F,\Tbul)$.
Clearly, $\mtil$ is the image of $s$,  which proves the surjectivity of \eqref{e:Sha1}.

We prove the injectivity. Let $s\in\Sha^1(F,\Tbul)$ be an element
with image 0 in  $\big(M[V_E]_0\big)_\Gt$.
It has image $1_\infty$ in $\prod_{v|\infty} H^1(F_v,\Tbul)$,
and by Theorem \ref{thm:tnglob}(1) it equals 1,
which proves the injectivity of the homomorphism \eqref{e:Sha1}.
Thus \eqref{e:Sha1} is an isomorphism.
The proof for \eqref{e:Sha2} is similar.
\end{proof}

\begin{lemma}\label{l:cokernels}
For the finite group $\G=\G(E/F)$ and the  $\G$-module $$M=\coker[X_*(T^{-1})\to X_*(T^0)]$$
as in \S\ref{ss:M[VE]0}, let $S_E\subseteq V_E$ be a nonempty $\G$-invariant subset,
and let $S_F$ denote the image of $S_E$ in $V_F$.
Then there are functorial isomorphisms of abelian groups
\begin{align}
&\coker\Big[ \bigoplus_{v\in S_F} H_1(\G\sub \bv,M)\to H_1(\G,M)\Big]\isoto \Sha^1_S(M),\tag1\\
&\coker\Big[ \bigoplus_{v\in S_F} M_\Gvt\to M_\Gt\Big]\isoto \Sha^2_S(M)\tag2
\end{align}
where the homomorphisms $H_1(\G\sub \bv,M)\to H_1(\G,M)$ and $M_\Gvt\to M_\Gt$ are the corestriction maps.
\end{lemma}

\begin{proof}
By Theorem \ref{t:Hinich} in Appendix \ref{app:Hinich}, the short exact sequence of $\G$-modules
\[ 0\to M[S_E]_0\to M[S_E]\to M\to 0\]
gives rise to a homology exact sequence
\begin{align*}
\dots\labelto{\delta_2} H_1&(\G,M[S_E]_0)\to H_1(\G,M[S_E]) \to H_1(\G,M) \\
\labelto{\delta_1}&\big(M[S_E]_0\big)_\Gt \to M[S_E]_\Gt \to M_\Gt\\
&\labelto{\delta_0} \big(M[S_E]_0 \otimes \Q/\Z\big)_\G\to \big(M[S_E]\otimes \Q/\Z\big)_\G
     \to  (M\otimes \Q/\Z)_\G\to 0.
\end{align*}
We have identifications
\[ H_1(\G,M[S_E])=\bigoplus_{v\in S_F} H_1(\G\sub \bv,M),\ \quad\  (M[S_E])_\Gt=\bigoplus_{v\in S_F} M_\Gvt\hs\]
given by the Shapiro isomorphisms for homology in degrees $i=1$ and $i=0$.
More precisely, we have the direct sum decomposition
\[ M[S_E]=\bigoplus_{v \in S_F}M[V_E(v)]\]
of $\G$-modules, and the isomorphism
\[ \Z[\G]\otimes_{\Z[\G\sub\bv]}M \to M[V_E(v)]\]
of $\Z[\G]$-modules
that sends \,$[\gamma]\otimes m$ \,to \,$\gamma \cdot (m \cdot [\bv])$
\,for $\gamma\in\G$, $m\in M$, and $v\in S_F$\hs.
We also have the Shapiro isomorphism
\[ H_i(\Gv,M) \to H_i(\G,\Z[\G]\otimes_{\Z[\Gv]}M)\]
that is given as the composition of the homomorphism
\[ H_i(\Gv,M) \to H_i(\Gv,\Z[\G]\otimes_{\Z[\Gv]}M)\]
induced by the inclusion $M \to \Z[\G]\otimes_{\Z[\Gv]}M$
that sends $m$ to $[1]\otimes m$, and of the corestriction map
\[ H_i(\Gv,\Z[\G]\otimes_{\Z[\Gv]}M) \to H_i(\G,\Z[\G]\otimes_{\Z[\Gv]}M).\]
This leads to the isomorphism $H_i(\Gv,M) \to H_i(\G,M[V_E(v)])$
given as the composition of the homomorphism
\[ H_i(\Gv,M)\to H_i(\Gv,M[V_E(v)])\]
induced by the embedding  $M \to M[V_E(v)]$ sending $m$ to $m\cdot[\bv]$,
and of the corestriction map
\[ H_i(\Gv,M[V_E(v)]) \to H_i(\G,M[V_E(v)]).\]
On the other hand, the homomorphism $H_i(\G,M[V_E(v)]) \to H_i(\G,M)$
is given by the augmentation map $M[V_E(v)] \to M$.
Using the functoriality of the corestriction map and the fact
that composing the inclusion $M \to M[V_E(v)]$ with the augmentation $M[V_E(v)] \to M$ gives the identity on $M$,
we see that composing the isomorphism $H_i(\Gv,M) \to H_i(\G,M[V_E(v)])$
with the augmentation $H_i(\G,M[V_E(v)]) \to H_i(\G,M)$ produces the corestriction map.
The lemma follows.
\end{proof}

\begin{corollary}\label{c:coker}
Let $T^{-1}\to T^0$ be as above,
and let $S_E\subset V_E$ be a $\G$-invariant subset such that there exists
a $\G$-equivariant map $\phi\colon V_E\to S_E$\hs.
Then there are  functorial isomorphisms of finite abelian groups
\begin{align}
&\coker\Big[\hs \bigoplus_{v\in S_F} H_1(\G\sub \bv,M)\to H_1(\G,M)\Big]\isoto \Sha^1(F,T^{-1}\to T^0),\tag1\\
&\coker\Big[\hs \bigoplus_{v\in S_F} M_\Gvt\to M_\Gt\Big]\isoto \Sha^2(F,T^{-1}\to T^0)\tag2.
\end{align}
\end{corollary}

\begin{proof}
By Lemma \ref{l:cokernels} we have isomorphisms
\begin{align*}
&\coker\Big[ \bigoplus_{v\in S_F} H_1(\G\sub \bv,M)\to H_1(\G,M)\Big]\isoto \Sha^1_S(M),\\
&\coker\Big[ \bigoplus_{v\in S_F} M_\Gvt\to M_\Gt\Big]\isoto \Sha^2_S(M).
\end{align*}
Since there exists a $\G$-equivariant map $\phi\colon V_E\to S_E$\hs,
by Corollary \ref{c:SE-VE} we have isomorphisms $\Sha^n_S(M)\isoto\Sha^n(M)$ for $n=1,2$.
By Proposition \ref{l:Sha-iso} we have isomorphisms
\[\Sha^n(M) \isoto \Sha^n(F,T^{-1}\to T^0) \quad\ \text{for}\ \,n=1,2.   \]
The corollary follows.
\end{proof}

\section{The Tate--Shafarevich groups for a reductive group}
\label{ss:Sha}

Let $G$ be a reductive group over a global field $F$ with fundamental group $M=\pi_1(G)$.
We use the notation and assumptions of \S\ref{ss:M[VE]0} and \S\ref{ss:explicit-F}.

Recall that the Tate--Shafarevich kernel of $G$ is
\[\Sha^1(F,G)=\ker\Big[H^1(F,G)\to \prod_{v\in\V_F} H^1(F_v,G)\Big].\]
Recall that $H^i_\ab(F,G)=H^i(F, T_\ssc\to T)$.
We define
\begin{align*}
&\Sha^1_\ab(F,G)=\ker\Big[H^1_\ab(F,G)\to \prod_{v\in\V_F} H^1_\ab(F_v,G)\Big],\\
&\Sha^2_\ab(F,G)=\ker\Big[H^2_\ab(F,G)\to \prod_{v\in\V_F} H^2_\ab(F_v,G)\Big].
\end{align*}

\begin{proposition}\label{p:ab-Sha}
The abelianization map $\ab\colon H^1(F,G)\to H^1_\ab(F,G)$
induces a functorial   bijection $\ab_\Sha\colon\Sha^1(F,G)\isoto \Sha^1_\ab(F,G)$.
\end{proposition}

\begin{proof} The induced map $\ab_\Sha$ is clearly functorial in $G$ and $F$.
By \cite[Theorem 5.12]{Brv98} this map is bijective.
\end{proof}

\begin{proposition}\label{p:Sha-Sha-Sha}

Let $S_E$ be such that  there exists a $\G$-equivariant map $V_E\to S_E$.
Then there are isomorphisms of finite abelian groups
\begin{align}
&\Sha^1_\ab(F,G)\isoto \Sha^1(M)\isoto\Sha^1_S(M),\tag{1}\\
&\Sha^2_\ab(F,G)\isoto \Sha^2(M)\isoto\Sha_S^2(M).\tag{2}
\end{align}
\end{proposition}

\begin{proof}
The proposition follows from Proposition \ref{l:Sha-iso} and Corollary \ref{c:SE-VE}.
\end{proof}

\begin{corollary}\label{c:coker-ab}
For $G$ and $S_E$ as above:
\begin{enumerate}

\item[\rm (i)] {\rm \hs (goes back to \cite[Corollary 5.13]{Brv98} and Kottwitz \cite[(4.2.2)]{Kot84}\hs)}
There is a  bijection of finite sets
\begin{equation*}
\coker\Big[\, \bigoplus_{v\in S_F} H_1(\G\sub \bv,M)\to H_1(\G,M)\Big]\isoto \Sha^1(F,G);
\end{equation*}

\item[\rm (ii)] There is  an isomorphism of finite abelian groups
\begin{equation*}
\coker\Big[\, \bigoplus_{v\in S_F} M_\Gvt\to M_\Gt\Big]\isoto \Sha^2_\ab(F,G)
\end{equation*}
where $H_1(\G\sub \bv,M)\to H_1(\G,M)$ and $M_\Gvt\to M_\Gt$ are the corestriction maps.
\end{enumerate}
\end{corollary}

\begin{proof}
The corollary follows from
Proposition \ref{p:ab-Sha} and Corollary \ref{c:coker}.
\end{proof}

\begin{remark}The bijections of  Proposition \ref{p:ab-Sha}
and Corollary \ref{c:coker-ab}(1)
define a functorial abelian group structure on the finite pointed set $\Sha^1(F,G)$,
which is the same group structure as the one defined by Sansuc \cite[Theorem 8.5]{Sansuc81};
compare Kottwitz \cite[\S4]{Kot84}.
\end{remark}

\begin{corollary}
Assume that either (1) there exists $v\in V_F$ with $\G\sub \bv=\G$, or (2) the group $\G$ is cyclic.
Then $\Sha^1(F,G)=1$ and $\Sha^2_\ab(F,G)=1$.
\end{corollary}

\begin{proof}
Take $S_E=V_E$.
If (1) holds, then both cokernels in Corollary \ref{c:coker-ab} are trivial,
whence we obtain that $\Sha^1(F,G)=1$ and $\Sha^2_\ab(F,G)=1$.
By the Chebotarev density theorem, (2) implies (1).
\end{proof}

\section{Sylow-cyclic groups}

We change our notation.  Let $\G$ be an arbitrary finite group
and $M$ be a finitely generated $\G$-module.
For a Sylow $p$-subgroup $P$ of $\G$, let $i_P\colon P\into\G$ denote the inclusion homomorphism.

\begin{lemma}[{Andersen \cite{Andersen}}]
\label{l:Andersen}
The following homomorphisms are surjective, $P$ running over the Sylow subgroups of $\G$:
\begin{equation}\tag{1}
\xymatrix@C=15mm{\bigoplus_P H_k(P,M)\ar[r]^-{\sum i_{P*}} &H_k(\G,M)}\quad\text{for}\ \,k\ge 0,
\end{equation}
\begin{equation}\tag{2}
\xymatrix@C=15mm{\bigoplus_P M_{P,\Tors}\ar[r]^-{\sum i_{P*}} &M_\Gt\hs,}
\end{equation}
where $H_k$ denotes the $k$-th group homology, and $i_{P*}$ are the corresponding corestriction maps.
\end{lemma}

\begin{proof} We prove that (1) is surjective.
Since $H_k(\G,M)$ is a finite abelian group, it is the direct sum
 of its $p$-torsion parts  $H_k(\G,M)_{(p)}$:
\[ H_k(\G,M)=\bigoplus_p H_k(\G,M)_{(p)}\hs.\]
For a Sylow $p$-subgroup $P$ of $\G$, consider the restriction map
\[\res_{\G/P}\colon H_k(\G,M)\to H_k(P,M).\]
We have
\begin{equation}\label{e:cor-res}
i_{P*}\circ \res_{\G/P}=[\G:P]\colon\hs H_k(\G,M)\to H_k(P,M)\to H_k(\G,M);
\end{equation}
see \cite[Theorem 9.94]{RotmanHA}.
Since $[\G:P]$ is prime to $p$, the map \eqref{e:cor-res} is an isomorphism
when restricted to the $p$-torsion part $H_k(\G,M)_{(p)}$ of $H_k(\G,M)$.
It follows that the homomorphism
\[i_{P*}\colon H_k(P,M)\to H_k(\G,M)_{(p)}\]
is surjective, and hence so is the homomorphism (1) of the lemma.

 We prove that (2) is surjective.
From \eqref{e:cor-res} for $k=0$,
for a Sylow $p$-subgroup $P$ of $\G$
we obtain that
\[ i_{P*}\circ \res_{\G/P}=[\G:P]\colon\hs M_\Gt\to M_{P,\Tors}\to M_\Gt\hs,\]
and we conclude as for (1).
\end{proof}

We return to the notation of \S\ref{ss:Sha}.
In particular, $M=\pi_1(G)$.
The absolute Galois group $\G_F$ acts on $M$.
Let $\G_\eff$ denote the image of $\G_F$ in $\Aut M$,
and let $E/F$ be the finite Galois extension in $F^s$
with Galois group $\G(E/F)=\G_\eff$.

\begin{proposition}\label{p:Sylow-decomp}
Assume that any Sylow subgroup $P$ of\/ $\G_\eff$ is a decomposition group for $E/F$,
that is, there exists a place $w$ of $E$
whose stabilizer $\G_\eff(w)$ in $\G_\eff$ is $P$.
Then
(1) $\Sha^1(F,G)=1$, \.and \,(2) $\Sha^2_\ab(F,G)=1$.
\end{proposition}

\begin{proof}
By Corollary \ref{c:coker-ab}(i) and our assumption, $\Sha^1(F,G)$ is an image of the co\-kernel
\[\coker\bigg[\bigoplus_P H_1(P,M)\to H_1(\G_\eff,M)\bigg] \]
where $P$ runs over the Sylow subgroups of $\G_\eff$.
By Lemma \ref{l:Andersen} this cokernel is 0, which gives (1).
The proof of (2) is similar to that of (1); we use Corollary \ref{c:coker-ab}(ii).
\end{proof}

\begin{definition}
Following a suggestion of Kasper Andersen,
we say that a finite group $\D$ is {\em Sylow-cyclic} if all its Sylow subgroups  are cyclic.
Sansuc \cite[page 13]{Sansuc81} calls Sylow-cyclic subgroups {\em metacyclic}.
Note that any cyclic group is Sylow-cyclic.
\end{definition}

\begin{corollary}
Assume that the group $\G_\eff$ is Sylow-cyclic. Then
\begin{enumerate}
\item (\hs\cite[Corollary 2.10]{Borovoi11} in characteristic 0) $\Sha^1(F,G)=1$;
\item $\Sha^2_\ab(F,G)=1$.
\end{enumerate}
\end{corollary}

\begin{proof}
Indeed, then by the Chebotarev density theorem, all Sylow subgroups are decomposition groups,
and we conclude by Proposition \ref{p:Sylow-decomp}.
\end{proof}

\section[A computer-assisted computation]{Example: a computer-assisted computation \\ of $H^1(F,G)$ and  $\Sha^1(F,G)$}
\label{ss:computer}
Let $F=\Q\big(\sqrt{7}\hs\big)$ and let $E=F(\zeta_8)$
where $\zeta_8$ is a primitive 8-th root of unity.
Write $\G=\G\sub{E/F}$.
We embed $\Z/8\Z$ into $(R_{E/F}\SL_8)(F)=\SL_8(E)$ by sending $1+8\Z$ to $\zeta_8$.
We set $\KK=\Z/8\Z\subset R_{E/F}\SL_8$.
Following Sansuc \cite[Example 5.8]{Sansuc81},  we consider  $G=(R_{E/F}\SL_8)/\KK$.

Write $M=\pi_1(G)$. By Corollary \ref{c:ss-pi1} in Appendix \ref{app:duality}, we have
\[ M\cong\Hom(X^*(\KK),\Q/\Z).\]
We obtain
\begin{gather*}
 X^*(\KK)=\Hom(\Z/8\Z,\ov F^\times)\cong\mu_8\hs,\\
M\cong\Hom(\mu_8,\Q/\Z)=\Hom(\mu_8,\tfrac18\Z/\Z)\cong\Hom(\mu_8,\Z/8\Z).
\end{gather*}
We fix an isomorphism (not $\G$-equivariant)
 $\Hom(\mu_8,\Z/8\Z)\isoto\Z/8\Z$
by sending a homomorphism $\chi\colon\mu_8\to\Z/8\Z$ to $\chi(\zeta_8)\in \Z/8\Z$.

\begin{lemma}\label{l:KConrad}
For $F=\Q\big(\sqrt{7}\hs\big)$ and the Galois extension $E=F(\zeta_8)$,
all decomposition groups are cyclic.
\end{lemma}

\begin{proof}[Proof due to K. Conrad \cite{KConrad}.]
Since $\Q(\zeta_8)/\Q$ ramifies only at $2$ (we are ignoring infinite places),
$E/F$ can ramify only at a place over $2$,
and there is just one such place in $F$ since
$2$ is totally ramified in $F$: \,$(2) = \pp^2$ where $\pp = (2,1+\sqrt{7}\hs)$.
For all other finite places $v$, the decomposition groups are cyclic.

To figure out the ramification and decomposition group at $\pp$,
one can pass to the completion and compute there.
We have $F_{\pp} = \Q_2(\sqrt{7}\hs)$, which equals $\Q_2(\sqrt{-1}\hs)$ since $-7$ is a square in $\Q_2$. Thus
\[
F_{\pp}(\zeta_8) = \Q_2(\sqrt{7},\zeta_8) = \Q_2(\sqrt{-1},\zeta_8)= \Q_2(\zeta_8),
\]
which is a quadratic totally ramified extension of $F_{\pp} = \Q_2(\sqrt{-1}\hs)$.
So the decomposition group $\G\sub \pp$ has order  2: it is cyclic.
 \end{proof}

\begin{lemma}\label{l:8-ab-M}
For $G$ as above,
for each of the two archimedean places $v$ of $F$, the canonical map
\[H^1(F_v,G)\onto H^1_\ab(F_v,G)\into M_\Gvt\]
is in fact bijective.
\end{lemma}

\begin{proof}
We have $F_v\cong \R$ and $E_\bv\simeq \C$.
Write $G_v=G\times_F F_v$; then $(G_v)_\ssc\cong R_{\C/\R}\SL_8$
and $H^1(\R, (G_v)_\ssc)\cong H^1(\C,\SL_8)=1$.
The nontrivial element (complex conjugation)
of the Galois group $\G\sub \bv=\G\sub{E_\bv/F_v}\cong\G\sub{\C/\R}$
sends $\zeta_8$ to $\zeta_8^{-1}$, and so it acts on $M$ as $-1$.
We see that
\[M_\Gvt= M/2M\quad\ \text{and}\quad\ H^1_\ab(F_v, G)\cong H^{-1}(\G\sub \bv, M)\cong M/2M.\]
Therefore, the injective homomorphism $H^1_\ab(F_v,G)\into M_\Gvt$ is bijective.

We have the exact sequence  \eqref{e:Gsc}:
\begin{equation}\label{e:RGv}
H^1(\R, (G_v)_\ssc)\to H^1(\R, G_v)\labelto\ab H^1_\ab(\R,G_v)\hs.
\end{equation}
We see that the kernel of the abelianization map $\ab$ comes from $H^1(\R, (G_v)_\ssc)$.
Let $\eta=[c]$ be an element of $H^1(\R,G_v)$, where $c\in Z^1(\R,G_v)$.
We twist \eqref{e:RGv} by $c$; see Serre \cite[Chapter I, \S5.3]{SerGC}.
We obtain that the fiber $\ab^{-1}(\ab(\eta))\subset H^1(\R,G_v)$
comes from $H^1(\R,{}_c(G_v)_\ssc)$, where $_c(G_v)_\ssc$
is the corresponding inner twisted form of $(G_v)_\ssc$.
All  inner twisted forms of $(G_v)_\ssc=R_{\C/\R}\,\SL_8$
are classified by the image in $H^1(\R, \Aut R_{\C/\R}\SL_8)$
of the Galois cohomology set
\[ H^1(\R,\hs R_{\C/\R}\PGL_8)\cong H^1(\C,\hs\PGL_8)=1.\]
Thus  all inner forms of $(G_v)_\ssc=R_{\C/\R}\hs\SL_8$
are isomorphic to $(G_v)_\ssc$\hs,
and we see that
 \[H^1(\R,{}_c(G_v)_\ssc)\cong H^1(\R,\hs R_{\C/\R}\SL_8)=1,\]
and therefore the fiber $\ab^{-1}(\ab(\eta))$ consists of one element $\eta$.
Thus the surjective map $\ab\colon H^1(F_v,G)\onto H^1_\ab(F_v,G)$ is bijective, which completes the proof.
\end{proof}

\begin{corollary}\label{c:bijective}
For $G$ as above,
the surjective abelianization map
$H^1(F, G)\labelto\ab H^1_\ab(F,G)$ is bijective, and the canonical homomorphism
\[ H^1_\ab(F,G) \to \big(M[V_E]_0)_\Gt=\big(M[V_E]_0)_\G\]
is an isomorphism.
\end{corollary}

\begin{proof}
 By Theorem \ref{t:Brv-GA}(2), the pointed set $H^1(F,G)$ is the fiber product of the maps
\[ H^1_\ab(F,G)\longrightarrow \prod_{v|\infty} H^1_\ab(F_v,G)\longleftarrow  \prod_{v|\infty} H^1(F_v,G).\]
By Lemma  \ref{l:8-ab-M} the arrow at right is bijective, and therefore the map
$H^1(F, G)\labelto\ab H^1_\ab(F,G)$ is bijective  as well.

 By Theorem \ref{t:H-abelian}(1), the group $H^1_\ab(F,G)$ is the fiber product of the homomorphisms
\[ \big(M[V_E]_0\big)_\Gt\longrightarrow \prod_{v|\infty}M_\Gvt \longleftarrow\prod_{v|\infty} H^{-1}(\G(\bv),M).\]
By Lemma  \ref{l:8-ab-M}  and the Tate--Nakayama theorem,
the arrow at right is bijective, and therefore the homomorphism
$H^1_\ab(F,G)\to \big(M[V_E]_0\big)_\Gt$ is bijective as well.

 We have $\big(M[V_E]_0)_\Gt=\big(M[V_E]_0)_\G$ because $M$ is torsion.
\end{proof}

We identify $M$ with $\Z/8\Z=\{\bar x=x+8\Z\ |\  x\in \Z\}$.
Then $\Aut(M)=(\Z/8\Z)^\times=\{\hs\bar 1,\bar 3,\bar 5,\bar 7\hs\}$
acting on $M=\Z/8\Z$ by  $\bar a\cdot\bar x=\ov{ax}$.
The action of $\G=\G\sub{E/F}$ on $M$ permits us to identify  $\G$ with $(\Z/8\Z)^\times$.
The group $\G=(\Z/8\Z)^\times$ has three nontrivial cyclic subgroups:
$\langle\hs\bar3\hs\rangle, \langle\hs\bar5\hs\rangle, \langle\hs\bar7\hs\rangle$.
For each $a=3,5,7$ we choose a place   $v_a$ of $F$
with decomposition group $\langle\hs\bar a\hs\rangle$.
We set
\begin{equation*}
S_F=\{ v_3, v_5, v_7\}\subset V_F\hs,\quad\  S_E=\{ \bv_3, \bv_3', \bv_5, \bv_5', \bv_7, \bv_7'\}\subset V_E
\end{equation*}
where for $a=3,5,7$ we denote by $\bv_a$ and $\bv_a'$ the two lifts of $v_a\in V_F$ to $V_E$\hs.
We write $S_F^\cmp=V_F\smallsetminus S_F$\hs, \,$S_E^\cmp=V_E\smallsetminus S_E$.
Since by Lemma \ref{l:KConrad} all decomposition groups are cyclic,
we see that the three cyclic subgroups
\[\G(\bv_3)=\langle\bar 3\rangle,\  \G(\bv_5)=\langle\bar 5\rangle,\  \G(\bv_7)=\langle\bar 7\rangle\]
are all maximal decomposition subgroups in $\G$,
and therefore our set $S_F$ satisfies \eqref{e:wS}.
We choose a $\G$-equivariant map $\phi\colon S_E^\cmp\to S_E$.
By diagram (1) of Lemma \ref{l:AV}, we have an isomorphism of abelian groups
\[\big(M[S_E]_0\big)_\G
      \oplus\bigoplus_{v\in S_F^\cmp} M_\Gv\isoto \big(M[V_E]_0\big)_\G\hs.\]

We computed the finite abelian group $\big(M[S_E]_0\big)_\G$
using the GAP System for Computational Discrete Algebra  \cite{GAP4};
see the listing of our computer program and the output of this program
in Appendix F to the arXiv version \cite{BorKal23} of this paper.
We obtained the following:
\begin{equation}\label{e:2-8}
\big(M[S_E]_0\big)_\G=\langle\eta_2,\eta_8\rangle\simeq \Z/2\Z\oplus\Z/8\Z
\end{equation}
where $\eta_2$ is of order 2 and $\eta_8$ is of order 8.
Here
\begin{equation*}
\eta_8=[\bv_3- \bv_5]\in \big(M[S_E]_0\big)_\G
\quad\ \text{where}\ \ \bv_3\hs, \bv_5\in S_E\subset M[S_E].
\end{equation*}
Moreover, we obtained that
\begin{equation}\label{e:kerMSE0}
\ker\Big[\hs\loc_S\colon \big(M[S_E]_0\big)_\G\to M[S_E]_\G\hs\Big]=\langle 4\eta_8\rangle\simeq \Z/2\Z.
\end{equation}
It follows from  \eqref{e:kerMSE0} that
\begin{multline}\label{e:Sha}
\Sha^1(F,G)\cong\Sha^1_\ab(F,G)\cong\ker\Big[ \big(M[S_E]_0\big)_\Gt\to M[S_E]_\Gt\hs\Big]\\
=\ker\Big[ \big(M[S_E]_0\big)_\G\to M[S_E]_\G\Big]\simeq \Z/2\Z.
\end{multline}

\begin{remark}
One can deduce the existence of a bijection $\Sha^1(F,G)\simeq\Z/2\Z$ of  \eqref{e:Sha} from known results.
Indeed, we have $H^1_\ab(F,G)\cong H^2(F,\KK)$, whence
\[\Sha^1(F,G)\cong\Sha^1_\ab(F,G)\cong \Sha^2(F,\KK).\]
By the  Poitou--Tate duality (see  \cite[Theorem (8.6.7)]{NSW08}
or Harari \cite[Theorem 17.13(b)]{Harari20}\hs), we have
\[\Sha^2(F,\KK)\cong\Hom\big(\Sha^1(F,X^*(\KK)\hs),\Q/\Z\big)
     \cong\Hom\big(\Sha^1(F,\mu_8),\Q/\Z\big).\]
Thus $\Sha^1(F,G)\cong \Hom\big(\Sha^1(F,\mu_8),\Q/\Z\big)$.
Moreover, for $F=\Q(\sqrt7)$ we have that $\Sha^1(F,\mu_8)\cong\Z/2\Z$,
see Sansuc \cite[(2.7)]{Sansuc81},
whence $\Sha^1(F,G)\cong \Z/2\Z$.
\end{remark}

Proposition \ref{p:FS}  explicitly describes
the infinite abelian group $H^1_\ab(F,G)$
as a direct sum  of finite abelian groups
where each of the summands can be explicitly computed.
In particular, using \eqref{e:2-8} and \eqref{e:kerMSE0} we can prove the following proposition.

\begin{proposition}\label{p:using-computer}
For $G=(R_{E/F}\SL_8)/\KK$ as above,
the subgroup $\Sha^1_\ab(F,G)$ is not a direct summand of $H^1_\ab(F,G)$.
\end{proposition}

\begin{proof}
We embed $\big(M[S_E]_0\big)_\G\into \big(M[V_E]_0\big)_\G$ and identify
$\big(M[V_E]_0\big)_\G$ with $H^1_\ab(F,G)$ as in Corollary \ref{c:bijective}.
Set
\[ a=4\eta_8\in \big(M[S_E]_0\big)_\G\subset\big(M[V_E]_0\big)_\G= H^1_\ab(F,G).\]
By \eqref{e:kerMSE0} the element $a$ is the generator of the subgroup of order 2
\ $\Sha^1_\ab(F,G)\subset H^1_\ab(F,G)$.
Set
\[ b=2\eta_8\in \big(M[S_E]_0\big)_\G\subset\big(M[V_E]_0\big)_\G=H^1_\ab(F,G).\]
Then $a=2b$.
Now the proposition follows from Lemma \ref{l:not-direct} below.
\end{proof}

Let $A$ be a finite abelian group.
Let $e(A)$ denote the {\em exponent of}\/ $A$,
that is, the least common multiple of the orders of elements of $A$.
Since $A$ is abelian, there exists an element $a_e$ of $A$ of order $e(A)$.
Moreover, $e(A)\cdot a=0$ for all $a \in A$.

\begin{lemma}[probably known]\label{l:not-direct}
Let $A$ be a nontrivial finite subgroup of an abelian group $B$.
Let $a_e\in A$ be an element of order $e(A)$.
Let $p$ be a prime factor of $e(A)$.
Assume that there exists an element $b\in B$ such that $pb=a_e$.
Then $A$ is not a direct summand of $B$.
\end{lemma}

\begin{proof}[Proof suggested by an anonymous referee]
If $B = A \oplus A'$,   then $b = x + x'$ with $x \in A$ and $x' \in A'$.
We obtain that  $a_e = pb = px + px'$ and thus $a_e = px$ and $px' = 0$ since $a_e \in A$.
But then $a_e$ is killed by $e(A)/p$, which is a contradiction.
\end{proof}

\begin{corollary}[probably known]
\label{c:div-not-direct}
Consider inclusions of abelian groups $A\subset D\subseteq B$ where $A$ is a nontrivial finite subgroup of a divisible group $D$.
Then $A$ is not a direct summand of $B$.
\end{corollary}

\begin{proof}
Indeed, let $a_e\in A$ and $p\hs|\hs e(A)$ be as in Lemma \ref{l:not-direct}.
Since $D$ is divisible, there exists an element $b\in D\subset B$ such that $p b=a_e$\hs,
and we conclude by the lemma.
\end{proof}

\section{$\Sha^2_\ab(F,G)$ is not a direct summand}
\label{ss:not-direct}

\begin{theorem}\label{t:not-direct}
Let $G$ be a reductive group over a global field $F$.
If $\Sha^2_\ab(F,G)$ is nontrivial,
then it is not a direct summand of $H^2_\ab(F,G)$.
\end{theorem}

\begin{corollary}\label{c:not-direct}
Let $T$ be an torus over a global field $F$.
If  $\Sha^2(F,T)$ is nontrivial,
then it is not a direct summand of $H^2(F,T)$.
\end{corollary}

If $F$ is a function field, then by Corollary \ref{l:divisible} the group $H^2_\ab(F,G)$ is divisible,
and by Corollary \ref{c:div-not-direct} $\Sha^2_\ab(F,G)$
is not a direct summand of $H^2_\ab(F,G)$.
It remains to prove Theorem \ref{t:not-direct} in the case when $F$ is a number field. We need a lemma.

\begin{lemma}\label{l:ker-alpha}
With the notation of \S\ref{ss:M[VE]0}, consider the
exact sequence of $\G$-modules
\[ 0\to M[V_f(E)]_0\labelto\iota M[V_E]_0\labelto{\vk} M[V_\infty(E)]\]
where $\iota$ is the inclusion map, and $\vk$ is the restriction to $M[V_E]_0$
of the natural projection $M[V_E]\to M[V_\infty(E)]$.
Then the following induced sequence is exact:
\[0\to  \big(M[V_f(E)]_0\big)_\G\otimes \Q/\Z\labelto{\iota_*} \big(M[V_E]_0\big)_\G\otimes \Q/\Z
 \labelto{\vk_*}  M[V_\infty(E)]_\G\otimes \Q/\Z\to 0.\]
\end{lemma}

\begin{proof}
We show that $\vk$ admits a $\G$-equivariant splitting (homomorphic section).
Indeed, it follows from the  Chebotarev density theorem that there exists a $\G$-equivariant map
$\phi\colon V_\infty(E)\to V_f(E)$. We construct a $\G$-equivariant splitting:
\[\phitil\colon M[V_\infty(E)]\to M[V_E]_0\hs, \quad \sum_{w\in V_\infty(E)}\!\!\!\!\!\! m_w\cdot w\ \,\longmapsto\!\!\!\!\!
\sum_{w\in V_\infty(E)}\!\!\!\!\! \big(m_w\cdot w-m_w\cdot\phi(w)\big)\]
where $m_w\in M$. The lemma follows from  the existence of such a  splitting.
\end{proof}

\begin{proof}[Proof of Theorem \ref{t:not-direct}]
By Theorem \ref{t:H-abelian}(2) we may and will identify the groups
$H^2_\ab(F,G)=\cF^2(M)$.
We embed the fiber product $\cF^2(M)$ into the direct product
\[ \Big( (M[V_E]_0)_\G\otimes \Q/\Z\Big)\,\times\,\prod_{v|\infty} H^0(\Gv, M).\]
We consider the subgroup
\[\ker \vk_*\hs\times \{0\}\subseteq \Big( (M[V_E]_0)_\G\otimes \Q/\Z\Big)\,\times\,\prod_{v|\infty} H^0(\Gv, M)\]
where $\vk_*$ is as in Lemma \ref{l:ker-alpha}.
Then we have
\begin{equation}\label{e:ker-vk}
\ker\vk_*\hs\times \{0\}\subseteq \cF^2(M).
\end{equation}

On the other hand, by the definition of $\cF^2(M)$  there is a commutative diagram
\[
\xymatrix{
\cF^2(M)\ar[r]\ar[d]_-{\loc_\infty}        &\big(M[V_E]_0\big)_\G\otimes \Q/\Z\ar[d]^-{l_\infty}\\
\prod\limits_{v|\infty}H^2_\ab(F_v,G)\ar[r]   &\prod\limits_{v|\infty} M_\Gv\otimes \Q/\Z
}
\]
where the right-hand vertical arrow $l_\infty$ factors as
\[l_\infty\colon\,\big(M[V_E]_0\big)_\G\otimes \Q/\Z\labelto{\vk_*} M[V_\infty(E)]_\G\otimes \Q/\Z
   \, \isoto\, \bigoplus\limits_{v|\infty} M_\Gv\otimes \Q/\Z.\]
From this diagram we see that
\begin{equation}\label{e:subset2}
\Sha^2(M)\subset \ker l_\infty\hs\times \{0\}= \ker \vk_*\hs\times \{0\}.
\end{equation}
From  \eqref{e:ker-vk} and \eqref{e:subset2} we obtain inclusions of abelian groups
\begin{equation}\label{e:two-inclusions}
\Sha^2(M)\subset  \ker\vk_*\times \{0\}\subset \cF^2(M).
\end{equation}

By Lemma \ref{l:ker-alpha} the homomorphism
\[\iota_*\colon \big(M[V_f(E)]_0\big)_\G\otimes \Q/\Z \to \big(M[V_E]_0\big)_\G\otimes \Q/\Z\]
is injective with image $\im\iota_*=\ker\vk_*$.
Hence the group
\[\ker\vk_*\times\{0\}=\im\iota_*\times\{0\}\cong\big(M[V_f(E)]_0\big)_\G\otimes \Q/\Z\] is divisible,
and by applying Corollary \ref{c:div-not-direct} to the inclusions \eqref{e:two-inclusions},
we conclude that the subgroup $\Sha^2(M)$, if nontrivial, is not a direct summand of $\cF^2(M)$.
The theorem follows.
\end{proof}

We give examples of $F$-tori $T$ with $\Sha^2(F,T)\neq 0$.

\begin{lemma}[{Sansuc \cite[Remark 1.9.4]{Sansuc81}}\hs]
\label{l:Sansuc-Sha2}
Let $E/F$ be a finite Galois extension of global fields of degree $n$ with Galois group $\G$.
Consider the $(n-1)$-dimensional torus
\[ T=(R_{E/F}\GG_m)/\GG_{m,F}\hs.\]
Then $\Sha^2(F,T)$ is a cyclic group of order $n/l$ where $l=\lcm(n_v)$
is the least common multiple of the orders $n_v=|\G\sub \bv|$ of the decomposition groups $\G\sub \bv$
where $v$ runs over the set $V_F$ of places of $F$.
\end{lemma}

\begin{proof}
Sansuc states the lemma only in the case of number fields
and gives no proof, so we prove the lemma here.
Write $X=X^*(T)$ (the character group of $T$).
By the global Poitou--Tate duality we have a canonical isomorphism
\[ \Sha^2(F,T)\isoto \Hom(\Sha^1(F^s\hm/F,X),\Q/\Z);\]
see \cite[Theorem (8.6.9)]{NSW08}.
Since $\G(F^s/F)$ is profinite and $X$ is torsion-free,
we have  reductions $H^1(F^s/F, X)=H^1(E/F,X)$ and  $\Sha^1(F^s\hm/F, X)=\Sha^1(E/F,X)$;
see \cite[Lemma 1.9]{Sansuc81}.
Therefore, it remains to compute $\Sha^1(\G,X)$.

The $\G$-module  $X$ fits into the exact sequence
\[ 0\to X\to \Ind_\G\Z\labelto{N} \Z\to 0\]
where $N\big(\sum_{\gamma\in \G}a_\gamma\cdot\gamma\big)=\sum_{\gamma\in\G} a_\gamma$.
This exact sequence gives rise to a Tate cohomology exact sequence
\[0= H^0(\G,\Ind_\G X)\to H^0(\G,\Z)\to H^1(\G, X)\to H^1(\G,\Ind_\G X)=0\]
whence we obtain an isomorphism
\[ H^1(\G,X)\cong H^0(\G,\Z)=\Z/n\Z\quad\ \text{where}\ \,n=|\G|.\]

Let $v\in V_F$, and let $\bv\in V_E$ be a place of $E$ over $v$.
As above, we obtain $H^1(\Gv,X)\cong \Z/n_v\Z$ where $n_v=|\Gv|$,
and we have a commutative diagram
\[
\xymatrix{
\Z/n\Z\ar[r]^-\sim\ar[d]_-{\res^0_v} &H^1(\G,X)\ar[d]^-{\res^1_v}\\
\Z/n_v\Z\ar[r]^-\sim                  &H^1(\Gv,X)
}
\]
where the map
\[\res_v^0\colon \Z/n\Z=H^0(\G,\Z)\to H^0(\Gv,\Z)=\Z/n_v\Z\]
is surjective, because by a formula in Appendix \ref{app:cheat-sheet}
this map is induced by the identity map $\id\colon \Z\to\Z$.
We see that $\ker (\res^1_v)$ is the subgroup of order $m_v\coloneqq n/n_v$
of the cyclic group $H^1(\G,X)\cong\Z/n\Z$.
It follows that $\Sha^1(\G,X)=\bigcap_v\ker (\res^1_v)$
is the subgroup in $H^1(\G,X)\cong\Z/n\Z$
of order $\gcd(m_v)=n/\lcm( n_v)$.
The lemma follows.
\end{proof}

\begin{lemma}[well-known]
\label{l:-13-17}
Let $F=\Q$ and  $E=\Q(\sqrt{13},\sqrt{17})$.
Then all decomposition groups for $E/F$ are cyclic.
\end{lemma}

\begin{proof}
The field extension $E/F$ is unramified outside $\{\infty, 2,13,17\}$,
 and so for $v\notin \{\infty, 2,13,17\}$
the decomposition group $\G(v)$ is cyclic,
where by abuse of notation we write $\G(v)$ for $\Gv$.
The group $\G(\infty)=1$ is clearly cyclic.
Since $17\equiv 1\pmod 8$, we can identify $\Q_2(\sqrt{13},\sqrt{17})=\Q_2(\sqrt{13})$
and  therefore $\G(2)$ is cyclic.
Clearly, we have $\big(\frac{17}{13}\big)=1$,
whence $\Q_{13}(\sqrt{13},\sqrt{17})=\Q_{13}(\sqrt{13})$,
and therefore the decomposition group $\G(13)$ is cyclic.
By the quadratic reciprocity law we have $\big(\frac{13}{17}\big)=\big(\frac{17}{13}\big)=1$,
whence $\Q_{17}(\sqrt{13},\sqrt{17})=\Q_{17}(\sqrt{17})$,
and therefore the decomposition group $\G(17)$ is cyclic.
\end{proof}

\begin{example}Following Sansuc \cite[Remark 1.9.4]{Sansuc81},
we take $F=\Q$ and  $E=\Q(\sqrt{13},\sqrt{17})$.
By Lemma \ref{l:-13-17} all decomposition groups for $E/F$ are cyclic, that is, of order 1 or 2.
We consider the 3-dimensional torus $T=(R_{E/F}\GG_m)/\GG_{m,F}$\hs.
By Lemma \ref{l:Sansuc-Sha2} we have $\Sha^2(\Q,T)\simeq\Z/2\Z\neq 0$.
By Theorem \ref{t:not-direct} the subgroup $\Sha^2(\Q,T)\subset H^2(F,T)$ is not a direct summand.
\end{example}

\begin{example}
Following Sawin \cite{Sawin}, we take
$F=\FF_q(t)$, the field of rational functions in one variable $t$ over a finite field $\FF_q$.
We assume that $q\equiv 1\pmod{4}$, for instance, $q=5$ or $q=9$,
and we take $E=\FF_q\big(\sqrt{t},\sqrt{t^2-1}\,\big)$.
Then by  Lemma \ref{l:Sawin} below, all decomposition groups for $E/F$ are cyclic.
Again we consider the 3-dimensional torus $T=(R_{E/F}\GG_m)/\GG_{m,F}$;
then by Lemma \ref{l:Sansuc-Sha2} we have $\Sha^2(F,T)\simeq\Z/2\Z$, and
by Theorem \ref{t:not-direct} the subgroup $\Sha^2(F,T)\subset H^2(F,T)$ is not a direct summand.
\end{example}

\begin{lemma}[Sawin \cite{Sawin}]
\label{l:Sawin}
Let $F =\FF_q(t)$, $E= \FF_q ( \sqrt{t}, \sqrt{t^2-1} ) $ where  $q\equiv 1 \bmod 4$.
Then all decomposition groups for $E/F$ are cyclic.
\end{lemma}

\begin{proof}
It suffices to check that at each place where one of the two extensions
$\FF_q(\sqrt{t})/\FF(t)$ and $\FF_q(t,\sqrt{t^2-1})/\FF(t)$
ramifies, the other is split,
as this clearly gives a cyclic decomposition group, and the unramified places are always cyclic.

The first extension ramifies at $0, \infty$.
At $t=0$, the second extension $y^2=t^2-1 $  locally looks like $y^2 = 0-1 = -1$,
which is split since $q \equiv 1 \bmod 4$.
 We can write  the second extension as $\left(\frac{y}{t}\right)^2= 1- \frac{1}{t^2}$,
which at $t=\infty$ locally looks like $\left(\frac{y}{t}\right)^2 = 1- \frac{1}{\infty^2}  =1$, which is split.

The second extension ramifies at $1,-1$ where the first extension locally looks respectively
like $y^2=1$ and $y^2=-1$, and these extensions again are split since $q\equiv 1\bmod 4$.
\end{proof}

\appendix

\chapter{A long exact sequence containing $B_\Gt$}
\label{app:Hinich}

\centerline{ by Vladimir Hinich}
\bigskip

In this appendix, for a {\em finite} group $\G$ and a short exact sequence of $\G$-modules \eqref{e:SG}
we construct a long exact sequence in group homology \eqref{e:LG} of Theorem \ref{t:Hinich}.
We compute the connecting homomorphisms and investigate the functoriality in $\G$.

\section{Exact sequence}

\begin{theorem}\label{t:Hinich}
A finite group $\G$ and a short exact sequence of $\G$-modules
\begin{equation}\label{e:SG}
0\to B'\overset \iota\lra B\overset \varkappa\lra B''\to 0
\end{equation}
give rise to a long exact sequence of homology
\begin{multline}\label{e:LG}\tag{${\rm L}_\G$}
\dots \overset{\delta_2}\lra H_1(\G,B')\li H_1(\G,B) \lj H_1(\G,B'')\overset{\delta_1}\lra\\
   (B')_\Gt \li B_\Gt \lj (B'')_\Gt\overset{\delta_0}\lra\\
         (B'\otimes\Q/\Z)_\G \li (B\otimes\Q/\Z)_\G \lj (B''\otimes\Q/\Z)_\G\to 0
\end{multline}
(where $H_m$ denotes the $m$-th homology group)
depending functorially on $\G$ and on the short exact sequence \eqref{e:SG}.
\end{theorem}

\begin{proof}
The functor  from  the category of $\G$-modules to the category of abelian groups
\[ B\rsa \Q/\Z\otimes_\Z B_\G\]
is the same as
\[ B\rsa\Q/\Z\otimes_\Lam\! B\]
where $\Lam=\Z[\G]$ is the group ring of $\G$.
From the short exact sequence of $\G$-modules  \eqref{e:SG},
we obtain a long exact sequence
\begin{multline}\label{e:LG-Tor}
\dots\Tor_2^\Lam(\Q/\Z,B') \li \Tor_2^\Lam(\Q/\Z,B) \lj \Tor_2^\Lam(\Q/\Z,B'')\labelto{\delta_2^{\Tor}}\\
 \qquad\quad\Tor_1^\Lam(\Q/\Z,B') \li \Tor_1^\Lam(\Q/\Z,B) \lj \Tor_1^\Lam(\Q/\Z,B''))\labelto{\delta_1^{\Tor}}\\
\Q/\Z\otimes_\Lam\! B'\li \Q/\Z\otimes_\Lam\! B\lj \Q/\Z\otimes_\Lam\! B''\to 0
\end{multline}
depending functorially on $\G$ and on the short exact sequence  \eqref{e:SG}.
Now Theorem \ref{t:Hinich} follows from the next proposition.
The connecting homomorphisms $\delta_m$ in  \eqref{e:LG} will be computed in \S\ref{ss:Hinich-2}.
\end{proof}

\begin{proposition}\label{p:Hinich-iso}
For a finite group $\G$, write $\Lam=\Z[\G]$. Then for any $\G$-module $B$, there are  functorial isomorphisms
\[ \phi_1\colon \Tor_1^\Lam (\Q/\Z,B)\isoto B_\Gt\,,\quad\
\phi_m\colon\Tor_m^\Lam(\Q/\Z, B)\isoto H_{m-1}(\G,B)\,\text{ for }\,m\ge 2.\]
\end{proposition}

We need a lemma.

\begin{lemma}\label{l:Tor-0}
For a finite group $\G$ and any $\G$-module $B$, we have
$\Tor_m^\Lam (\Q,B)=0$ for all $m\ge 1$  and
$\Tor_0^\Lam(\Q,B)=\Q\otimes_\Z B_\G$.
\end{lemma}

\begin{proof}
Let
\[ P_\bullet:\quad\dots\to P_2\to P_1\to P_0\to\Z\to 0\]
be the standard complex, which is a $\Lam$-free resolution of the trivial $\G$-module $\Z$;
see \cite[Chapter IV, \S2]{CasFro86}.
Recall that $P_m=\Z[\G^{m+1}]$.
We consider the {\em Bar complex} $\BarC(B)=P_\bullet \otimes_\Lambda B$,
also known as the {\em complex of  homogeneous chains.}
Tensoring with $\Q$ over $\Z$, we obtain a flat resolution of $\Q$
\[ \dots\to \Q\otimes_\Z P_2\to \Q\otimes_\Z P_1\to \Q\otimes_\Z P_0\to\Q\to 0.\]
Tensoring this resolution with the $\Gamma$-module  $B$ over $\Lam=\Z[\G]$,
we obtain the complex $(\Q\otimes_\Z P_\bullet)\otimes_\Lam B\,$:
\begin{equation}\label{e:QPB}
\dots (\Q\otimes_\Z P_2)\otimes_\Lam B\to (\Q\otimes_\Z P_1)\otimes_\Lam B
    \to (\Q\otimes_\Z P_0)\otimes_\Lam B\to\Q\otimes_\Lam B\to 0.
\end{equation}
By definition, $\Tor_m^\Lam(\Q,B)$ is the $m$-th homology group of this complex.

However, we can obtain the complex \eqref{e:QPB} from $P_\bullet$
by tensoring first with $B$ over $\Lam$,
and after that with $\Q$ over $\Z$:
\[ \Q\otimes_\Z\big(P_\bullet\otimes _\Lam B\big)\,\cong\, (\Q\otimes_\Z P_\bullet)\otimes_\Lam B.\]
Since $\Q$ is a flat $\Z$-module, we obtain canonical isomorphisms
\[ \Tor_m^\Lam(\Q,B)\cong \Q\otimes_\Z \Tor_m^\Lam(\Z,B)=\Q\otimes_\Z H_m(\G,B).\]

Now, since the group $\G$ is finite, the abelian group $H_m(\G,B)$ for $m\ge 1$
is killed by multiplication by $\#\G$;
see, for instance, \cite[Chapter IV,  \S6, Corollary 1 of Proposition 8]{CasFro86}.
It follows that $\Q\otimes_\Z H_m(\G,B)=0$.
Thus $\Tor_m^\Lam (\Q,B)=0$, as required. Since $H_0(\G,B)=B_\G$,
we deduce that $\Tor_0^\Lam(\Q,B)=\Q\otimes_\Z B_\G$.
\end{proof}

\begin{proof}[Proof of Proposition \ref{p:Hinich-iso}]
We  calculate the graded group $\Tor^\Lambda_\bullet(\Z,B)=H_\bullet(\Gamma,B)$ as the homology of
the complex $\BarC(B)$ of homogeneous chains.
We calculate $\Tor^\Lambda_\bullet(\Q,B)$ as the homology of
$\Q\otimes\BarC(B)$ (the tensor product taken over $\Z$). Moreover,
$\Tor^\Lambda_\bullet(\Q/\Z,B)$ is the homology of the complex
$\Q/\Z\otimes\BarC(B)$.
We consider the  $\Z$-flat resolution
$R_\bullet=(R_1\to R_0)$  of $\Q/\Z$ defined by the formulas
\[  R_1=\Z,\ R_0=\Q \]
with the differential given by the standard embedding $\Z\into\Q$.
We denote by $u$ the standard generator of $R_1$.
Then we can compute  $\Tor^\Lambda_\bullet(\Q/\Z,B)$
as the homology of the tensor product of complexes $R_\bullet\otimes\BarC(B)$,
with differential given by the standard formula
\[d\colon\,\hs u\otimes x+a\otimes y\,\hs\longmapsto\,\hs 1\otimes x-u\otimes dx+a\otimes dy\]
where $x\in (\BarC\, B)_{m-1}$, $y\in (\BarC\, B)_{m}$, and  $a\in\Q$.

We have an exact triangle
\begin{equation}\label{e:triangle}
\begin{aligned}
\xymatrix{
& & R_\bullet\otimes\BarC(B) \ar_{-1}^\pi[dl] & \\
& \BarC(B) \ar[rr] & & \Q\otimes\BarC(B)\ar^\iota[ul]
}
\end{aligned}
\end{equation}
where the morphism of complexes $\iota$ is induced by the morphism  $(0\to \Q)\into (\Z\to \Q)=R_\bullet$,
and where $\pi$ is induced by the degree $-1$ map $R_\bullet\to\Z$ sending $R_0=\Q$ to 0 and
sending $u\in R_1$ to $1\in\Z$.
Here by Lemma \ref{l:Tor-0} the homology $H_m(\Q\otimes\BarC(B))$ vanishes for $m\ge 1$
and
$H_0(\Q\otimes\BarC(B))=\Q\otimes B_\Gamma$.

The exact triangle \eqref{e:triangle} gives rise to a long exact sequence of homology
\begin{multline*}
\to\Tor^\Lambda_{m}(\Q,B)\to\Tor^\Lambda_m(\Q/\Z,B)\labeltooo{H_m(\pi)}
H_{m-1}(\Gamma,B)\to\Tor^\Lambda_{m-1}(\Q,B)\to\dots\\
\dots\to\Tor_1^\Lambda(\Q,B)\to\Tor_1^\Lambda(\Q/\Z,B)\to B_\G\to\Q\otimes B_\G
\end{multline*}
which establishes the isomorphisms
 $H_m(\pi):\Tor^\Lambda_m(\Q/\Z,B)\isoto H_{m-1}(\Gamma,B)$ for $m>1$
and the isomorphism $H_1(\pi):\Tor_1^\Lambda(\Q/\Z,B)\isoto B_\Gt$\hs.
This completes the proofs of Proposition \ref{p:Hinich-iso}
and Theorem \ref{t:Hinich}.
\end{proof}

\begin{lemma}
\label{lem:viator}
The map $H_m(\pi)$ sends the element of $\Tor^\Lambda_m(\Q/\Z,B)$
represented by an $m$-cycle $u\otimes x+a\otimes y$
where \,$x\in(\BarC\hs B)_{m-1}$\hs, \,$y\in(\BarC\hs B)_m$\hs, $a\in \Q=R_0$\hs,  \hs to the element
of $H_{m-1}(\Gamma,B)$ represented by $x$.
\end{lemma}

\begin{proof}
Obvious because $\pi(u\otimes x+a\otimes y)=x$.
\end{proof}

\section{Connecting homomorphisms}
\label{ss:Hinich-2}

We compare the connecting homomorphisms $\delta_i, i\geq 1$ for the
functors $H_\bullet(\Gamma,-)$ in the
exact sequence \eqref{e:LG} of Theorem \ref{t:Hinich}, with the connecting homomorphisms for the functor
$\Tor_\bullet^\Lambda(\Q/\Z,-)$ in  \eqref{e:LG-Tor}. We also compute the
connecting homomorphism $\delta_0$ in the exact sequence~\eqref{e:LG}.

\begin{proposition}\label{p:delta}
For $m\ge 1$, the following diagram anti-commutes:
\begin{align*}
\xymatrix@C=11mm{
\Tor_{m+1}^\Lambda(\Q/\Z, B'')\ar[r]^-{\delta_{m+1}^{\Tor}}\ar[d]_{\psi''_{m+1}}   &\Tor_{m}^\Lambda(\Q/\Z, B')\ar[d]^{\psi'_{m}} \\
H_{m}(\G, B'')\ar[r]^-{\delta_{m}}                                       &H_{m-1}(\G, B')
}
\end{align*}
that is,
 \[(\psi'_{m}\circ \delta_{m+1}^{\Tor})(t)=-(\delta_{m}\circ \psi''_{m+1})(t)\quad\ \text{for all}\ \, t\in \Tor_{m+1}^\Lambda(\Q/\Z, B'').\]
In this diagram, the bottom horizontal arrow is the connecting map coming from  the short exact sequence \eqref{e:SG},
and the top horizontal arrow is the connecting map from the exact sequence \eqref{e:LG-Tor}.
Moreover, for $m\ge 1$ the vertical arrows
\[\psi_{m+1}''\colon \Tor_{m+1}^\Lambda(\Q/\Z, B'')\to H_{m}(\G, B''),\quad\
   \psi_{m+1}'\colon \Tor_{m+1}^\Lambda(\Q/\Z, B')\to H_{m}(\G, B')\]
are the corresponding homomorphisms $\phi_{m+1}$ for $B''$ and $B'$
from Proposition \ref{p:Hinich-iso}, respectively,
and $\psi'_1$ is the composite map
\begin{equation}\label{l:psi-phi}
\psi'_1\colon \Tor_1^\Lam (\Q/\Z,B')\labelto{\phi'_1} (B')_\Gt\into B'_\G=H_0(\G,B').
\end{equation}
\end{proposition}

\begin{proof}
To obtain explicit formulas for the connecting homomorphisms, we use the connecting homomorphisms defined
by the short exact sequence of complexes
\[
0\to R_\bullet\otimes\BarC'\to R_\bullet\otimes\BarC\to R_\bullet\otimes\BarC''\to 0,
\]
where $\BarC':=\BarC(B')$, $\BarC:=\BarC(B)$, and $\BarC'':=\BarC(B'')$.

Let  $u\otimes x''+\frac{1}{n}\otimes y''$, $n\in\Z$, be an $(m+1)$-cycle
in $R_\bullet\otimes\BarC''$.
Here $x''\in \BarC''_m$ and $y''\in\BarC''_{m+1}$\hs.
 We have
\[ 0=d\big(u\otimes x''+\frac{1}{n}\otimes y''\big)=1\otimes x''-u\otimes dx''+\frac{1}{n}\otimes dy''=
-u\otimes dx''+\frac{1}{n}\otimes(nx''+dy'').
\]
It follows that  $dx''=0\in \BarC''_{m-1}$, that is, $x''$ is an $m$-cycle.

We write $[z]$ for the homology class of a cycle $z$.
We wish to calculate a representative for
$\delta_{m+1}^{\Tor}[u\otimes x''+\frac1n\otimes y'']$,
where, to simplify the notation,
\[\text{we write}\quad\delta_{m+1}^{\Tor}\big[u\otimes x''+\frac1n\otimes y''\big]
\quad\text{for}\quad
\delta_{m+1}^{\Tor}\big(\hs\big[u\otimes x''+\frac1n\otimes y''\big]\hs\big).\]
We lift  $x'',y''\in\BarC''$ to $x,y\in\BarC$,
so that the $(m+1)$-cycle $u\otimes x''+\frac{1}{n}\otimes y''$
in $R_\bullet\otimes\BarC''$ lifts
to an $(m+1)$-chain $u\otimes x+\frac{1}{n}\otimes y$ in $R_\bullet\otimes\BarC$.
We calculate the differential $-u\otimes dx+\frac{1}{n}\otimes(nx+dy)$ of this chain,
and routinely deduce that this element is an image
of some $m$-cycle  $u\otimes x'+\frac{1}{n'}\otimes y'$ in $R_\bullet\otimes\BarC'$.
Then
\[ \Big[u\otimes x'+\frac{1}{n'}\otimes y\Big]
       =\delta_{m+1}^{\Tor}\Big[u\otimes x''+\frac{1}{n}\otimes y''\Big]
       =\Big[-u\otimes dx+\frac{1}{n}\otimes(nx+dy)\Big],\]
whence $x'=-dx$.
The $m$-cycle $x''\in\BarC''$
represents a homology class $[x'']$ in $H_m(\Gamma,B'')$, and the image of $[x'']$
under the connecting homomorphism
\[\delta_m\colon H_{m}(\G, B'')\to H_{m-1}(\G, B')\]
 is  $[dx]$. Since $dx=-x'$,
we see that the diagram in the  proposition  indeed anti-commutes.

Since $H_1(\G,B)$ is killed by $\#\G$
(see \cite[Chapter IV,  \S6, Corollary 1 of Proposition 8]{CasFro86}),
the image of the connecting homomorphism
\[ \delta_1\colon H_1(\G,B'')\to H_{0}(\G,B')\]
consists of torsion elements
(note that the group $H_0(\G,B')=B'_\G$ might not be torsion).
Thus the homomorphism $\psi'_1$ indeed factors as in \eqref{l:psi-phi}.
This completes the proof of Proposition \ref{p:delta}.
\end{proof}

We describe the connecting homomorphism  $\delta_0$ in the sequence \eqref{e:LG} of Theorem \ref{t:Hinich}.
Let $x''\in B''$ be such that the image $(x'')_\G$ of $x''$
in $(B'')_\G$ is contained in  $(B'')_\Gt$\hs.
This means that there exist $n\in\Z_{>0}$ and elements $y''_{\g}\in B''$ for all $\g\in\G$
such that
\[ nx''=\sum_{\g\in\G}\big(\upgam y''_{\g}-y''_{\g}\big).\]
We lift $x''$ to some $x\in B$\hs, we lift each $y''_{\g}$ to some $y_{\g}\in B$\hs,
and we consider the element
\begin{equation*}
z\coloneqq nx-\sum_{\g\in\G}\big(\upgam y_{\g}-y_{\g}\big)\in B.
\end{equation*}
Then $\vk(z)=0\in B''$\hs, whence $z=\iota(z')$ for some  $z'\in B'$.
We consider the image $(z')_\Gtf$ of $z'\in B'$ in $(B')_\Gtf$ where $(B')_\Gtf\coloneqq B_\G/B_\Gt$\hs,  and we put
\begin{equation*}
\delta_0\big(\hs(x'')_\G\big)= \frac1n\otimes (z')_\Gtf \,
    \in\, \Q/\Z\otimes_\Z (B')_\Gtf=\Q/\Z\otimes_\Z (B')_\G
\end{equation*}
where we write $\frac1n$ for the image in $\Q/\Z$ of $\frac1n\in \Q$.

\begin{proposition} \label{p:delta0}
The following diagram commutes:
\begin{equation*}
\xymatrix@C=11mm{
\Tor_1^\Lambda(\Q/\Z, B'')\ar[r]^-{\delta_1^{\Tor}}\ar[d]_{\psi''_1}  &\Q/\Z\otimes_\Lambda B'\ar@{=}[d] \\
(B'')_\Gt\ar[r]^-{\delta_0}                                   &(\Q/\Z\otimes_\Z B')_\G
}
\end{equation*}
In this diagram,
the top horizontal arrow is the connecting map from the exact sequence \eqref{e:LG-Tor},
the map $\psi''_1$ is as $\phi_1$ in \ref{p:Hinich-iso}, and the map $\delta_0$ is defined above.
\end{proposition}

\begin{proof}

Choose $x''\in B''$ representing a class in $(B'')_\Gt$.
Then we have $nx''=\sum _{\gamma\in\Gamma}(\upgam y''_\gamma-y''_\gamma)$
for some integer $n\neq 0$ and some elements $y''_\g\in B''$ with $\g$ running over $\G$.
Consider the $1$-chain
\[
z'':=u\otimes x''-\frac{1}{n}\otimes\sum \gamma^{-1}\otimes
1\otimes_\Lambda y''_\gamma
\]
in $(R_\bullet\otimes \BarC'')_1=
R_1\otimes\BarC''_0\oplus R_0\otimes\BarC''_1$.
We have
\[dz''=1\otimes x''-u\otimes dx''-\frac1n \otimes\sum _{\gamma\in\Gamma}(\upgam y''_\gamma-y''_\gamma)=0\]
so that $z''$ is a $1$-cycle.
By Lemma \ref{lem:viator}, the image in $H_0(\G,B'')=B''_\G$ of the homology class
$[z'']\in\Tor_1^\Lambda(\Q/\Z,B'')$
under $H_1(\pi)$ is the class of $x''$.

To apply $\delta_1^{\Tor}$ to the class of $z''$, we lift $z''$ to
$(R_\bullet\otimes\BarC)_1$, we lift $x'', y''_\gamma\in B''$  to some $x, y_\gamma\in B$
and obtain a 1-chain
\[
z:=u\otimes x-\frac{1}{n}\otimes\sum \gamma^{-1}\otimes 1\otimes_\Lambda y_\gamma
\]
in $R_\bullet\otimes\BarC$.
Its differential $dz=1\otimes x-\frac{1}{n}\otimes\sum(\upgam y_\gamma-y_\gamma)$
is the image of an element $x'$ of $(R_\bullet\otimes\BarC')_0=\Q\otimes B'$,
which represents a class $[x'] \in(\Q/\Z\otimes B')_\G$.
Then $\delta_1^{\Tor}\big((\psi_1'')^{-1}([x''])\big)=[x']$ and
$[x']=\delta_0[x'']$ by the description of $\delta_0$ presented
above. This proves the claim.
\end{proof}

\section{Change of groups}

\begin{proposition}\label{p:LG-LD}
For a finite group $\G$, a  short exact sequence of $\G$-modules
\begin{equation*}
0\to B'\overset \iota\lra B\overset \varkappa\lra B''\to 0
\end{equation*}
as in Theorem \ref{t:Hinich}, and a subgroup $\D$ of $\G$,
we have  morphisms of the homology long exact sequences
 of Theorem \ref{t:Hinich} for $\G$ and for $\D$
\[\Res\colon L_\G\to L_\D\quad\ \text{and}\quad\ \Cor\colon L_\D\to L_\G\hs.\]
\end{proposition}

We specify the restriction and corestriction  morphisms $\Res$ and $\Cor$.

For $i\ge 1$, the homomorphisms
\[ \res_i\colon H_i(\G,B)\to H_i(\D,B)\]
are the usual restriction maps of \S\ref{sub:rescor} and Appendix \ref{app:cheat-sheet}.
The same refers to the homomorphisms  $\res_i'\colon H_i(\G,B')\to H_i(\D,B')$  and  $\res_i''\colon H_i(\G,B'')\to H_i(\D,B'')$.

The homomorphisms
\[\res_0\colon  B_\Gt\to B_\Dt\quad\ \text{and}\quad\ \res_{-1}\colon \Q/\Z\otimes B_\G\to \Q/\Z\otimes B_\D\]
are induced by the transfer map $N_{\D\lmod\G}\colon B_\G\to B_\D$ of \S\ref{sub:explicit0-1}.
The same refers to the homomorphisms $\res_0'$, $\res_{-1}'$, $\res_0''$, $\res_{-1}''$.

For $i\ge 1$, the homomorphisms
\[ \Cor_i\colon H_i(\D,B)\to H_i(\G,B)\]
are the usual corestriction maps of \S\ref{sub:rescor} and Appendix \ref{app:cheat-sheet}.
The same refers to the homomorphisms  $\Cor_i'\colon H_i(\D,B')\to H_i(\G,B')$  and  $\Cor_i''\colon H_i(\D,B'')\to H_i(\G,B'')$.

The homomorphisms
\[\Cor_0\colon  B_\Dt\to B_\Gt\quad\ \text{and}\quad\ \Cor_{-1}\colon \Q/\Z\otimes B_\D\to \Q/\Z\otimes B_\G\]
are induced by the natural projection $\proj\colon B_\D\to B_\G$\hs.
The same refers to the homomorphisms $\Cor_0'$, $\Cor_{-1}'$, $\Cor_0''$, $\Cor_{-1}''$.

\begin{proof}[Proof of Proposition \ref{p:LG-LD}]
It suffices to check the commutativity of the diagrams
\begin{equation}\label{e:Change-i}
\begin{aligned}
\xymatrix@C=10mm{
H_i(\G,B'')\ar[r]^-{\delta_i}\ar[d]_-{\Res_i} & H_{i-1}(\G,B')\ar[d]^-{\Res_{i-1}}\\
H_i(\D,B'')\ar[r]^-{\delta_i}                 & H_{i-1}(\D,B')
}
\qquad\quad
\xymatrix@C=10mm{
H_i(\D,B'')\ar[r]^-{\delta_i}\ar[d]_{\Cor_i} & H_{i-1}(\D,B')\ar[d]^{\Cor_{i-1}}\\
H_i(\G,B'')\ar[r]^-{\delta_i}                & H_{i-1}(\G,B')
}
\end{aligned}
\end{equation}
for $i\ge 1$, and the commutativity of the diagrams
\begin{equation}\label{e:Change-0}
\begin{aligned}
\xymatrix@C=10mm{
(B'')_\Gt\ar[r]^-{\delta_0}\ar[d]_{N_{\D\lmod\G}} & (\Q/\Z\otimes B')_\G\ar[d]^{N_{\D\lmod\G}}\\
(B'')_\Dt\ar[r]^-{\delta_0}                 & (\Q/\Z\otimes B')_\D
}
\qquad\
\xymatrix@C=10mm{
(B'')_\Dt\ar[r]^-{\delta_0}\ar[d]_{\proj} & (\Q/\Z\otimes B')_\D\ar[d]^{\proj}\\
(B'')_\Gt\ar[r]^-{\delta_0}                & (\Q/\Z\otimes B')_\G
}
\end{aligned}
\end{equation}

The commutativity of the diagrams \eqref{e:Change-i} is well known.
The commutativity of the diagram at right in \eqref{e:Change-0} is obvious.
The commutativity of the diagram at left in \eqref{e:Change-0}
can be proved by a standard calculation using \cite[Proof of Lemma 9.91(ii)]{RotmanHA}.
\end{proof}

\chapter{$H^1$ for a crossed module}
\label{app:crossed}

In this appendix, following a request of an anonymous referee,
we recall the definition of the first group cohomology
with coefficients in an equivariant crossed module.
For details see \cite[\S3.3]{Brv98}.

A (left) {\em  crossed module} is a group homomorphism   $\alpha\colon H\to G$ together  with a
left action $G\times H\to H$ of $G$ on $H$ denoted $(g,h)\mapsto\hs ^g h$ such that
\begin{align*}
&h\hs h' h^{-1}=\hs^{\alpha(h)}\hm h',\\
&\alpha(\hs^g\hm h)=g\cdot \alpha(h)\cdot g^{-1}.
\end{align*}
for all $h,h'\in H$, $g\in G$.
We regard a crossed module $H\to G$ as a complex
(of not necessarily abelian groups)  with $H$ in degree $-1$ and $G$ in degree 0.
We say that a profinite group $\G$  acts on a crossed module $H\to G$ if $\G$  acts on $H$ and $G$ so that
\[\alpha(\hs^\sigma\hm h)=\hs^\sigma\hm (\alpha(h)),
    \quad\ ^\sigma\hm (\hs^g\hm  h)=\hs^{^\sigma\!\hm  g}(\hs^\sigma\hmm h)\]
for all $h\in H$, $g\in G$, $\sigma\in\G$.
Then we say that $H\to G$ is a {\em $\G$-equivariant crossed module}.

Let $H\to G$ be a $\G$-equivariant crossed module.
We define the set of 0-(hyper)\-cochains
\[ C^0=C^0(\G, H\to G) =\Maps(\G,H)\times G\]
where $\Maps(\G,H)$ denotes the set of continuous (that is, locally constant) maps
from the profinite group $\G$ to the group $H$.
We define a group structure on $C^0$ as follows.
Let $(\vphi_1,g_1), (\vphi_2,g_2)\in C^0$.
We set
\[(\vphi_1,g_1)\cdot (\vphi_2,g_2)=(\vphi', g_1g_2)\quad\ \text{where}\ \, \vphi'(\sigma)=\hs^{g_1}\hm\vphi_2(\sigma)\cdot\vphi_1(\sigma).\]

Let  $H\to G$ be a $\G$-equivariant crossed module. We define the 1-(hyper)\-co\-homology set
$H^1(\G, H\to G)$ in terms of cocycles.

 We define the set of 1-(hyper)cocycles $Z^1=Z^1(\G,H\to G)$ to be the set of pairs
 $(u,\psi)$,  $u\in \Maps(\G\times \G, H)$ , $\psi\in \Maps(\G,G)$, such that
\begin{align*}
&\psi(\sigma\tau)=\alpha\big(u(\sigma,\tau)\big)\cdot\psi(\sigma)\cdot\hs^\sigma\hm\psi(\tau),\\
&u(\sigma,\tau\ups)\cdot\hs^{\psi(\sigma)\hs\sigma}\hm u(\tau,\ups)=u(\sigma\tau,\ups)\cdot u(\sigma,\tau),
\end{align*}
for all $\sigma,\tau,\ups\in\G$.

We define a right action $*$ of the group $C^0$ on the set of 1-cocycles $Z^1$ as follows.
For   $(u,\psi)\in Z^1$ and $(a,g)\in C^0$,
we set
\[ (u,\psi)*(a,g) = (u',\psi')\]
where
\begin{align*}
& \psi'(\sigma)=g^{-1}\cdot\alpha(a(\sigma))\cdot\psi(\sigma)\cdot \hs^\sigma\! g,\\
& u'(\sigma,\tau)=\hs^{g^{-1}}\!\big(a(\sigma\tau)\cdot u(\sigma,\tau)\cdot\hs^{\psi(\sigma)\hs\sigma}\! a(\tau)^{-1}\cdot a(\sigma)^{-1}\big)
\end{align*}
One can check that $(u',\psi')\in Z^1$.
Now we define the first cohomology set
\[H^1(\G,H\to G)=Z^1/C_0\hs,\]
to be the set of orbit of the group $C^0$ in the set $Z^1$.

\chapter{Duality}
\label{app:duality}

The lemmas in this appendix are well-known. Since we do not have references, we provide short proofs.

\section{Duality and group action}

\begin{lemma}
\label{l:A^*}
Let $\G$ be a group and $A$ be a finite abelian group with $\G$-action.
Write $A^*=\Hom(A,\Q/\Z)$. Then
\begin{enumerate}
\item[\rm(1)] $(A^*)^\G\cong (A_\G)^*$,
\item[\rm(2)] $(A^\G)^*\cong (A^*)_\G$\hs.
\end{enumerate}
\end{lemma}

\begin{proof}
The object $A^\Gamma$ and the homomorphism $A^\Gamma \to A$ are characterized by the universal property
that any $\Gamma$-homomorphism $X \to A$ from a trivial $\Gamma$-module $X$ factors uniquely through $A^\Gamma$.
The dual property characterizes the homomorphism $A \to A_\Gamma$.
Since $A \mapsto A^*$ is a contravariant self-equivalence on the category of finite abelian groups
(as follows from the classification theorem of such groups, or from Pontryagin duality), the claim follows.
\end{proof}

\section{Duality via tori and lattices}

\begin{lemma}
\label{l:M-via-K}
Let $F$ be a field, and let $T^{-1}\to T^0$ be
a {\emm surjective} homomorphism of $F$-tori
with finite kernel $\KK$ (not necessarily smooth).
Let $X^*(\KK)\coloneqq \Hom(\KK_\Fbar\hs, \GG_{{\rm m},\Fbar})$ denote the character group of $\KK$.
Set $M=\coker\big[X_*(T^{-1})\to X_*(T^0)\big]$.
Then there is a functorial $\G(F^s\hm/F)$-equivariant isomorphism
of finite abelian groups
\[ M\cong \Hom(X^*(\KK), \Q/\Z).\]
\end{lemma}

\begin{proof}
The short exact sequence
\[ 1\to \KK\to T^{-1}\to T^0\to 1\]
induces a short exact sequence of character groups
\[0\to X^*(T^0)\to X^*(T^{-1})\to X^*(\KK)\to 0.\]
Applying the functor $\Hom(\,\cdot\,,\Z)$ to this short exact sequence,
we obtain an exact sequence
\begin{multline}\label{e:Ext1}
\Hom(X^*(T^{-1}),\Z) \to \Hom(X^*(T^0),\Z)\\
       \to \Ext^1_\Z(X^*(\KK),\Z) \to\Ext^1_\Z(X^*(T^{-1}),\Z)
\end{multline}
where $\Ext^1_\Z$ denotes the functor $\Ext^1$ on the category of $\Z$-modules (abelian groups).
For a finitely generated abelian group $A$, there is a canonical isomorphism
\begin{equation}\label{e:Ext-Hom}
\Ext^1_\Z(A,\Z)\cong \Hom(A_\Tors,\Q/\Z)
\end{equation}
(see, for instance, \cite[Lemma 2.1]{Borovoi23-JLT}\hs).
Thus from \eqref{e:Ext1} we obtain an exact sequence
\[X_*(T^{-1}) \to X_*(T^0) \to \Hom(X^*(\KK),\Q/\Z) \to 0,\]
whence we obtain a canonical isomorphism
\[ M=\coker\big[X_*(T^{-1}) \to X_*(T^0)\big]\isoto \Hom(X^*(\KK),\Q/\Z). \qedhere\]
\end{proof}

\begin{corollary}\label{c:ss-pi1}
Let $G$ be a {\em semisimple} group over a field $F$, and write $\KK=\ker[\rho\colon G_\ssc\to G]$.
Then there is a functorial $\G(F^s\hm/F)$-equivariant isomorphism
of finite abelian groups
\[ \pi_1(G)\cong\Hom\big(X^*(\KK), \Q/\Z\big).\]
\end{corollary}

\begin{proof}
With the notation of \S\ref{ss:fund-gp}
we have $\KK=\ker[T_\ssc \to T]$ and $\pi_1(G)=\coker[X_*(T_\ssc)\to X_*(T)]$.
Now the corollary follows from the lemma.
\end{proof}

\chapter{Fiber product}
\label{app:fiber}

Consider a diagram of sets and maps
\begin{equation}\label{e:fiber-diag}
\xymatrix@1{ S_1\ar[r]^-{f_1}  & B &S_2\ar[l]_{f_2} }
\end{equation}
and its fiber product
\[ S_1\underset{B,f_1,f_2}\times S_2=\{(s_1,s_2)\in S_1\times S_2\ |\  f_1(s_1)=f_2(s_2)\}. \]
Let $U$ be a set, and consider the map
\[f_1'\colon U\times S_1\to B,\quad (u,s_1)\mapsto f_1(s_1).\]

\begin{lem}[well-known]
\label{l:fiber-factor}
The canonical bijection
\[ U\times (S_1\times S_2)\isoto (U\times S_1)\times S_2\hs,\quad \big(u,(s_1,s_2)\big)\mapsto \big((u,s_1),s_2\big)\]
induces a bijection
\begin{equation}\label{e:product-fiber}
U\times (S_1\underset{B,f_1,f_2}\times S_2)\isoto  (U\times S_1)\underset{B,f_1',f_2}\times S_2\hs.
\end{equation}
\end{lem}

\begin{proof}
Each of both sides of \eqref{e:product-fiber} equals \
\begin{multline*}
\big\{(u,s_1,s_2)\in U\times S_1\times S_2\ |\ f_1(s_1)=f_2(s_2)\big\}\\
\subseteq\hs U\times S_1\times S_2 = U\times (S_1\times S_2) = (U\times S_1)\times S_2.
\end{multline*}
\end{proof}

\begin{rmk}\label{r:fiber-groups}
If $U$ is a group and \eqref{e:fiber-diag} is a diagram of groups and homomorphisms,
then \eqref{e:product-fiber} is a group isomorphism.
\end{rmk}

\chapter{Cheat sheet on Tate cohomology}
\label{app:cheat-sheet}

Recall from \eqref{eq:biinf} in \S\ref{ss:St-res} the standard doubly-infinite resolution
\[
\xymatrix@C=5.8mm@R=2.5mm{
\dots\ar[r] &P_3\ar[r]^{\partial_3}&P_2 \ar[r]^{\partial_2} &P_1 \ar[r]^{\partial_1} &P_0 \ar[r]^{\partial_0} &P_{-1}\ar[r]^{\partial_{-1}} &P_{-2}\ar[r]^{\partial_{-2}} & P_{-3}\ar[r] &\dots\\
     &\Z[\Gamma^4]\ar@{=}[u]  &\Z[\Gamma^3]\ar@{=}[u] &\Z[\Gamma^2]\ar@{=}[u] &\Z[\Gamma]\ar@{=}[u] &\Z[\Gamma]\ar@{=}[u] &\Z[\Gamma^2]\ar@{=}[u] &\Z[\Gamma^3]\ar@{=}[u]
}
\]
For a $\Gamma$-module $A$ we obtain the homogeneous cochain complex for $A$
\[
\xymatrix@C=3.6mm@R=2.5mm{
\dots C^{-2}(\Gamma,A)\ar[r]^-{d^{-1}}  &C^{-1}(\Gamma,A)\ar[r]^-{d^{0}}  &C^{0}(\Gamma,A)\ar[r]^-{d^{1}}  &C^{1}(\Gamma,A)\ar[r]^-{d^2}&C^{2}(\Gamma,A)\dots \\
 \Maps_\Gamma(\Gamma^2,A)\ar@{=}[u] &\Maps_\Gamma(\Gamma,A)\ar@{=}[u] &\Maps_\Gamma(\Gamma,A)\ar@{=}[u] &\Maps_\Gamma(\Gamma^2,A)\ar@{=}[u]&\Maps_\Gamma(\Gamma^3,A)\ar@{=}[u]
}
\]
with differential $d^n f = f \circ \partial_n$, explicitly given for $n > 0$ by
\[ (d^n f)(g_0,\dots,g_n) = \sum_{i=0}^n (-1)^i f(g_0,\dots,\hat g_i,\dots,g_n), \]
for $n=0$ by
\[ (d^0 f)(g)=\sum_{h \in \Gamma}f(h), \]
and for $n<0$ by
\[ (d^n f)(g_0,\dots,g_{-n-1}) = \sum_{i=0}^{-n} (-1)^i \sum_{h \in \Gamma} f(g_0,\dots,g_{i-1},h,g_i,\dots,g_{-n-1}). \]
In low degrees these formulas read
\begin{align*}
(d^{-2}f)(a,b)&=\sum_{h\in\G}f(h,a,b)-f(a,h,b)+f(a,b,h)\\
(d^{-1}f)(a)&=\sum_{h\in\G} f(h,a) - f(a,h)\\
(d^0f)(a)&=\sum_{h\in\G} f(h)\\
(d^1f)(a,b)&=f(b)-f(a)\\
(d^2f)(a,b,c)&=f(b,c)-f(a,c)+f(a,b)\\
(d^3f)(a,b,c,d)&=f(b,c,d)-f(a,c,d)+f(a,b,d)-f(a,b,c).
\end{align*}

We can also use the inhomogeneous complex
\[
\xymatrix@C=6.8mm@R=2.5mm{
\dots C^{-2}(\Gamma,A)\ar[r]^-{d^{-1}}  &C^{-1}(\Gamma,A)\ar[r]^-{d^0}  &C^{0}(\Gamma,A)\ar[r]^-{d^1}  &C^{1}(\Gamma,A)\ar[r]^-{d^2}  &C^{2}(\Gamma,A)\dots\\
\Maps(\Gamma,A)\ar@{=}[u] &A\ar@{=}[u] &A\ar@{=}[u] &\Maps(\Gamma,A)\ar@{=}[u] &\Maps(\Gamma^2,A)\ar@{=}[u].
}
\]
For a homogeneous cochain $f\in C^n(\Gamma,A)$ with corresponding inhomogeneous cochain $c\in C^n(\Gamma,A)$, we have
\begin{align*}
c(g_1,\dots,g_n)&=f(1,g_1,g_1g_2,\dots,g_1\cdots g_n),\quad n\ge 0\\
c(g_1,\dots,g_{-n-1})&=f(1,g_1,g_1g_2,\dots,g_1\cdots g_{-n-1}),\quad n<0\\
f(g_0,\dots,g_n)&=g_0c(g_0^{-1}g_1,\hs g_1^{-1}g_2,\hs g_2^{-1}g_3,\hs\dots,\hs g_{n-1}^{-1}g_n)\quad\ n\ge0\\
f(g_0,\dots,g_{-n-1})&=g_0c(g_0^{-1}g_1,\hs g_1^{-1}g_2,\hs\dots,\hs g_{-n-2}^{-1}g_{-n-1})\quad\ n<0.
\end{align*}
In the inhomogeneous complex, the differentials become
\begin{eqnarray*}
(d^n c)(g_1,\dots,g_n)&=&g_1c(g_2,\dots,g_n)\\
&-&c(g_1\!\cdot\! g_2,g_3,\dots,g_n)\\
&+&c(g_1,g_2\!\cdot\! g_3,g_4,\dots,g_n)\\
&\dots&\!\!\!\!\dots\dots\dots\dots\dots\dots\dots\dots\\
&+&(-1)^{n-1}c(g_1,g_2,g_3,\dots,g_{n-1}\!\cdot\! g_n)\\
&+&(-1)^nc(g_1,g_2,g_3,\dots,g_{n-1})\\
\end{eqnarray*}
for $n>0$,
\[ d^0c = \sum_{h \in \Gamma} hc \]
for $n=0$, \,and
\begin{eqnarray*}
(d^n c)(g_1,\dots,g_{-n-1})=\sum_{h \in \Gamma}\Big(&&hc(h^{-1},g_1,\dots,g_{-n-1})\\
&-&c(h,h^{-1}\!\cdot\! g_1,g_2,\dots,g_{-n-1})\\
&+&c(g_1,h,h^{-1}\!\cdot\! g_2,g_3,\dots,g_{-n-1})\\
&.&\!\!\!\!\!\!\dots\dots\dots\dots\dots\dots\dots\dots\dots\dots\dots\\
&+&(-1)^{-n-1}c(g_1,g_2,\dots,g_{-n-2},h,h^{-1}\!\cdot\! g_{-n-1})\\
&+&(-1)^{-n}c(g_1,g_2,\dots,g_{-n-1},h)\Big)
\end{eqnarray*}
for $n<0$. See also \cite[Chapter VII,  \S3 and \S4]{SerLF}.

In low degrees these formulas read
\begin{align*}
&(d^{-2}c)(g) =\sum_{h \in \Gamma}hc(h^{-1},g)-c(h,h^{-1}g)+c(g,h)\\
&d^{-1}c=\sum_{h \in \Gamma} hc(h^{-1})-c(h)\\
&d^0c=\sum_{h \in \Gamma} hc\\
&(d^1c)(g)=gc-c\\
&(d^2c)(g,h)=gc(h)-c(gh)+c(g).
\end{align*}

In \S\ref{sub:rescor} we gave explicit formulas for restriction and corestriction in terms of homogeneous cochains.
Here we write the formulas in terms of inhomogeneous cochains.
Let $\Delta \subset \G$ be a subgroup and let $s\colon \Delta \lmod \G \to \G$ be a section of the natural projection.
The corresponding formulas in terms of inhomogeneous cochains are given as follows. For $c \in C^n(\Gamma,A)$,
the value of $\tx{res}^n(c)(h_1,\dots,h_n)$ for $n \geq 0$ is
\[ c(h_1,\dots,h_n)\]
and the value of $\tx{res}^n(c)(h_1,\dots,h_{-n-1})$ for $n < 0$ is
\[ \sum_{x \in (\Delta \lmod \Gamma)^{-n}}\hskip-4mm s(x_0)c\Big(s(x_0)^{-1}h_1s(x_1),\,s(x_1)^{-1}h_2s(x_2),\hs\dots,\hs s(x_{-n-2})^{-1}h_{-n-1}s(x_{-n-1})\Big).\]
For $c \in C^n(\Delta,A)$, the value of $\tx{cor}^n(c)(g_1,\dots,g_n)$ for $n \geq 0$ is
\[ \sum_{x \in \Delta \lmod \Gamma\hskip-2mm}\!\! s(x)^{-1}c\Big(\hmm s(x)g_1s(xg_1)^{-1}\!,s(xg_1)g_2s(xg_1g_2)^{-1}\!,\dots,s(xg_1\cdots g_{n-1})g_n s(xg_1 \cdots g_n)^{-1}\hmm\Big)\hm, \]
while the value of $\tx{cor}^n(c)(g_1,\dots,g_{-n-1})$ for $n<0$ is
\[
\begin{cases}
	c(g_1,\dots,g_{-n-1})&\tx{if}\quad g_1,\dots,g_{-n-1} \in \Delta,\\
	0                            &\tx{else}.
\end{cases}	
\]

Let $\G$ be the  finite  cyclic group  $\Gamma=\Z/N\Z$.
For a $\Gamma$-module $M$
 we compute the periodicity isomorphism
 $$H^{0}(\Gamma,M) \to H^{-2}(\Gamma,M)$$
 in terms of inhomogeneous cocycles.
 We represent the elements of $\G$ as natural numbers $0 \leq a < N$. 	
 Recall that $H^2(\Gamma,\Z)$ is cyclic of order $N$ with generator represented
 by the inhomogeneous $2$-cocycle
\[ z(a,b) = \begin{cases} 0,&a+b < N\\ 1,&a+b \geq N; \end{cases} \]
see \cite[p.~102, above Proposition III.1.9]{MilneCFT}.
Recall further that cup product with the class of $z$
induces an isomorphism $H^n(\Gamma,M) \to H^{n+2}(\Gamma,M)$ for any $n \in \Z$;
see \cite[Chapter IV, \S8, Theorem 5]{CasFro86}.

\begin{lem}
For any $m \in M^\Gamma$, the function
\[ c_m(a) = \begin{cases} m,& a=N-1\\ 0,& a \neq N-1\end{cases} \]
is an  inhomogeneous $(-2)$-cocycle, and the map
\[[m] \mapsto [c_m]\colon\, H^0(\G,M)\to H^{-2}(\G,M)\]
 is the inverse of the isomorphism $H^{-2}(\Gamma,M) \to H^0(\Gamma,M)$
  given by cup product with $[z]$.
\end{lem}

\begin{proof}
The fact that $c_m$ is a cocycle is verified immediately using the formula for $d^{-1}$ above.
Using the formula for the cup product in degrees $(2) \times (2)$ from \cite[Chapter IV, \S7]{CasFro86}
and translating it to inhomogeneous cocycles using the formulas above, we see that
for any inhomogeneous $(-2)$-cocycle $c : \Gamma \to M$ we have
\[ [z] \cup [c] = \left[\sum_{a,b \in \Z/N\Z} z(a,b) \cdot {^{(a+b)}}c(-b)\right]. \]
Applying this to the particular $z$ and $c$ above, we see that the only non-zero summand
in the sum is the one for $b=1$ and $a=N-1$, whose value is $^{(a+b)}m=m$.
\end{proof}

\chapter{A computer computation}
\label{app:computer}

Following a suggestion of an anonymous referee, in the arXiv version
we give  the listing of the GAP program {\sf Computing-Sha.g}
and the corresponding output file.
We used them in the calculation of $\Sha^1(F,G)$ in Section \ref{ss:computer}.
\bigskip

\centerline{\bf The GAP program  \sf Computing-Sha.g}

{\small

\begin{lstlisting}

# Computing Sha in GAP

PPrint:=function(str,mat)  #Printing a matrix
	local j;
	
	Print("\n",str,"\n");
	
	for j in [1..Length(mat)] do
		Print( mat[j], ",\n" );
	od;
	
	Print("\n");
	
end;

#############################################

actS:=function(g,s)
# A certain action of g=3,5,7  on  a vector s of length 6
# We consider the group G=(Z/8Z)^*=[1,3,5,7] and its action on the set S_E=[1..6]
    #with three orbits [1,2], [3,4], [5,6].
# When acting on [1,2] the stabilizer is [1,3],
# when acting on [3,4] the stabilizer is [1,5],
# when acting on [5,6] the stabilizer is [1,7].
# We consider the action of g in G on s in M[S_E] where M=Z
    #(later it will be Z/8Z).
local sr;
	sr:=[0,0, 0,0, 0,0]; # Constructing a vector of length 6

	if g=3 then
		sr[1]:=g*s[1]; sr[2]:=g*s[2];
		sr[3]:=g*s[4]; sr[4]:=g*s[3];
		sr[5]:=g*s[6]; sr[6]:=g*s[5];
	fi;		

	if g=5 then
		sr[1]:=g*s[2]; sr[2]:=g*s[1];
		sr[3]:=g*s[3]; sr[4]:=g*s[4];
		sr[5]:=g*s[6]; sr[6]:=g*s[5];
	fi;

	if g=7 then
		sr[1]:=g*s[2]; sr[2]:=g*s[1];
		sr[3]:=g*s[4]; sr[4]:=g*s[3];
		sr[5]:=g*s[5]; sr[6]:=g*s[6];
	fi;		

return sr;
end;

####################################################
####################################################

#Write to output file:
LogTo("Out-Computing-Sha.txt");

Print("The output file for Computing-Sha\n\n");

#We consider M[S_E]=\{ \sum_{s \in S_E} m_s | s \in S_E, m_s \in M\}
# and M[S_E]_0= \{ \sum_{s \in S_E} m_s \cdot s \in M[S_E]| \sum_s m_s =0 \}


############################ Computing a basis b of M[S_E]_0
# and a set d of generators of coboundaries in M[S_E]_0

b:=[];
d:=[];

for i in [1..5] do

	sd:=[0,0, 0,0, 0,0]; # Constructing a null vector of length 6
	sd[i]:=1;
	sd[6]:=-1; # a basis vector of M[S_E]_0
	
	Add( b, sd );  # b is a basis of M[S_E]_0 of cardinality 5
	
	for g in [3,5,7] do
		Add( d, actS(g,sd)-sd ) ;  # Coboundaries g*sd - sd:
		                           # 15 vectors of length 6
	od;
		
od;

PPrint("b", b);

PPrint("d", d);

############# Computing a set dD  of generators of coboundaries in M[S_E]

dD:=[];

for i in [1..6] do

	sd:=[0,0, 0,0, 0,0]; # Constructing a  null vector of length 6
	sd[i]:=1; # a basis vector of M[S_E]
	
	for g in [3,5,7] do
		Add( dD, actS(g,sd)-sd) ;  # Coboundaries g*sd - sd:
		                           # 18 vectors of length 6
	od;
	
od;

PPrint("dD",dD);

############################# Computing the Tate-Shafarevich kernel Sha

Print("\n");

d8b:=Concatenation( d, 8*b ); #We calculate modulo 8
K_b_d8b:=ComplementIntMat(b,d8b); # See the description of this function on the Net:
                                  # Google gap ComplementIntMat
Print("Complement b,d8b moduli:  ", K_b_d8b.moduli, "\n");
PPrint("Complement b,d8b sub:\n", K_b_d8b.sub );

#################################


K2:=[K_b_d8b.sub[4]/2,K_b_d8b.sub[5]/8];
PPrint("K2: generators of the quotient group <b>/<d,8*b>\n", K2);

KD8:=Concatenation( K2, dD, 8*IdentityMat(6));
KSha0:=NullspaceIntMat(KD8);

KSha:=[];
for ss in KSha0 do
	sh:=[0,0]; # null vector of length 2
	for j in [1..2] do	
		sh[j]:=ss[j];
	od;
	Add(KSha,sh);
od; # We cut the first two columns

KShaEZ:=Filtered(KSha,x->x<>0*x);  # We erase the zero rows
PPrint( "KShaEZ:\n", KShaEZ );	

K2ker:=KShaEZ*K2;
Print( "K2ker: the kernel of the homomorphism \n          ");
Print(" (M[S_E]_0)_\Gamma -> (M[S_E])_\Gamma");
PPrint(" ",  K2ker );
Print("Comparing with K2, we see that Sha^1(F,G) is of order 2\n");
Print("with generator K2ker[2]=", K2ker[2],".\n\n");

####################################
Print("#############################\n\n");

Print("Looking for a nice representative of the nontrivial element of Sha. \n\n");

dD8:=Concatenation( 8*IdentityMat(6), dD );
d8b:=Concatenation( 8*b, d );

v:=[4, 0, -4, -0, 0, 0,];
Print("We check the vector v:=[4,0,-4,0,0,0] in M[S_E]_0, \n");
Print("which is equivalent to K2ker[2]:\n");
Print("if w=[4,0,0,0,0,-4], then v equiv K2ker[2]+5*w-w mod 8.\n");
Res:=SolutionIntMat(d8b,v);
Print("\nIn M[S_E]_0: v=", v,"\nResult: ", Res,"\n");
Print("Since the result is `fail', our vector v is not in <d,8b>, \n");
Print("and hence its class [v] is not 0 in (M[S_E]_0)_Gamma. \n");

ResD:=SolutionIntMat(dD8,v);
Print("\nIn M[S_E]: v=", v,"\nResult:\n", ResD,"\n");
Print("Since the result is NOT`fail', our vector v is in <dD,8*id>, \n");
Print("and hence its class [v] is in Sha. \n");

Print("\nThe results show that, as expected, our vector v is a representative \n");
Print("of the only nontrivial element of \n");
Print("        Sha^1(F,G) = ker[ (M[S_E]_0)_\Gamma --> (M[S_E])_Gamma ].\n");

LogTo();

\end{lstlisting}
}


\bigskip\bigskip

\centerline{\bf The output file for the GAP program {\sf  Computing-Sha.g}}

{\small

\begin{lstlisting}
b
[ 1, 0, 0, 0, 0, -1 ],
[ 0, 1, 0, 0, 0, -1 ],
[ 0, 0, 1, 0, 0, -1 ],
[ 0, 0, 0, 1, 0, -1 ],
[ 0, 0, 0, 0, 1, -1 ],


d
[ 2, 0, 0, 0, -3, 1 ],
[ -1, 5, 0, 0, -5, 1 ],
[ -1, 7, 0, 0, 0, -6 ],
[ 0, 2, 0, 0, -3, 1 ],
[ 5, -1, 0, 0, -5, 1 ],
[ 7, -1, 0, 0, 0, -6 ],
[ 0, 0, -1, 3, -3, 1 ],
[ 0, 0, 4, 0, -5, 1 ],
[ 0, 0, -1, 7, 0, -6 ],
[ 0, 0, 3, -1, -3, 1 ],
[ 0, 0, 0, 4, -5, 1 ],
[ 0, 0, 7, -1, 0, -6 ],
[ 0, 0, 0, 0, -4, 4 ],
[ 0, 0, 0, 0, -6, 6 ],
[ 0, 0, 0, 0, 6, -6 ],


dD
[ 2, 0, 0, 0, 0, 0 ],
[ -1, 5, 0, 0, 0, 0 ],
[ -1, 7, 0, 0, 0, 0 ],
[ 0, 2, 0, 0, 0, 0 ],
[ 5, -1, 0, 0, 0, 0 ],
[ 7, -1, 0, 0, 0, 0 ],
[ 0, 0, -1, 3, 0, 0 ],
[ 0, 0, 4, 0, 0, 0 ],
[ 0, 0, -1, 7, 0, 0 ],
[ 0, 0, 3, -1, 0, 0 ],
[ 0, 0, 0, 4, 0, 0 ],
[ 0, 0, 7, -1, 0, 0 ],
[ 0, 0, 0, 0, -1, 3 ],
[ 0, 0, 0, 0, -1, 5 ],
[ 0, 0, 0, 0, 6, 0 ],
[ 0, 0, 0, 0, 3, -1 ],
[ 0, 0, 0, 0, 5, -1 ],
[ 0, 0, 0, 0, 0, 6 ],


Complement b,d8b moduli:  [ 1, 1, 1, 2, 8 ]

Complement b,d8b sub:

[ 1, 7, -7, 1, 7, -9 ],
[ 0, 4, -3, 1, 3, -5 ],
[ 0, 4, -4, 0, 5, -5 ],
[ 0, 2, -2, 2, 2, -4 ],
[ 0, 8, -8, 0, 8, -8 ],


K2: generators of the quotient group <b>/<d,8*b>

[ 0, 1, -1, 1, 1, -2 ],
[ 0, 1, -1, 0, 1, -1 ],


KShaEZ:

[ 2, 0 ],
[ 0, 4 ],

K2ker: the kernel of the homomorphism
           (M[S_E]_0)_Gamma -> (M[S_E])_Gamma

[ 0, 2, -2, 2, 2, -4 ],
[ 0, 4, -4, 0, 4, -4 ],

Comparing with K2, we see that Sha^1(F,G) is of order 2
with generator K2ker[2]=[ 0, 4, -4, 0, 4, -4 ].

#############################

Looking for a nice representative of the nontrivial element of Sha.

We check the vector v:=[4,0,-4,0,0,0] in M[S_E]_0,
which is equivalent to K2ker[2]:
if w=[4,0,0,0,0,-4], then v equiv K2ker[2]+5*w-w mod 8.

In M[S_E]_0: v=[ 4, 0, -4, 0, 0, 0 ]
Result: fail
Since the result is `fail', our vector v is not in <d,8b>,
and hence its class [v] is not 0 in (M[S_E]_0)_Gamma.

In M[S_E]: v=[ 4, 0, -4, 0, 0, 0 ]
Result:
[ 101454990022566097934, -4439323767294845870, 11542241794966599262,
  7694827863311066175, -28214368832140575975, 188095792214270506500,
  -402268501076428515038, 7102918027671753392, 0, 0, 0, 0,
  -20519540968829509800, -28214368832140575975, 0, 0, 0, 0,
  188095792214270506500, -413810742871395114300, 0, 0, 0, 0 ]
Since the result is NOT`fail', our vector v is in <dD,8*id>,
and hence its class [v] is in Sha.

The results show that, as expected, our vector v is a representative
of the only nontrivial element of
        Sha^1(F,G) = ker[ (M[S_E]_0)_Gamma --> (M[S_E])_Gamma ].


\end{lstlisting}
}



\providecommand{\bysame}{\leavevmode\hbox to3em{\hrulefill}\thinspace}
\providecommand{\MR}{\relax\ifhmode\unskip\space\fi MR }
\providecommand{\MRhref}[2]{%
  \href{http://www.ams.org/mathscinet-getitem?mr=#1}{#2}
}
\providecommand{\href}[2]{#2}

\end{document}